\documentclass{amsart}

\newif\ifcomments
\commentstrue  

\usepackage{amsmath,amsthm,amssymb,amscd,url,enumerate}
\usepackage[utf8]{inputenc}
\usepackage{hyperref}
\usepackage{pdflscape}
\usepackage{caption}
\usepackage[noadjust]{cite}
\usepackage{xypic}
\usepackage{tikz}
\usetikzlibrary{decorations.markings,arrows}
\usepackage{tikz-cd}
\usepackage{amsfonts}
\usepackage{amsmath}
\usepackage{graphicx}
\usepackage{amsthm}
\usepackage{xcolor}
\usepackage{mathrsfs}
\usepackage[margin=1in]{geometry}
\usepackage[ruled,vlined,linesnumbered,algosection]{algorithm2e}
\usepackage{algpseudocode}
\usepackage[inline]{enumitem}

\SetEndCharOfAlgoLine{}
\DecMargin{2.2mm}
\SetKwInput{Input}{Input}
\SetKwInput{Output}{Output}
\SetKwFor{While}{While}{do}{}
\SetKwFor{ForEach}{For each}{do}{}
\SetKwIF{If}{ElseIf}{Else}{If}{then}{else if}{else}{endif}
\SetKw{Return}{Return}
\SetKw{KwRet}{return}
\SetKw{Recurse}{Recurse}
\SetKwFor{For}{For}{do}{}
\SetKw{Break}{break}

\newcommand{\amod}{\!\! \pmod}
\newcommand{\poly}{\operatorname{poly}}

\theoremstyle{plain} 
\newtheorem{theorem}{Theorem} 

\newtheorem{proposition}[theorem]{Proposition}
\newtheorem{lemma}[theorem]{Lemma}

\newtheorem{heuristic}[theorem]{Heuristic}

\theoremstyle{definition}
\newtheorem{definition}[theorem]{Definition}

\theoremstyle{remark}
\newtheorem{remark}[theorem]{Remark}

\newtheorem{example}[theorem]{Example}
\newtheorem{problem}[theorem]{Problem}

\numberwithin{theorem}{section}

\newcommand{\Gal}{\operatorname{Gal}}

\newcommand{\ZZ}{\mathbb{Z}}
\newcommand{\QQ}{\mathbb{Q}}
\newcommand{\FF}{\mathbb{F}}
\newcommand{\Fp}{\mathbb{F}_p}
\newcommand{\Fpp}{\mathbb{F}_{p^2}}
\newcommand{\Fpbar}{{\overline{\mathbb{F}}_p}}
\newcommand{\Fpi}{{i}}
\newcommand{\Fpa}{{\mathbf{a}}}
\newcommand{\End}{\operatorname{End}}

\newcommand{\Einit}{E_{\operatorname{init}}}

\newcommand{\Eone}{E}

\newcommand{\SSOpr}{\operatorname{SS}_\mathcal{O}^{pr}}
\newcommand{\SSOnotprim}{\operatorname{SS}_\mathcal{O}}
\newcommand{\ClO}{\Cl(\mathcal{O})}

\newcommand{\hO}{h_\mathcal{O}}
\newcommand{\mulM}{\mathbf{M}}

\newcommand{\lcm}{\operatorname{lcm}}

\DeclareMathOperator{\Cl}{Cl}

\let\SS\relax   \DeclareMathOperator{\SS}{SS}
\definecolor{Bittersweet}{rgb}{1.0, 0.44, 0.37}
\definecolor{KateColour}{rgb}{0.3, 0.6, 0.3}

\ifcomments
\newcommand{\RS}[1]{\textcolor{violet}{{\sf (Renate:} {\sl{#1})}}}
\newcommand{\KL}[1]{\textcolor{CornflowerBlue}{{\sf (Kristin:} {\sl{#1})}}}
\newcommand{\KS}[1]{\textcolor{KateColour}{{\sf (Kate:} {\sl{#1})}}}
\newcommand{\HT}[1]{\textcolor{cyan}{{\sf (Ha:} {\sl{#1})}}}
\newcommand{\MC}[1]{\textcolor{Bittersweet}{{\sf (Mingjie:} {\sl{#1})}}}
\newcommand{\SA}[1]{\textcolor{blue}{{\sf (Sarah:} {\sl{#1})}}}
\else
\newcommand{\RS}[1]{}
\newcommand{\KL}[1]{}
\newcommand{\KS}[1]{}
\newcommand{\HT}[1]{}
\newcommand{\MC}[1]{}
\newcommand{\SA}[1]{}
\fi

\title{Orienteering with one endomorphism}
\author[Arpin, Chen, Lauter, Scheidler, Stange, Tran]{Sarah Arpin, Mingjie Chen, Kristin E. Lauter, Renate Scheidler, Katherine E. Stange, Ha T. N. Tran}

\date{\today}

\address{%
Mathematics Institute, Universiteit Leiden, Leiden, The Netherlands}
\email{Sarah.Arpin@colorado.edu}
\address{School of Computer Science, University of Birmingham, University Road West, Birmingham, UK B15 2TT}
\email{m.chen.1@bham.ac.uk}
\address{Facebook AI Research, Meta, Seattle, WA}
\email{klauter@fb.com}
\address{%
Department of Mathematics and Statistics, University of Calgary,
2500 University Drive NW, Calgary, Alberta, Canada T2N 1N4}
\email{rscheidl@ucalgary.ca}
\address{%
Department of Mathematics, University of Colorado,
Campus Box 395, Boulder, Colorado 80309-0395}
\email{kstange@math.colorado.edu}
\address{Department of Mathematical and Physical Sciences, Concordia University of Edmonton, 
7128 Ada Blvd NW, Edmonton, AB T5B 4E4, Canada}
\email{hatran1104@gmail.com}



\thanks{
Katherine E. Stange and Sarah Arpin were supported by NSF-CAREER CNS-1652238.  Katherine E. Stange was also supported by Simons Fellowship 822143. Ha T. N. Tran was supported by the Natural Sciences and Engineering Research Council of Canada (NSERC) (funding RGPIN-2019-04209 and DGECR-2019-00428). R. Scheidler was supported by the Natural Sciences and Engineering Research Council of Canada (NSERC) (funding RGPIN-2019-04844).  Mingjie Chen was supported by NSF grants DMS-1844206 and DMS-1802161.
}

\begin{document}

\maketitle

\begin{abstract}
   
 In supersingular isogeny-based cryptography, the path-finding problem reduces to the endomorphism ring problem.  Can path-finding be reduced to knowing just one endomorphism?  It is known that a small endomorphism enables polynomial-time path-finding and endomorphism ring computation \cite{BonehLove}. An endomorphism gives an explicit orientation of a supersingular elliptic curve. In this paper, we use the volcano structure of the oriented supersingular isogeny graph to take ascending/descending/horizontal steps on the graph and deduce path-finding algorithms to an initial curve. Each altitude of the volcano corresponds to a unique quadratic order, called the primitive order. We introduce a new hard problem of computing the primitive order given an arbitrary endomorphism on the curve, and we also provide a sub-exponential quantum algorithm for solving it. In concurrent work \cite{WesolowskiOrientations}, it was shown  that the endomorphism ring problem in the presence of one endomorphism with known primitive order reduces to a vectorization problem, implying path-finding algorithms. Our path-finding algorithms are more general in the sense that we don't assume the knowledge of the primitive order associated with the endomorphism.
\end{abstract}

\section{Introduction}

The security of isogeny-based cryptosystems depends upon a constellation of hard problems.  Central are the path-finding problem introduced in~\cite{CharlesGorenLauter} (to find a path between two specified elliptic curves in a supersingular $\ell$-isogeny graph), and the endomorphism ring problem (to compute the endomorphism ring of a supersingular elliptic curve).  Only exponential algorithms are known for general path-finding, in the absence of information beyond the $j$-invariants you wish to navigate between.  However, if the endomorphism rings are known, the KLPT algorithm allows for polynomial-time path-finding \cite{KLPT}.  In fact, it is known that the path-finding and endomorphism ring problems are equivalent \cite{EHLMP_reductions,Wesolowski_IsogPathandEndoRing}.  These are the central problems in isogeny based cryptography, despite the recent complete break of SIDH/SIKE \cite{castryck-sidh-attack} and \cite{maino-attack-sidh}. The hardness of these problems is in no way affected by the attack, and they form the basis of the CGL hash function \cite{CharlesGorenLauter}, CSIDH \cite{CSIDH}, and OSIDH \cite{colo2019orienting}, among others.

A natural question to ask is whether knowledge of a single explicit endomorphism (which generates only a rank 2 subring of the rank 4 endomorphism ring) can be used for path-finding.  Answering this question is the goal of this paper:  we give explicit algorithms transforming knowledge of one endomorphism into a way-finding tool that can detect ascending, descending and horizontal directions with regards to the corresponding orientation, and use this to walk to $j=1728$.

By \emph{explicit endomorphism}, we mean one given in some form in which its action on the curve is computable, and its minimal polynomial is known (but note that, given an endomorphism, both its norm and trace are in many cases computable; see Section \ref{sec:runtime-lemmata}).  For example, such an endomorphism may be given as a rational map, or a composition chain of rational maps, and these are the two cases we focus on in this paper.  The data of such an endomorphism is equivalent to the data of an \emph{orientation} of a supersingular elliptic curve $E$, namely a map $\iota: K \rightarrow \QQ \otimes_\ZZ \End(E)$, where $K$ is the imaginary quadratic field generated by a root of the minimal polynomial of the endomorphism.  

The study of orientations provides some structure to the supersingular isogeny graph, which has recently been exploited \cite{colo2019orienting,onuki2021,DD}.  In particular, the $\ell$-isogeny graph of \emph{oriented} supersingular elliptic curves over $\overline{\FF}_p$ has a volcano structure familiar from the ordinary case: Each connected component consists of a single cycle, called a \textit{rim}, of vertices connected by \textit{horizontal} edges, and \textit{descending} edges connecting the rim the non-rim vertices at lower \textit{altitudes} of the volcano. Non-rim vertices only have ascending/descending edges.  This graph maps onto the supersingular $\ell$-isogeny graph over $\overline{\FF}_p$.  Our approach is to use the orientation provided by a given explicit endomorphism to discern ascending, descending and horizontal directions with regards to the volcano.   This provides a sort of tool for `orienteering'.  (The sport of \emph{orienteering} involves finding one's way to checkpoints across varied terrain using only map and compass.)

The core result of our paper is an algorithm that finds an $\ell$-isogeny path from a given supersingular elliptic curve $E$ to an initial curve $\Einit$, given a single explicit endomorphism of $E$.  We take $\Einit$ to be the curve with $j$-invariant $j=1728$, but other choices are possible (see Section~\ref{sec:otherinitial}).
The overall plan is as follows.  First, climb the oriented volcano from $E$, oriented by the given endormorphism, to the volcano rim (using the given endomorphism as our `orienteering tool').  Then, by orienting the curve $j=1728$ with the same field, we can climb to the rim from there also.  Finally, we attempt to meet by circling the rim.

This approach is limited by our ability to traverse a potentially large segment of the rim, or to hit the same rim in a large cordillera of volcanoes, whose size is generally equal to the class number of the corresponding quadratic order. If we simply walk the rim, then, classically, the runtime depends linearly on this class number.  Using a quantum computer to solve the vectorization problem (see Section~\ref{sec:quantum-rim-walking}) yields a subexponential algorithm.

\subsection{Main theorems}\label{sec:main_thms}
We rely on a number of heuristic assumptions: \begin{enumerate*}[label=(\roman*)]

\item The Generalized Riemann Hypothesis (hereafter referred to as GRH).

\item Powersmoothness in a quadratic sequence or form is as for random integers (a powersmooth analogue of the heuristic assumption underlying the quadratic sieve; see Heuristics~\ref{heur:translates} and \ref{heur:bqf}).

\item The orientations of a fixed $j$-invariant are distributed reasonably across all suitable volcanoes (Heuristic~\ref{heur:uniform-volcanoes}).

\item This distribution is independent of a certain integer factorization (Heuristic \ref{heuristic:uniform-cordillera}).

\item The aforementioned integer factorization is prime with the same probability as a random integer (Heuristic~\ref{heur:prime}; this heuristic is similar to those used in \cite{deQuehenEtAl_ImprovedTorPt} and \cite{KLPT}).

\end{enumerate*}

We state our main results here; their proofs can be found in Section~\ref{sec:proof-intro}. We use the notation $L_x(y) = \exp( O((\log x)^y (\log \log x)^{1-y} ))$.
Our first theorem gives a classical algorithm for $\ell$-isogeny path-finding that is subexponential in $\log p$ times a certain class number, for a wide range of input endomorphisms.

Let $\Delta'$ be the $\ell$-fundamental part of the discriminant $\Delta$ of an endomorphism $\theta$ of a supersingular curve $E$ (obtained\footnote{Except when $\ell = 2$, if $\Delta = 2^{2k}\Delta''$ where $4\nmid \Delta''$ and $\Delta'' \equiv 2,\,3 \amod 4$, then we set $\Delta' := 4\Delta''$.} by removing the largest even power of $\ell$).  Let $h_{\Delta'}$ be the class number of the quadratic order of discriminant $\Delta'$.  
Note that $\Delta'$ can be significantly smaller than ${\Delta}$. 

\begin{theorem}
\label{thm:intro-classical-kristin}
Assume $|\Delta'| \le p^2$.
Under the heuristic assumptions given above, there is a classical algorithm (given explicitly in Section~\ref{sec:consequences}; see also Algorithm~\ref{alg:pathto1728}) that, given an endomorphism $\theta$ of sufficiently large degree $d$ which can be efficiently evaluated on points, finds an $\ell$-isogeny path of length $O(\log p + h_{\Delta'})$ from $E$ to the curve with $j=1728$ in runtime $h_{\Delta'} L_d(1/2) \poly(\log p)$.
\end{theorem}

The term `sufficiently large' as applied to the degree $d$ asks that $L_d(1/2) \ge \poly(\log p)$. The term `efficiently' means that the endomorphism can be evaluated on points $P \in E(\FF_{p^k})$ in time polynomial in $\log d$, in $k$ and in $\log p$.  An example of such an endomorphism is an endomorphism given as a chain of isogenies of small degree, but we can also accommodate less efficient endomorphism representations. The full formal statement given in Theorem~\ref{thm:intro-classical} tracks the cost of this evaluation in the final runtime: it is assumed that the endomorphism $\theta$ can be evaluated on points $P \in E(\FF_{p^k})$ in time denoted $T_\theta(k,p)$, and the algorithm runtime, more precisely, is $T_\theta(L_d(1/2), p) + h_{\Delta'}L_d(1/2) \poly(\log p)$.  The algorithm comes in two phases:  the first phase is to represent the given endomorphism as an isogeny chain in runtime $T_\theta(L_d(1/2),p)$ depending on the representation of $\theta$; the second phase walks the isogeny graph using this representation and always has runtime $h_{\Delta'} L_d(1/2) \poly(\log p)$.  Phase one is included to allow for an abstract notion of an input endomorphism (see Section \ref{sec:represent}).

Any $\theta$ of degree $d$ which is represented in terms of rational maps has $T_\theta(k,p) = \poly(d, k,\log p)$, hence the final runtime would be $ \poly(d \log p) + h_{\Delta'} L_d(1/2) \poly(\log p)$.  But $\theta$ could be represented as a composition chain of isogenies in such a way that $T_\theta(k,p)$ is polynomial in $\log d$.  In this case, the final runtime would be $h_{\Delta'} L_d(1/2) \poly(\log p)$.
The factor $L_d(1/2)$ in the runtime arises from the need, during the algorithm, to sieve for endomorphisms of powersmooth degree amongst translates $\theta + [d]$, $d \in \ZZ$.

The algorithm can perform significantly better in some special cases, such as when the endomorphism is presented in an efficient way (in which case the first phase may be skipped), the curve is already at a rim (in which case the sieving is avoided), or the class number $h_{\Delta'}$ is small (in which case the walk is short), etc.  Specifically, modifications of the algorithm lead to special cases:
\begin{enumerate}
    \item If the input endomorphism is rationally represented in polynomial space, or the class number is polynomial in $\log p$ (with some conditions on $\ell$), the algorithm becomes polynomial in $\log p$ (Theorem~\ref{thm:intro-poly}).  The cryptographic weaknesses in these cases are already known by other methods \cite{BonehLove}.
    \item\label{specialcase2} If $\ell$ is inert in the field $\QQ(\sqrt{\Delta})$, the runtime improves for endomorphisms in suitable form to $L_d(1/2) + h_{\Delta'}  \poly(\log p)$, and the path length is improved to $O(\log p)$ (Proposition~\ref{prop:walkto1728}).
    \item If, in addition to \eqref{specialcase2}, $\Delta' = \Delta$, then the runtime improves further to $h_{\Delta'} \poly( \log p)$  (Proposition~\ref{prop:walkto1728}).
     \item If the degree of the endomorphism has $B(p)$-powersmooth factorization and its discriminant is coprime to $\ell$, then the runtime improves to $h_{\Delta'} \poly(B(p) \log p)$ (Theorem~\ref{thm:consequence-smooth}). 
    \item If degree and discriminant have suitable factorizations, then the runtime can improve to $\poly(\log p)$ even for non-small endomorphisms (Theorem~\ref{thm:consequence-poly}).  Such endomorphisms exist on all supersingular elliptic curves.
    
\end{enumerate}

Our second theorem gives a quantum algorithm for finding a smooth isogeny to an initial curve that runs in subexponential time in $\log |\Delta|$, and polynomial in $\log p$.

\begin{theorem}
\label{thm:intro-quantum-kristin}
Under the heuristic assumptions given above, there is a quantum algorithm (given explicitly in Section~\ref{sec:consequences}; see also Algorithm~\ref{alg:pathto1728-quantum}) which, given an endomorphism $\theta$ of degree $d$ and discriminant $\Delta$ satisfying $d \ll |\Delta| \le p^2$ and which can be efficiently evaluated on points, will return an $L_{|\Delta|}(1/2)$-smooth isogeny of norm $O(\sqrt{|\Delta|})$ from $E$ to the curve of $j=1728$, and runs in time subexponential in $\log |\Delta|$ and polynomial in $\log p$.
\end{theorem}

The term `efficiently' is as for Theorem~\ref{thm:intro-classical-kristin}.  In the full formal statement in Theorem~\ref{thm:intro-quantum}, the runtime, more precisely, is $T_\theta(O(\log^2 d), p) L_{|\Delta|}(1/2)$.

In both theorems, one may use other suitable initial curves besides $j=1728$; see Section~\ref{sec:otherinitial}.  

\subsection{A new hard problem} \label{sec:primitiveorderhardproblem}
Each altitude of an oriented volcano corresponds to a unique order in $K$, called the primitive order for the oriented curves at that altitude. The orders get smaller as the altitude gets lower, decreasing in index by $\ell$ at each step. 
Given an elliptic curve $E$ oriented by an endomorphism~$\theta$, the knowledge of the primitive order $\mathcal{O}$ with respect to $(E,\theta)$ plays a vital role in the algorithms: our classical algorithm computes a suborder of $\mathcal{O}$ whose relative index in $\mathcal{O}$ is coprime to $\ell$ in order to walk horizontally more efficiently; our quantum algorithm requires the full knowledge of $\mathcal{O}$ in order to solve the $\mathcal{O}$-vectorization problem.

The primitive order $\mathcal{O}$ doesn't come for free; this is Problem~\ref{pr-vec_intro}. To the best of our knowledge, this paper is the first work that introduces this problem as a hard problem and provides a quantum algorithm (Proposition~\ref{prop:primitive-orientation}) for solving it in quantum sub-exponential time. 

\begin{problem}[\textsc{PrimitiveOrientation}] \label{pr-vec_intro}
	Given a supersingular elliptic curve $E$, and an endomorphism $\theta \in \End(E)$, determine the quadratic order $\mathcal{O}$ such that $\mathcal{O} \cong  \QQ(\theta) \cap \End(E)$.
\end{problem}

The importance of Problem~\ref{pr-vec_intro} comes from the increasing interest in orientations on elliptic curves. Given an arbitrary supersingular elliptic curve $E$, the best known way to define an orientation on $E$ is to perform random walks on the supersingular isogeny graph until a cycle on $E$ is found, whereby an endomorphism on $E$ is obtained by composing the edges  along the cycle. In order to take advantage of the associated orientation, it is important to be able to answer Problem~\ref{pr-vec_intro}. This most general setting for obtaining orientations on $E$ is the setting our paper works with. 

Classically, however, solving Problem~\ref{pr-vec_intro} as discussed in Section~\ref{sec:quantum-prim-orientation} takes time polynomial in the largest prime power factor of $f$, where $f$ is the conductor of $\ZZ[\theta]$. Luckily, with our classical path-finding algorithm (Theorem~\ref{thm:intro-classical}), we are able to circumvent the issue by computing a specific smaller order instead, which can be done in polynomial time. This is also what makes our path-finding algorithms more general comparing to the algorithms in a related paper~\cite{WesolowskiOrientations} (See Section~\ref{sec:relatedwork}).

\subsection{Other algorithms presented}  Some of the explicit building blocks of the results above may have independent applications.  In particular, we provide algorithms for the following tasks, among others: 

\begin{enumerate}
    \item Section \ref{sec:navigating} provides methods for detecting ascending, descending and horizontal directions in general.
  \item Remark \ref{rem:matrixrep} explains how to adapt the algorithms of this paper to an endomorphism given as an approximate element of the Tate module (i.e. given by its action on $\ell$-torsion).
    \item Section \ref{sec:algs-chain} presents a technique for obtaining a prime-power powersmooth isogeny chain endomorphism from the same quadratic order as a given endomorphism (Algorithm~\ref{alg:suitable-chain}).
        \item Section~\ref{sec:1728} discusses an algorithm which computes an orientation of the elliptic curve of $j$-invariant $1728$ (or other suitable curves; see Section \ref{sec:otherinitial}) by an $\ell$-power multiple of a given discriminant (Algorithm~\ref{alg:1728orientation}).  In other words, given a quadratic order $\mathcal{O}$, it finds 
    $j=1728$ 
    somewhere in the cordillera of an order containing $\mathcal{O}$.  In fact, it finds arbitrarily many such orientations, moving gradually further `down' the volcanoes.  This algorithm runs in heuristic polynomial time when the discriminant is coprime to $p$ and less than $p^2$ in absolute value.
    \item Section \ref{sec:walkrim} concerns a method for computing the class group action of $\Cl(\mathcal{O})$ on $\SSOpr$, the set of curves primitively oriented by $\mathcal{O}$.  In fact, we demonstrate how to navigate $\SSOpr$ using the class group action of $\Cl(\mathcal{O}')$ for any $\mathcal{O}' \subseteq \mathcal{O}$ such  that $\ell \nmid [\mathcal{O}: \mathcal{O}']$. 

    \item Section \ref{sec:quantum_part1} provides two new quantum algorithms.  Namely, an algorithm for vectorization  on an oriented volcano rim (Proposition~\ref{prop:vectorization}; prior work includes \cite[Section 6.1]{chenu2021higherdegree}, \cite[Proposition 4]{WesolowskiOrientations}; our approach includes a novel method to evaluate isogenies on oriented curves), and for determining the quadratic order for which a given orientation is primitive (Proposition~\ref{prop:primitive-orientation}). We provide runtime analyses of these algorithms in turns of the degree and presentation of the given orientation and the prime $p$. 
    
    \item Given the input of an elliptic curve with orientation, Section \ref{sec:quantum_part2} provides a quantum algorithm (Algorithm~\ref{alg:pathto1728-quantum}) for finding a smooth isogeny to $j = 1728$. In Proposition~\ref{prop:quantum-walkto1728}, we analyze the runtime of this algorithm in terms of the degree and presentation of the given orientation and the prime $p$. 
    
    \item Section \ref{sec:division_by_ell}  contains an efficient algorithm for dividing an isogeny by $[\ell]$  (Algorithm~\ref{alg:divisionbyell}), originally outlined by McMurdy. We make McMurdy's approach explicit for an arbitrary small prime $\ell$ (he only made explicit the case $\ell = 2$, which is more straightforward).
\end{enumerate}

\subsection{Related work}\label{sec:relatedwork}

The question of the security of one endomorphism has recently been `in the air,' for example, with the uber isogeny assumption of \cite{SETA} (see Remark~\ref{remark:seta}).  Knowledge of a small explicit endomorphism is known to be a weakness \cite{BonehLove, BonehLove_ARXIV}. The work in this paper was done concurrently with \cite{WesolowskiOrientations}, which also provides path-finding algorithms in the setting of oriented curves. However, the two papers are very different in nature. The work in~\cite{WesolowskiOrientations} covers a web of reductions between a wide variety of hard problems related to orientations using quaternion algebras, which are of interest both in theory and applications. The path-finding algorithms are not stated as results in~\cite{WesolowskiOrientations} but rather implied by several reductions combined with  algorithms for solving the vectorization problem for oriented curves classically and quantumly. Our paper, by contrast, focuses on the path-finding problem. Our method is very explicit and works with isogenies and endomorphisms directly. We discuss the practical representations of isogenies and endomorphisms, provide complete algorithms, detailed runtime analysis and concrete numerical examples.  

The most important advantage of our path-finding algorithms over those given by~\cite{WesolowskiOrientations} is that we deal with orientations in greater generality. In both papers, an orientation is identified with an endomorphism. As discussed in Section~\ref{sec:primitiveorderhardproblem}, our input is an arbitrary endomorphism $\theta$, and it is a hard problem (Problem~\ref{pr-vec_intro}) to find the primitive order with respect to $(E,\theta)$. However, the input endomorphism $\theta$ in~\cite{WesolowskiOrientations} is one such that the order $\ZZ[\theta]$ is already the primitive order. Such an endomorphism is unlikely to be found for an arbitrary supersingular elliptic curve.

With due consideration of the added constraints on input for the algorithm in \cite{WesolowskiOrientations}, we can more accurately compare runtimes. Let $\Delta,\,\Delta'$ and $h_{\Delta'}$ be as in Section~\ref{sec:main_thms}. Classically, the runtime of the algorithm in~\cite{WesolowskiOrientations} is linear in $h_{\Delta'}^{1/2}$ whereas the runtime of our algorithm is linear in $h_{\Delta'}$. Quantumly, both algorithms run in subexponential time. If we consider the same input endomorphism in~\cite{WesolowskiOrientations} as in this work, then the runtime for solving Problem~\ref{pr-vec_intro} should be added to the runtime of~\cite{WesolowskiOrientations}. As discussed in Section~\ref{sec:quantum-prim-orientation}, solving Problem~\ref{pr-vec_intro} takes time polynomial in the biggest prime power factor of the conductor of $\ZZ[\theta]$ classically and subexponential time quantumly.

Lastly, the paper \cite{WesolowskiOrientations} assumes the stronger hypothesis that the discriminant of the input endomorphism has a known factorization. We do not assume this. The work \cite{WesolowskiOrientations} is not heuristic beyond a dependence on GRH and the solution to the vectorization problem (\hspace{1sp}\cite[Proposition 4]{WesolowskiOrientations}), whereas we rely on a number of heuristic assumptions as given in Section~\ref{sec:main_thms}. Our classical algorithm directly produces a path whose length depends on the class number (since it traverses a volcano rim), whereas a reduction to the vectorization problem as in the algorithms implied in \cite{WesolowskiOrientations} and our quantum algorithm produces a path of $\poly(\log p)$ length. 

Other related work includes \cite{CPV, DD}.  In \cite{papertwo}, the authors of the present article show that appropriately defined closed walks of the isogeny graph are in bijection with the rims of oriented isogeny volcanoes, giving a class number sum for their number.

\subsection{Other contributions}  We give careful runtime analyses for various tasks related to endomorphisms represented as rational functions or as composition chains of isogenies, including evaluation, translation, division-by-$[\ell]$, and Waterhouse transfer.  Additionally, we provide a review and some modest extensions to the theory of orientations as described in \cite{colo2019orienting, onuki2021}; see Section~\ref{sec:orientations}, in particular Section~\ref{sec:frob-class}.

In a follow-up paper \cite{papertwo}, we establish a theoretical bijection between volcano rims and cycles in the  $\ell$-isogeny graph, and address some of the aforementioned heuristics for oriented supersingular $\ell$-isogeny graphs used in this paper.

Throughout the paper we demonstrate our algorithms with a running example first introduced in Example~\ref{sec:intro_running_example}.  The examples are given in more detail in SageMath \cite{sagemath} worksheets with accompanying PDF details, available on GitHub \cite{github}.

\subsection{Outline}
In Section~\ref{sec:background}, we set some notations and conventions and also state a few runtime lemmata.  In Section~\ref{sec:orientations}, we introduce the main object of study, namely oriented $\ell$-isogeny graphs and their properties, including some heuristic behaviour.  In Section~\ref{sec:navigating}, the relationship between an endomorphism and an orientation is explained, and we also introduce a few new definitions that aid in navigating the oriented $\ell$-isogeny graph.  In Section~\ref{sec:combine_5_7_8_9}, we discuss the representation of endomorphisms, along with the basic functionalities for these representations required for later algorithms. We then compute orientations for the supersingular elliptic curve of $j$-invariant $1728$ in Section \ref{sec:1728}. In Sections \ref{sec:algs} and \ref{sec:path-1728}, we present algorithms for walking on an oriented $\ell$-isogeny graph and for classical path-finding to $j=1728$ and give  detailed runtime analyses and examples for illustration. We then provide quantum algorithms to solve the oriented vectorization and the primitive orientation problems in Section \ref{sec:quantum_part1} and a quantum algorithm for finding a smooth isogeny to $j=1728$ in Section \ref{sec:quantum_part2}.  In Section~\ref{sec:consequences}, we discuss the proofs of our main theorems as well as some special cases. Lastly, we leave to Section~\ref{sec:division_by_ell} the technical explanation of McMurdy's division-by-$\ell$ algorithm and provide its runtime analysis.  Throughout the paper, to aid in reading, important assumptions 
will be rendered in \textbf{bold}.

\subsection{Acknowledgements}
We would like to thank Catalina Camacho-Navarro, Elena Fuchs, Steven Galbraith, David Kohel, P\'eter Kutas, and Christophe Petit for helpful discussion.  We especially thank Benjamin Wesolowski, who took the time to share highly valuable suggestions on an earlier draft, particularly some important corrections concerning Proposition~\ref{prop:vectorization}.  We would also like to thank the conference \emph{Women in Numbers 5} for the opportunity to form this research group. 

\section{Background}\label{sec:background}

\subsection{Notations and conventions}
\label{ssec:notations}

Throughout the paper, let \textbf{$p$ be a cryptographically sized prime} (upon which runtimes will depend), and let \textbf{$\ell$ be a small prime} (whose size will be assumed $O(1)$ for runtimes).  In particular, $\ell \neq p$.  We will assume \textbf{both $p$ and $\ell$ are defined once throughout the paper} (so, for example, they will not be repeated as an input to every algorithm); the only exception being Sections~\ref{sec:quantum_part1} and \ref{sec:quantum_part2}. 

Every elliptic curve considered in the paper is assumed to be a \textbf{supersingular curve} over $\overline{\mathbb{F}}_{p}$. All such curves can be defined over $\mathbb{F}_{p^2}$.  Every isogeny and endomorphism is  assumed to have domains and codomains which are curves of this type. We use the notation $\End(E)$ for the endomorphism ring of the elliptic curve $E$ over $\overline{\mathbb{F}}_p$, and $\End^0(E) := \QQ \otimes_\ZZ \End(E)$ for the endomorphism algebra of $E$.  We use the notation $O_E$ for the identity element of an elliptic curve $E$, and $j(E)$ for the $j$-invariant.  We use the variables $\varphi$ and $\psi$ to denote isogenies, while $\theta$ is generally reserved for endomorphisms.  The dual isogeny to an isogeny $\varphi$ is denoted by $\widehat{\varphi}$.  Let $E^{(p)}$ denote the curve obtained by the action of Frobenius on $E$ (acting on the Weierstrass coefficients).  Let $\pi_p:E\to E^{(p)}$ denote the Frobenius isogeny, given by $\pi_p(x,y) = (x^p,y^p)$. Note that Frobenius is an endomorphism if $E$ is defined over $\Fp$.  Frobenius also acts on any isogeny $\varphi: E \rightarrow E'$ (acting on its coefficients) to give $\varphi^{(p)}: E^{(p)} \rightarrow (E')^{(p)}$ of the same degree.  Unless otherwise specified (such as Frobenius), \textbf{isogenies will be assumed to be separable} throughout the paper (many of the algorithms herein would not apply to inseparable endomorphisms or isogenies).

 There is only one fixed supersingular $\ell$-isogeny graph under consideration at any time, which we denote simply by $\mathcal{G}$. Namely, this is the graph whose vertices are $\overline{\mathbb{F}}_p$-isomorphism classes of supersingular elliptic curves (which we will often refer to simply by their $j$-invariants), and whose directed edges are $\ell$-isogenies (when there are no extra automorphisms, we can identify dual pairs to create an undirected graph).

We consider imaginary quadratic fields $K = \QQ(\sqrt{\Delta})$, where $\Delta < 0$ is a 
fundamental discriminant.  Then the ring of integers has the form $\mathcal{O}_K = \ZZ[\omega]$, where
\[
\omega = 
\begin{cases}
   \frac{1 + \sqrt{\Delta}}{2}   & \text{if } \Delta \equiv 1 \amod 4 , \\[3pt]
   \frac{\sqrt{\Delta}}{2}  &  \text{if } \Delta \equiv 0 \amod 4 .
\end{cases}
\]
Since we sometimes have multiple quadratic orders under consideration, we use the notation $(\alpha,\beta)_\mathcal{O}$ for the ideal generated by $\alpha$ and $\beta$ in $\mathcal{O}$.
The (possibly non-maximal) orders $\mathcal{O}$ of $K$ are parameterized by a positive integer called the \emph{conductor}.  If $\mathcal{O}$ has conductor $f$, then $\mathcal{O} = \ZZ[f\omega]$.  If $\ell \nmid f$, then we say that both~$\mathcal{O}$ and its discriminant are \emph{$\ell$-fundamental}.  Given a discriminant $\Delta$, its \emph{$\ell$-fundamental part} is the maximal $\ell$-fundamental discriminant dividing $\Delta$.

Write $B_{p,\infty}$ for the rational quaternion algebra ramified at $p$ and $\infty$.  \textbf{Every quadratic field $K$ is assumed to embed in the quaternion algebra $B_{p,\infty}$}, i.e.\ to be an imaginary quadratic field in which $p$ does not split \cite[Proposition 14.6.7(v)]{voight}; the only exception is in the discussion of Heuristic~\ref{heur:prime}.   Every quadratic order $\mathcal{O}$ is assumed to generate such a field $K$, and to \textbf{have discriminant not divisible by $p$}.  Every quadratic discriminant is assumed to be the discriminant of such a quadratic order $\mathcal{O}$, and we write~$\Delta_\mathcal{O}$.  We denote by $\mathcal{O}_K$ the maximal order of the quadratic field $K$ and reserve $\Delta_K$ for the discriminant of~$\mathcal{O}_K$.  

Complex conjugation (which is also the action of $\Gal(K/\QQ)$) is denoted by an overline: $\alpha \mapsto \overline{\alpha}$.  
We use the notation $\ClO$ and $\hO$ for the class group and class number, respectively, of a quadratic order $\mathcal{O}$.  

The reduced norm and trace of $B_{p,\infty}$ coincide with the norm and trace of an element when it is considered as a quadratic algebraic number; when we discuss norm and trace it is always this we refer to.

For runtime analyses we use big $O$ notation, including soft $\widetilde{O}$ for absorbing log factors.   The notation $\mulM(n)$ will indicate the runtime of field operations (addition, multiplication, inversion) in a finite field of cardinality~$n$; here, we note that $\mulM(n^k) = O(\mulM(n))$ when $k$ is constant.  In the later portions of the paper we are mainly concerned with the distinction between polynomial, subexponential and exponential algorithms.  We write runtime as $\poly(x)$ if there exists a polynomial $f$ so the runtime is $O(f(x))$.  When we are concerned only with whether runtime is polynomial, we will suppress the notation $\mulM$, by assuming that $\mulM(n) = \poly(\log n)$.  For subexponential runtimes, we use notation $L_x(y) = \exp( O( (\log x)^y (\log \log x)^{1-y} ) )$.  

For general background on isogeny-based cryptography and supersingular isogeny graphs, we will assume the reader is familiar with a resource such as \cite[Section 2]{EHLMP_reductions} or \cite{MathOfIsog}.

\subsection{Runtime lemmata}
\label{sec:runtime-lemmata}

In this section, we recall some basic runtimes for isogenies and torsion points, etc.  The first lemma is standard.

\begin{lemma}\label{lemma:lincomb}
Given $P, Q \in E[N]$, and $0 \le a, b < N$, computing $[a]P + [b]Q$ takes time $O((\log N) \mulM(p^{N^2}))$.
\end{lemma}

 \begin{lemma}[{\cite[Corollary 2.5]{bank2019cycles}}]
 \label{lemma:isog12}
 Let $\varphi: E \rightarrow E'$ be an isogeny between two supersingular elliptic curves, both defined over $\FF_{p^2}$.  Then $\varphi$ is defined over $\FF_{p^{12}}$.  If neither of $j(E)$ or $j(E')$ are $0$ or $1728$, then $\varphi$ is defined over $\FF_{p^4}$.
 \end{lemma}

\begin{lemma}
\label{lem:torsion-basis}
Let $t$ denote the smallest integer such that $E[N] \subseteq E(\FF_{p^t})$.  In particular, $t \le N^2-1$.
         Finding a basis of $E[N]$ has runtime $\widetilde{O}(N^4 (\log p) \mulM(p^{N^2}))$.
\end{lemma}

\begin{proof}
This can be proved by adapting the second paragraph of the proof of Lemma 5 in \cite{GalbraithPetitSilva_IdProtocols}.  In particular, the limiting runtime is the call to the equal-degree factorization algorithm of \cite{vzGathenShoup}, which takes time $\widetilde{O}(N^4 (\log p) \mulM(p^{N^2}))$.  See also \cite[Lemma 6.9]{bank2019cycles}.
\end{proof}

In practice, this can be done much faster in some cases, e.g. when $N$ is large and $t$ is small.

\begin{lemma}
\label{lemma:eval-torsion}
Consider an isogeny $\varphi: E \rightarrow E'$ of degree $d$, and a point $P \in E(\FF_{p^t})$, where $12 \mid t$.  Then computing $\varphi(P)$ takes time $O(d\mulM(p^t))$.  In particular, if $P \in E[N]$, then the time taken is $O(d \mulM(p^{\lcm(12,N^2)}))$. 
\end{lemma}

\begin{proof}

Write $\varphi$ as a rational map $\varphi(x,y) = (\varphi_1(x), \varphi_2(x) y)$; here the denominators and numerators of $\varphi_1(x)$ and $\varphi_2(x)$ are polynomials in $x$ of degree at most $3d$. By Lemma \ref{lemma:isog12}, we can assume that their coefficients are in $\FF_{p^{12}} \subseteq \FF_{p^t}$. 
 To compute  $\varphi(P)$, we apply Horner's algorithm \cite[p.\ 467]{knuth81}, which requires $O(d)$ operations in the field. 
 Assume that $P$ is an $N$-torsion point on $E$.  Then $t$ can be chosen such that $t \le \lcm(t, N^2)$ by Lemma \ref{lem:torsion-basis}. 
 \end{proof}

 In the case that $\varphi = [n]$ for some integer $n$, it is more efficient to use a standard double-and-add approach, which will also take polynomial time in the degree.

\begin{lemma}[{\cite{Velu}, \cite[Theorem 3.5]{shumow2009isogenies}, \cite[Section 5.1]{IonicaJoux}}]
\label{lemma:velu}
V\'elu's formulas for an isogeny of degree $d$ compute the isogeny in time $\widetilde{O}(d \mulM(p^{d^2}))$.
\end{lemma}

By Lemma \ref{lemma:isog12}, the isogeny created via V\'elu's formulas has coefficients in the field $\FF_{p^{12}}$.

\begin{lemma}
\label{lemma:composing}
Let $\varphi : E \rightarrow E'$ and $\psi:E' \rightarrow E''$ be isogenies represented as rational maps, of respective degrees $d$ and $d'$, where $E,E',E'', \varphi$ and $\psi$ are defined over some finite field $\FF$. Then computing the composition $\psi \circ \varphi: E \rightarrow E''$ as a rational map takes time $\widetilde{O}(dd'\mulM(\#\FF))$.
\end{lemma}
\begin{proof}
As usual, write $\varphi = \left ( \frac{u(x)}{v(x)}, \frac{s(x)}{t(x)}y \right )$ where $u(x), v(x), s(x), t(x) \in \FF[x]$ are polynomials of degree $O(d)$ with $\gcd(u,v) = \gcd(s,t) = 1$. Similarly, write $\psi = \left ( \frac{u'(x)}{v'(x)}, \frac{s'(x)}{t'(x)}y \right )$ with analogous conditions on $u'(x)$, $v'(x)$, $s'(x)$, $t'(x) \in \FF[x]$. Then 
\[ \psi \circ \varphi = \left ( \frac{u'(\frac{u(x)}{v(x)})}{v'(\frac{u(x)}{v(x)})}, \frac{s'(\frac{u(x)}{v(x)})}{t'(\frac{u(x)}{v(x)})} \frac{s(x)}{t(x)} y \right ) \ . \]
Obtaining $\psi \circ \varphi$ requires computing four compositions of the form $f(\frac{u(x)}{v(x)})$ where $f \in \{u', v', s', t' \}$ has degree $O(d')$. Writing $f(x) = \sum_{i=0}^n f_i x^i$ with $n = O(d')$, we have
\[ f \left (\frac{u(x)}{v(x)} \right ) = \frac{F(u(x), v(x))}{v(x)^n} \quad \mbox{where} \quad F(x,y) = \sum_{i=0}^n f_i x^i y^{n-i} \ . \]
The computation of $F(u(x),v(x))$ is dominated by computing the powers of $u(x)$ and $v(x)$ which can be accomplished in time $\widetilde{O}(dd'\mulM(\#\FF))$ using fast polynomial multiplication \cite{harvey2019polynomial}. An alternative way to compute $F(u(x),v(x))$ that is slightly faster but has asymptotically the same runtime is via the Horner-like recursion
\[ F_n(x) = f_n \ , \qquad F_{i-1}(x) = f_{i-1}v(x)^{n-i+1} + F_i(x) u(x) \quad (n \ge i \ge 1) \ , \]
where it is easy to see that $F_0(x) = F(u(x),v(x))$.
\end{proof}

\begin{lemma}
\label{lemma:translating}
Let $E$ be an elliptic curve defined over some finite field $\FF$, $\theta \in \End(E)$ an endomorphism represented as a rational map, and $N$ an integer.  Then computing the endomorphism $\theta + [N] \in \End(E)$ as a rational map takes time $\widetilde{O}(\max \{ \deg \theta, N^2 \} \mulM(\#\FF))$.
\end{lemma}

\begin{proof}
By \cite[Exercise 3.7, pp.\ 105f.]{silverman2009arithmetic}, we have
\[ [N](x,y) = \left ( \frac{\phi_N(x)}{\psi_N(x)^2}, \frac{\omega_N(x,y)}{\psi_N(x,y)^3} \right ) \ , \]
where $\phi_N = x \psi_N^2 - \psi_{N+1}\psi_{N-1}$, $\omega_n = (\psi_{N+2}\psi_{N-1}^2 - \psi_{N-2}\psi_{N+1}^2)/4y$ and $\psi_n$ is the $n$-th division polynomial on $E$. The required division polynomials have degree $O(N^2)$ and can be computed in $O(\log(N))$ steps using the recursive formulas
\[\psi_{2n+1} = \psi_{n+2} \psi_n^3 - \psi_{n-1} \psi_{n+1}^3 \ , \quad \psi_{2n} = \frac{1}{2y} \psi_n (\psi_{n+2} \psi_{n-1}^2 - \psi_{n-2} \psi_{n+1}^2) \ . \]
Using the point addition formulas on $E$ and fast polynomial multiplication techniques \cite{harvey2019polynomial}, the rational map $\theta + [N]$ can be computed using $\widetilde{O}(\max \{ \deg \theta, N^2 \})$ operations in $\FF$.
\end{proof}

 Throughout the paper, we will assume that \textbf{all endomorphisms are provided with a trace and norm} (which is the same as the degree) that carries through computations; see Section \ref{sec:represent}.  If the trace is not provided, then it can be computed using \cite[Lemma 1]{WesolowskiOrientations}, \cite[Lemma 4]{EHLMP_reductions}, \cite[Theorem 3.6]{bank2019cycles}.

\section{Oriented isogeny graphs}
\label{sec:orientations}

In this section, we recall and strengthen basic results about oriented isogeny graphs, mainly based on work of Col\`{o}-Kohel \cite{colo2019orienting} and Onuki \cite{onuki2021}, and provide some minor new extensions of the general theory.

\subsection{Orientations}
\label{sec:isog-graph}

Fixing a curve $E$, we have $\End^0(E) \cong B_{p,\infty}$.  
The field $K$ embeds into $B_{p,\infty}$ if and only if~$p$ does not split in $K$.  There may be many distinct such embeddings.  
We define a \emph{$K$-orientation} of $E$ to be an embedding
$\iota: K \rightarrow \End^0(E)$. 
If $\mathcal{O}$ is an order of $K$, then an \emph{$\mathcal{O}$-orientation} is a $K$-orientation such that $\iota(\mathcal{O}) \subseteq \End(E)$.
We say that a $K$-orientation $\iota$ is a \emph{primitive} $\mathcal{O}$-orientation if $\iota(\mathcal{O}) = \End(E) \cap \iota(K)$. It will often be expedient to have a local notion of primitivity:  for a prime~$\ell$, we say that a $K$-orientation~$\iota$ is an \emph{$\ell$-primitive $\mathcal{O}$-orientation} if it is an $\mathcal{O}$-orientation and the index $[\End(E) \cap \iota(K) : \iota(\mathcal{O})]$ is coprime to~$\ell$.  In particular, a primitive $\mathcal{O}$-orientation is exactly one which is $\ell$-primitive for all primes $\ell$.

If $\varphi: E \rightarrow E'$ is an isogeny of degree $\ell$, where $\iota$ is a $K$-orientation of $E$, then there is an induced $K$-orientation $\iota' = \varphi_*(\iota)$ on $E'$ defined to be $\varphi_*(\iota)(\omega) := \frac{1}{\ell} \varphi \circ \iota(\omega) \circ \widehat{\varphi} \in \End^0(E')$.

A \emph{$K$-oriented elliptic curve} is a pair $(E, \iota)$ where $\iota: K \rightarrow \End^0(E)$ is a $K$-orientation.  
 An isogeny of $K$-oriented elliptic curves $\varphi: (E,\iota) \rightarrow (E', \iota')$ is an isogeny $\varphi:E \rightarrow E'$ such that $\iota' = \varphi_*(\iota)$; we call this a $K$-oriented isogeny and write $\varphi \cdot (E, \iota) = (\varphi(E), \varphi_*(\iota))$.  One verifies directly that $\varphi_2 \cdot \varphi_1 \cdot (E, \iota) = (\varphi_2 \circ \varphi_1) \cdot (E,\iota)$.   A $K$-oriented isogeny is a \emph{$K$-isomorphism} if it is an isomorphism of the underlying curves.

\subsection{Oriented isogeny graphs}
 Fixing a quadratic field $K$, we define the graph $\mathcal{G}_K$ of $K$-oriented supersingular curves over $\overline{\FF}_p$. This is the graph whose vertices are $K$-isomorphism classes 
 of pairs $(E, \iota)$ and for which an edge joins $(E,\iota)$ and $(E', \iota')$ for each $K$-oriented isogeny (defined over $\overline{\FF}_p$) of degree $\ell$ between these oriented curves.  
 If $\varphi : (E, \iota) \rightarrow (E', \iota')$ is a $K$-oriented isogeny, then $\widehat{\varphi} : (E', \iota') \rightarrow (E, \iota)$ is also one (since $\widehat{\varphi}_*(\iota') = \widehat{\varphi}_*(\varphi_*(\iota)) = [\ell]_*(\iota) = \iota$). 
 Therefore the edges may be taken to be undirected by pairing isogenies with their duals, when the vertices involved are not $j=0$ or $1728$. 
 Also, isogenies are taken up to equivalence, meaning we quotient by the same isomorphisms as for the vertices; see \cite[Definition 4.1]{onuki2021}.  The graph $\mathcal{G}_K$ has (out-)degree $\ell+1$ at every vertex. 
 (Note that our graph differs slightly from the definition in \cite[Section 4]{onuki2021}, where only the images of curves over a number field with complex multiplication are included; we discuss this distinction in the next section.)  This graph was first studied in \cite{colo2019orienting}.

Every $K$-orientation is a primitive $\mathcal{O}$-orientation for a unique order $\mathcal{O} := \iota(K) \cap \End(E)$.  Therefore, the set of vertices of $\mathcal{G}_K$ is stratified by the order $\mathcal{O}$ by which a vertex is primitively oriented.

\begin{definition} Let $\SSOpr$ denote the set of isomorphism classes of $K$-oriented supersingular elliptic  curves for which the orientation is a primitive $\mathcal{O}$-orientation.
\end{definition}

This set is non-empty if and only if $p$ is not split in $K$ and does not divide the conductor of $\mathcal{O}$ \cite[Proposition 3.2]{onuki2021}.  As mentioned in Section~\ref{ssec:notations}, we make those assumptions throughout the paper.

Let $\varphi: (E,\iota) \rightarrow (E',\iota')$ be a $K$-oriented $\ell$-isogeny.  Suppose that $\iota$ is a primitive $\mathcal{O}$-orientation and $\iota'$ is a primitive $\mathcal{O}'$-orientation. There are exactly three possible cases:
\begin{enumerate}
    \item $\mathcal{O} = \mathcal{O}'$, in which case we say $\varphi$ is \emph{horizontal},
        \item $\mathcal{O} \supsetneq \mathcal{O}'$, in which case $[\mathcal{O}:\mathcal{O}'] = \ell$ and we say $\varphi$ is \emph{descending},
            \item $\mathcal{O} \subsetneq \mathcal{O}'$, in which case $[\mathcal{O}':\mathcal{O}] = \ell$ and we say $\varphi$ is \emph{ascending}.
\end{enumerate}

\begin{example}[\textbf{Introducing our running example}]\label{sec:intro_running_example}
  To illustrate the algorithms in this paper, we consider supersingular elliptic curves defined over $\Fpbar$ for $p = 179$. As $p\equiv 3\pmod{4}$, the curve $E:y^2 = x^3 -x$ with $j(E) = 1728$ is supersingular. This curve is well-known to have extra automorphisms, and its endomorphism ring is generated by the endomorphisms $[1],[i],\frac{[1] + \pi_p}{2}, \frac{[i] + [i]\circ\pi_p}{2}$, where $[i](x,y):=(-x,iy)$ and $\pi_p$ is as defined in Section~\ref{ssec:notations}. We define $K:=\QQ(\sqrt{\Delta})$ with $\Delta = -47$ and $\omega = \frac{1 + \sqrt{-47}}{2}$. We consider the oriented $2$-isogeny graph of supersingular elliptic curves with respect to this imaginary quadratic field $K$. 
\end{example}

\begin{figure}[h!]

\begin{center}

 \includegraphics[width=15cm,height=10cm,keepaspectratio]{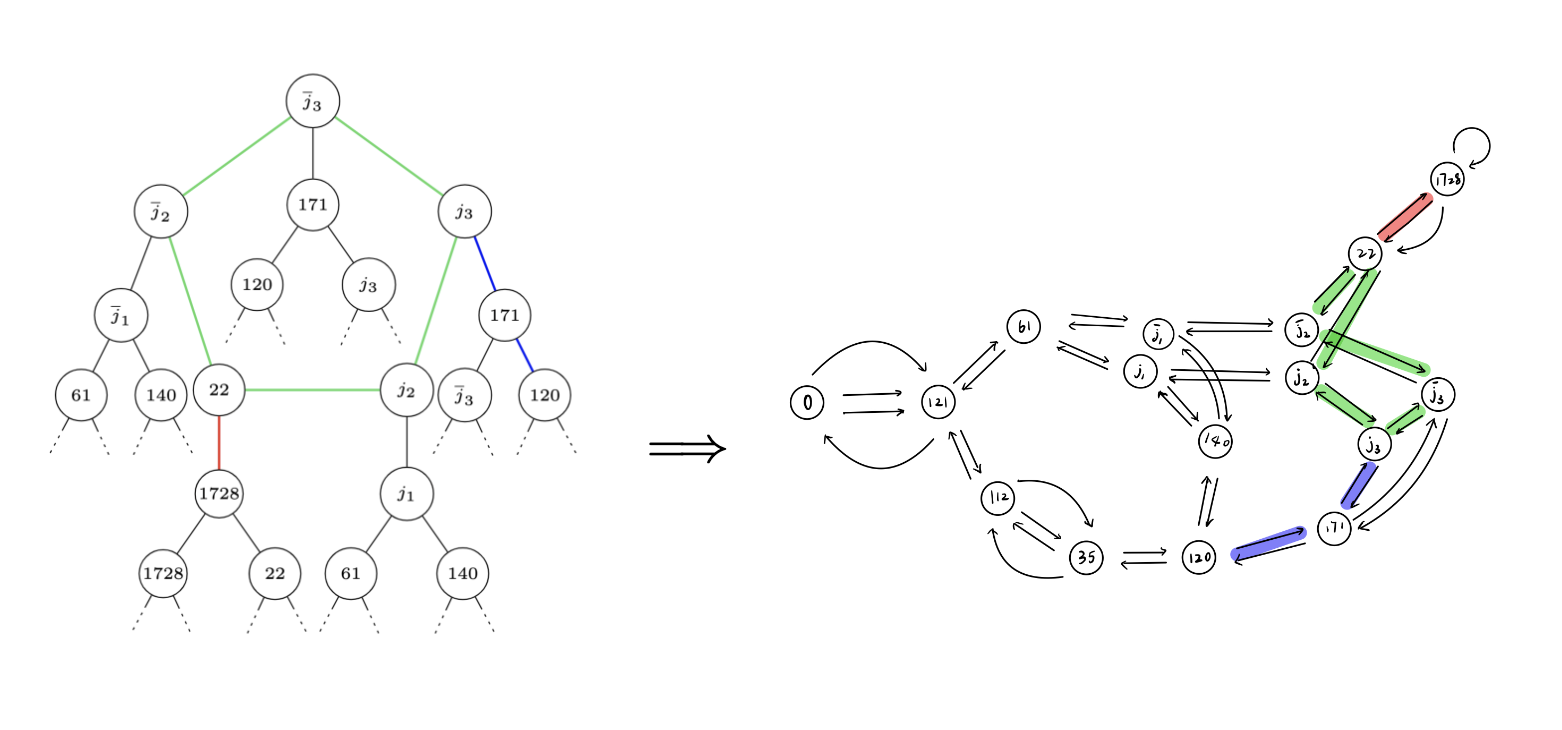} 
 \caption{On the left hand side is a component of $\mathcal{G}_K$ for $p = 179$, $\ell = 2$ and $K = \QQ(\sqrt{-47})$. On the right hand side is the supersingular $2$-isogeny graph over $\FF_{p^2}$. Here $j_1 = 64i+5 ,\,j_2 =99i+107 ,\,j_3 =5i+109$, where $i$ denotes a root of $-1$ in $\FF_{p^2}$. Since the oriented graph is undirected while the supersingular isogeny graph is directed, we have undirected edges in the left graph and directed edges in the right graph. Note that the green 5-cycle represents the rim of the volcano.}\label{fig:path_to_Eint} 
\end{center}
\end{figure}

\subsection{Frobenius and class group actions}
\label{sec:frob-class}

Let $\mathcal{O}$ be a quadratic order of $K$.  Next we define an action of $\ClO$ on $\SSOpr$. 
For an invertible ideal $\mathfrak{a}$ of~$\mathcal{O}$ embedded into $\End(E)$ via a $K$-orientation $\iota$, there exists a horizontal isogeny $\varphi_\mathfrak{a}$ defined by the kernel $E[\iota(\mathfrak{a})] := \cap_{\theta \in \iota(\mathfrak{a})} \ker(\theta)$ \cite[Section 3]{colo2019orienting}\cite[Proposition 3.5]{onuki2021}, and we write
\[
\mathfrak{a} \cdot (E, \iota) := \varphi_\mathfrak{a} \cdot (E,\iota).
\]
A different choice of $\varphi_\mathfrak{a}$ with the same kernel gives an isomorphic oriented curve \cite[Section 3.3]{onuki2021}, so this is well-defined on the oriented $\ell$-isogeny graph $\mathcal{G}_K$.
The action of $\ClO$ is free, but not necessarily transitive; it may have as many as two orbits \cite[Proposition 3.3]{onuki2021}.  In particular,
\begin{equation}
    \label{eqn:cl-frob}
    \#\SSOpr \in \{ \hO, 2\hO \}.
\end{equation}

Consider the effect of the Frobenius isogeny on an oriented curve, namely $\pi_p \cdot (E, \iota) = (E^{(p)}, \iota^{(p)})$ where $\iota^{(p)} := (\pi_p)_*(\iota)$.  
For any isogeny $\varphi$, we have $\pi_p \circ \varphi (x,y) = \varphi^{(p)}(x^p,y^p) = \varphi^{(p)} \circ \pi_p (x,y)$.  Hence, one has $(\pi_p)_*(\iota)(\alpha) = \frac{1}{p}\pi_p \circ \iota(\alpha) \circ \widehat{\pi_p} = \frac{1}{p} \iota(\alpha)^{(p)} \circ \pi_p \circ \widehat{\pi_p} = \iota(\alpha)^{(p)}$.   Since $\varphi \mapsto \varphi^{(p)}$ gives an isomorphism $\End(E) \cong \End(E^{(p)})$, we see that $\pi_p$ is horizontal, so this gives an action on $\SSOpr$ for any $\mathcal{O}$ by the two-element group $\{ 1, \pi_p \} = \langle \pi_p \rangle$.  In fact, it is an action on the graph $\mathcal{G}_K$, not just the vertices, i.e.\ it preserves adjacency.  Onuki shows that when there are two orbits of the action of $\ClO$ on $\SSOpr$, then the second orbit can be reached from the first by the action of Frobenius \cite[Proposition 3.3]{onuki2021}.  In \cite{papertwo}, a complete classification of when there are two (instead of one) orbit is given.

For our algorithms, we will sometimes need to compute the action of $\mathcal{O}$ on $\SSOpr$ without actually knowing~$\mathcal{O}$.  We can define and use an action of a suborder $\mathcal{O}' \subseteq \mathcal{O}$ as a proxy.  To accomplish this, define, for $[\mathfrak{a}'] \in \Cl(\mathcal{O}')$, that $\mathfrak{a}' \cdot (E, \iota) := \cap_{\theta \in \iota(\mathfrak{a}')} \ker(\theta)$. 
Observe that there is a homomorphism $\rho: \Cl(\mathcal{O}') \rightarrow \Cl(\mathcal{O})$.  Using the previous proposition, this gives a group action of $\Cl(\mathcal{O}')$ on $\SSOpr$.  The following proposition states that these two definitions agree.  Although it implements the action of $\mathcal{O}$, using the kernel intersection formula does not require knowledge of $\mathcal{O}$.

\begin{proposition}
\label{prop:non-prim-action}
Let $\mathcal{O}' \subseteq \mathcal{O}$ with relative index $f$.
Let $\mathfrak{a}'$ be an ideal of $\mathcal{O}'$ which has norm coprime to~$f$.
Suppose that $E$ has a $K$-orientation $\iota$ which is $\mathcal{O}$-primitive.  
Let $\varphi_{\mathfrak{a}'}$ be defined as the isogeny with kernel $\cap_{\theta \in \iota(\mathfrak{a}')} \ker(\theta)$.  Let $\mathfrak{a} := \mathfrak{a}' \mathcal{O}$ be the extension of $\mathfrak{a}'$ to $\mathcal{O}$.    Then $\mathfrak{a} \cdot (E, \iota) = \varphi_{\mathfrak{a}'}(E,\iota)$.
\end{proposition}

\begin{proof}
We have $\iota(\mathfrak{a}') \subseteq \iota(\mathfrak{a}) \subseteq \End(E)$.  We will show $\cap_{\theta \in \iota(\mathfrak{a}')} \ker(\theta) =\cap_{\theta \in \iota(\mathfrak{a})} \ker(\theta)$.  From that, we can complete the proof, since
\[
\mathfrak{a} \cdot (E, \iota) = 
\varphi_{\mathfrak{a}}(E, \iota) =
\varphi_{\mathfrak{a}'}(E,\iota).
\]
    We immediately have $\cap_{\theta \in \iota(\mathfrak{a}')} \ker(\theta) \supseteq \cap_{\theta \in \iota(\mathfrak{a})} \ker(\theta)$.  We will show the index between these two groups must divide a power  of $f$.  But the larger of the groups has cardinality coprime to $f$ by hypothesis.  So this would imply they are equal.  
    
    Write $\mathfrak{a}' = \alpha_1 \mathcal{O}' + \alpha_2 \mathcal{O}'$ and $\mathcal{O} = \ZZ + g\omega \ZZ$ using the notation of Section~\ref{ssec:notations}.   Then
    \begin{align*}
    \cap_{\theta \in \iota(\mathfrak{a}')} \ker (\theta)
    &= \ker( \iota( \alpha_1))\cap \ker( \iota( \alpha_2)) \cap \ker( \iota( \alpha_1fg\omega)) \cap \ker( \iota( \alpha_2fg\omega)), \\
    \cap_{\theta \in \iota(\mathfrak{a})} \ker (\theta)
    &= \ker( \iota( \alpha_1))\cap \ker( \iota( \alpha_2)) \cap \ker( \iota( \alpha_1g\omega)) \cap \ker( \iota( \alpha_2g\omega)). 
    \end{align*}
     We have $\ker( \iota( \alpha_i g \omega) ) \subseteq \ker( \iota( \alpha_i f g \omega ))$ with index $f^2$.  Thus the index of $\cap_{\theta \in \iota(\mathfrak{a})} \ker (\theta)$ inside $\cap_{\theta \in \iota(\mathfrak{a}')} \ker (\theta)$ must divide a power of $f$.
\end{proof}

\subsection{Volcano structure}
\label{sec:volcano-structure}

Any component of the oriented $\ell$-isogeny graph $\mathcal{G}_K$  has a \emph{volcano structure} (see Figure~\ref{fig:path_to_Eint}), which is made precise by the following statement.  (This behaviour is similar to the ordinary $\ell$-isogeny graph, except here volcanoes have no floor; they descend forever.)  Here we remind the reader that $p \neq \ell$ throughout the paper.

\begin{proposition}[{\cite[Proposition 4.1]{onuki2021}}]
\label{prop:local-volcano}
Consider a vertex $(E,\iota)$ of the oriented $\ell$-isogeny graph associated to $K$, a quadratic field of discriminant $\Delta$.  Suppose that $\iota$ is a primitive $\mathcal{O}$-orientation for $E$. If $\ell$ does not divide the conductor of $\mathcal{O}$, then the following hold. 
\begin{enumerate}
    \item There are no ascending edges from $(E,\iota)$.
    \item There are $\left( \frac{\Delta}{\ell} \right) + 1$ horizontal edges from $(E,\iota)$.
    \item There remaining edges from $(E, \iota)$ are descending.
\end{enumerate}
If $\ell$ divides the conductor of $\mathcal{O}$, then the following hold.
\begin{enumerate}
    \item There is exactly one ascending edge from $(E,\iota)$.
    \item The remaining edges from $(E,\iota)$ are descending.
\end{enumerate}
\end{proposition}

When $\mathcal{O}$ has unit group $\{ \pm 1 \}$, i.e.\ except for the Gaussian and Eisenstein integers, the out-degree of $(E,\iota)$ is $\ell+1$.  For the out-degree in these special cases, see \cite[Proposition 2.11]{papertwo}.

Proposition \ref{prop:local-volcano} implies that each connected component of the oriented $\ell$-isogeny graph $\mathcal{G}_K$ is a \emph{volcano}, containing a \emph{rim} comprised of the vertices with no ascending edges. Each vertex on a rim is the root of a tree that radiates infinitely downward and in which each node other than the root generically has one parent and $\ell$ children. The vertices at altitude $r$ are precisely those pairs $(E, \iota)$ for which $\iota$ is a primitive $\mathcal{O}$-orientation such that the conductor of $\mathcal{O}$ has $\ell$-adic valuation $r$. Specifically, the vertices at the rims are exactly those for which $\mathcal{O}$ is $\ell$-fundamental. For any fixed $\ell$-fundamental order $\mathcal{O}$, we define the \emph{$\mathcal{O}$-cordillera} to be subgraph of~$\mathcal{G}_K$ comprised of only those volcanoes whose rims are pairs $(E, \iota)$ with $\iota$ a primitive $\mathcal{O}$-orientation. The vertices at the rims of the $\mathcal{O}$-cordillera are exactly $\SSOpr$.

The action of an ideal class $[\mathfrak{a}] \in \ClO$ gives a permutation on $\SSOpr$, which we can visualize as a directed graph. 
This consists of cycles, all of which are the same size, given by the order of $[\mathfrak{a}]$ in $\ClO$. 
Applying this to a prime ideal $\mathfrak{l}$ of $\mathcal{O}$ lying above $\ell$, the rims of the $\mathcal{O}$-cordillera are exactly these cycles.   
All these rims have the same size dividing~$\hO$, and each of them is either a single vertex, a single or double edge or a cycle.  If $\ell$ is inert, they are each singletons.  If $\ell$ is ramified, they are each of size $2$ with one connecting edge (the isogeny and its dual are identified).  If $\ell$ splits into two classes of order~$2$, we obtain a rim of size two with two connecting edges.  Otherwise, the rims are non-trivial cycles in the oriented $\ell$-isogeny graph, of size equal to the order of $[\mathfrak{l}] \in \ClO$. We summarize the discussion as follows.

\begin{proposition}
\label{prop:sso-r}
Let $\mathcal{O}$ be $\ell$-fundamental.
Let $R_\ell$ be the order of $[\mathfrak{l}] \in \ClO$, for $\mathfrak{l}$ a prime ideal of $\mathcal{O}$ lying above~$\ell$.
The $\mathcal{O}$-cordillera consists of $\#\SSOpr/R_\ell$ volcanoes of rim size $R_\ell$.
\end{proposition}

\subsection{From oriented isogeny graph to isogeny graph}

There is a graph quotient
$\mathcal{G}_K \rightarrow \mathcal{G}$
induced by forgetting the orientation. 

\begin{proposition}
Under this quotient, every component of $\mathcal{G}_K$ (i.e.\ every volcano) covers $\mathcal{G}$. 
\end{proposition}

\begin{proof}
Fix a volcano $\mathcal{V} \subset \mathcal{G}_K$.  Choose a vertex $(E, \iota) \in \mathcal{V}$.  The image $E$ under the quotient map lies on $\mathcal{G}$.  Since both $\mathcal{V}$ and $\mathcal{G}$ are regular of degree $\ell+1$ at every vertex, the image of $\mathcal{V}$ must be all of $\mathcal{G}$.
\end{proof}

As a corollary, every $j$-invariant occurs  on every volcano infinitely many times.  Given $p$, a result of Kaneko~\cite[Theorem 2']{Kaneko} implies that the multiple occurrences of a given $j$-invariant cannot occur too quickly as one descends the oriented $\ell$-isogeny volcano. In fact, there is at most one occurrence in the range $|\Delta|<p$ (here $\Delta$ is the discriminant 
at a certain altitude in the volcano).

 \subsection{Graph statistics and heuristics}
 \label{sec:graph-heuristics}

 In the $\ell$-isogeny graph $\mathcal{G}$, two vertices are at distance $d$ if the shortest path between them in the graph consists of $d$ edges.  The distance between two arbitrary vertices is known to be at most $2 \log p$  \cite[Theorem 1]{Pizer_RamGraphsHecke}.  In fact, for most pairs of vertices, the distance between them is at most $(1+\epsilon)\log p$ (see {\cite[Theorem 1.5]{Sardari}} for a precise statement).

 We will use the following heuristic to justify the runtimes in the paper. One expects the number of occurrences of a $j$-invariant in a volcano to be governed by the number of trees emanating from the rim of the volcano. The heuristic in essence asserts a uniform behaviour within any cordillera.
Specifically, the proportion of occurrences of any $j$-invariant in any individual volcano of a cordillera approaches the overall proportion of trees (or equivalently, of edges descending from a rim). A more precise statement is given in Heuristic~\ref{heur:uniform-volcanoes}. In a follow-up paper  \cite{papertwo}, we discuss this and some related heuristics in more detail.

 \begin{heuristic}
 \label{heur:uniform-volcanoes}
 Let $\mathcal{O}$ be an $\ell$-fundamental quadratic order.
 Consider the finite union $\SSOnotprim$ of $\mathcal{O}'$-cordilleras in the oriented supersingular $\ell$-isogeny graph for all $\mathcal{O'} \supseteq \mathcal{O}$.  Let $d(v)$ denote the distance of a vertex $v$ to the rim of its volcano.  Let $j(v)$ denote its $j$-invariant.  Define:
  \begin{itemize} \itemsep 0pt
     \item $R_{\mathcal{V}}$, the number of edges descending from the rim of the volcano $\mathcal{V} \in \SSOnotprim$;
     \item $R_{\SSOnotprim}$, the sum of the number of edges descending from all rims in $\SSOnotprim$.
\end{itemize}
 Then for any $j$-invariant $j_0$ and any volcano $\mathcal{V} \in \SSOnotprim$, the ratio
 \[
 \frac{
\#\{v \in \mathcal{V}: j(v)=j_0, d(v) \le t \}
}{
\#\{ v \in \SSOnotprim: j(v)=j_0, d(v) \le t \}
}
\] 
approaches $R_\mathcal{V}/R_{\SSOnotprim}$ as $t \rightarrow \infty$.
 \end{heuristic}

 Briefly, one expects this because sufficiently long random walks from any rim vertex will visit all vertices with a uniform distribution \cite[Theorem 1]{GalbraithPetitSilva_IdProtocols}. This observation suffices to prove the case the rims are singletons; other cases should behave similarly.

The following lemma is useful for runtime analyses of our main algorithms (Propositions ~\ref{prop:walkto1728} and ~\ref{prop:quantum-walkto1728}).  It states that sum of the class numbers of all the orders containing $\mathcal{O}$ (approximately the cardinality of the union of the sets $\SSOpr$ involved in $\SSOnotprim$ in Heuristic~\ref{heur:uniform-volcanoes}) is only marginally bigger than just the class number $h_\mathcal{O}$ (approximately the size of the largest $\SSOpr$ in the union).

\begin{lemma} \label{lem:hurwitz}
Let $\mathcal{O}$ be an imaginary quadratic order of 
conductor $f$ in some quadratic field $K$ with class number $h_{\mathcal{O}}$, and put 
\begin{equation} \label{classnosum}  H_{\mathcal{O}} = \sum_{\mathcal{O} \subseteq \mathcal{O}' \subseteq \mathcal{O}_K} 
\, h_{\mathcal{O}'} , 
\end{equation}
where the sum ranges over all the quadratic orders $\mathcal{O}'$ containing $\mathcal{O}$ and $h_{\mathcal{O}'}$ denotes the class number of $\mathcal{O}'$. Then $H_{\mathcal{O}} \le
h_{\mathcal{O}}\,O((\log \log f)^2)$ as $f \rightarrow \infty$.
\end{lemma}
\begin{proof}
Let $\mathcal{O}'$ be a quadratic order of discriminant $D'$ 
containing $\mathcal{O}$ and $f' = [\mathcal{O}':\mathcal{O}]$ the index of $\mathcal{O}$ in $\mathcal{O}'$. Then $f'$ divides $f$. 
By \cite[Corollary~7.28]{Cox_primesoftheform}, we have
\[ h_{\mathcal{O}} =  \frac{f' h_{\mathcal{O}'}}{w'/w} \prod_{\substack{q \mid f' \\ q \text{ prime}}} \left ( 1 - \left ( \frac{D'}{q} \right ) \frac{1}{q} \right )  , \]
where $w, w' \in \{ 2, 4, 6 \}$ are the sizes of the unit groups of $\mathcal{O}$ and $\mathcal{O}'$, respectively. Thus, 
\[ h_{\mathcal{O'}} \leq \frac{w'}{wf'} h_{\mathcal{O}} \, \prod_{\substack{q \mid f' \\ q \text{ prime}}} \left ( 1 - \frac{1}{q} \right )^{-1}  = \frac{w'}{w\varphi(f')} h_{\mathcal{O}} , \]
were $\varphi(\cdot)$ denotes Euler's phi function. It follows that
\[H_\mathcal{O} \leq \sum_{\mathcal{O} \subseteq \mathcal{O}' \subseteq \mathcal{O}_K} \frac{w'}{w \varphi(f')}\,h_\mathcal{O} = 
\frac{w'}{w} \left ( \sum_{f' \mid f} \frac{1}{\varphi(f')} \right ) \,h_\mathcal{O} \ .\]
By \cite[Exercise 3.9 (a)]{Apostol}, we have
\[ \frac{n}{\varphi(n)} < \frac{\pi^2}{6} \frac{\sigma(n)}{n} \]
for all integers $n \ge 3$, where $\sigma(\cdot)$ is the sum of divisors function. From Robin's Theorem \cite{Robin}, we obtain $\sigma(n)/n < c \log \log n$ for all $n \ge 3$ and some constant $c$. Therefore, 
\[ \sum_{3 \le f' \mid f} \frac{1}{\varphi(f')} < \frac{c\pi^2}{6} \sum_{3 \le f' \mid f} \frac{\log \log f'}{f'} < \frac{c\pi^2}{6} (\log \log f) \sum_{f' \mid f} \frac{1}{f'} = \frac{c\pi^2}{6} (\log \log f) \frac{\sigma(f)}{f} < \frac{(c\pi)^2}{6} (\log \log f)^2 \ , \]
and hence $H_{\mathcal{O}} = h_{\mathcal{O}} \, O((\log \log f)^2)$.
\end{proof}

\section{Navigating the $K$-oriented $\ell$-isogeny graph}
\label{sec:navigating}
In this section, we will show how to transform a given  endomorphism of a supersingular elliptic curve into a suitable orientation, and then use it to navigate the oriented $\ell$-isogeny graph.

\subsection{Conjugate orientations and orientations from endomorphisms}
\label{sec:or-graph-conj}

Motivated by our computational goals, we replace the abstract data of an orientation with the more computational data of an endomorphism.  
Given an element $\theta \in \End(E)$ along with its minimal polynomial $m_\theta(x)$, we can infer a unique $\ZZ[\theta]$-orientation only up to conjugation.  Namely, if $\alpha$ is a quadratic irrational root of $m_\theta(x)$, then we define $\iota_\theta(\alpha) = \theta$ and extend to a ring homomorphism.  The conjugate orientation is defined by $\widehat{\iota_\theta}(\alpha) = \widehat{\theta}$, or equivalently, by $\widehat{\iota_\theta}(\overline{\alpha}) = \theta$.  An example in \cite[Section 3.1]{onuki2021} demonstrates a pair of $\Gal(K/\mathbb{Q})$-conjugate $K$-oriented curves which are not isomorphic.
In other words, given $\varphi \in \End(E)$, one may be in either of two locations in the oriented $\ell$-isogeny graph:  $(E,\iota)$ or $(E,\widehat{\iota})$.  However, locally at least, navigating from either location looks the same, in the sense of ascending/descending/horizontal edges and $j$-invariants.

\begin{lemma}
	\label{lemma:conjugate-iota}
	The map $(E,\iota) \mapsto (E,\widehat{\iota})$ is a graph isomorphism and an involution, taking $\SSOpr$ back to itself for each $\mathcal{O}$.  If $\varphi: (E,\iota) \rightarrow (E',\iota')$ is a $K$-oriented $\ell$-isogeny, then $\varphi: (E,\widehat{\iota}) \rightarrow (E', \widehat{\iota'})$ is a $K$-oriented $\ell$-isogeny, and the 
	type (ascending, descending, or horizontal) is the same.
\end{lemma}

\begin{proof}
	The map is clearly a bijection on vertices.
	Observe that the dual of $\widehat{\varphi} \circ \iota \circ \varphi$ is $\widehat{\varphi} \circ \widehat{\iota} \circ \varphi$.  From this,  it follows that the map is a graph isomorphism.  The observation about type 
	follows from the fact that $\SSOpr$ is taken back to itself.
\end{proof}

As consequences of this lemma, for two vertices $(E, \iota)$ and $(E, \widehat{\iota})$, we have the following:
\begin{enumerate}
	\item the $j$-invariant is the same at both vertices;
	\item both vertices are at the same altitude in the volcano;
	\item if the vertices are not at a rim, the ascending isogeny from either vertex is the same;
	\item if the vertices are at the rim, the pair of horizontal isogenies from either vertex is the same;
	\item if we apply any fixed sequence of $\ell$-isogenies from both vertices, the sequence of $j$-invariants appearing on the resulting paths is the same.
\end{enumerate}

For these reasons, it will not, in practice, be necessary for us to know which of two conjugate orientations we are dealing with.  Therefore, we do not make any choice between the two.  In the remainder of the paper, we will not dwell on this distinction and will work with endomorphisms instead of orientations.

\begin{remark}\label{remark:sarah2}
It is a natural question to ask when a subset of the four oriented curves $(E, \iota)$, $(E^{(p)}, \iota^{(p)})$, $(E, \widehat{\iota})$ and $(E^{(p)}, \widehat{\iota}^{(p)})$ coincide.  This question may have importance to a more detailed runtime analysis than we present in this paper, for example.  It is considered in \cite{papertwo}.
\end{remark}

\subsection{$\ell$-primitivity, $\ell$-suitability, and direction finding}

Having associated an endomorphism to an orientation, we can now define the following.

\begin{definition}\label{def:l-primitive_suitable}
Let $\theta\in \End(E)$ be an endomorphism and $\alpha$ the corresponding quadratic element (up to conjugation). Then $\theta$ (as well as $\alpha$) is called \emph{$\ell$-primitive} if the associated orientations $\iota_\theta: \alpha \mapsto \theta$ and $\widehat{\iota_\theta}: \overline{\alpha} \mapsto \theta$ are $\ell$-primitive $\ZZ[\alpha]$-orientations. Moreover, $\theta$ (as well as $\alpha$) is called \emph{$N$-suitable}, for an integer~$N$, if $\alpha$ is of the form $f\omega + kN$ where $k$ is some integer, $f$ is the conductor of $\ZZ[\alpha]$, and $f\omega$ is the generator of $\ZZ[\alpha]$ as described in the conventions of Section \ref{ssec:notations}.

\end{definition}

The purpose of this definition is made clear by the following lemma.

\begin{lemma}
\label{lemma:suitable}
If $\theta \in \End(E)$ is $\ell$-suitable, then $\theta$ is not $\ell$-primitive if and only if $\theta/\ell \in \End(E)$.
\end{lemma}

\begin{proof}
The endomorphism $\theta$ is not $\ell$-primitive 
if and only if there exists a (unique) order $\mathcal{O}' \subseteq \End(E)$ of index $\ell = [\mathcal{O}':\ZZ[\theta]]$.  But this happens if and only if $\theta/\ell \in \End(E)$, since under the $\ell$-suitability hypothesis, $\ZZ[\theta/\ell]$ is precisely this order $\mathcal{O}'$.
\end{proof}

\begin{lemma}\label{lemma:ell_suitable_translation}
Let $\alpha \in O_K \backslash \ZZ$  with trace $t$. Let $f$ be the conductor and $\Delta_K$ the fundamental discriminant of $\ZZ[\alpha]$. Then
\[
\left\{ T \in \ZZ : \alpha + T \mbox{ is $N$-suitable} \right\}
= 
\left\{
\begin{array}{ll}
\frac{f-t}{2} + N \ZZ & \mbox{if } \Delta_K \equiv 1 \pmod 4 \\[5pt]
\frac{-t}{2} + N \ZZ & \mbox{if } \Delta_K \equiv 0 \pmod 4 \\
\end{array} \right. .
\]
\end{lemma}

In our algorithms, we sometimes choose an optimal $T$ in the sense of the following definition.

\begin{definition}\label{def:ell_minimal_translate}
If $\alpha + T$ has the smallest possible non-negative trace amongst all $N$-suitable translates of~$\alpha$, we say that $\alpha + T$ is a \emph{minimal $N$-suitable translate}.
\end{definition}

Knowing just one suitable endomorphism $\theta$ on an elliptic curve $E$, we can determine the type (ascending, descending or horizontal) of isogenies originating at $(E, \iota_\theta)$.

\begin{proposition}	\label{prop:character_mj}
Suppose $\psi: E \rightarrow E'$ is an $\ell$-isogeny and $\theta \in \End(E)$ is an $\ell$-suitable $\ell$-primitive endomorphism.  
Then, with regards to the orientation $\iota_\theta$ induced by $\theta$,
\begin{enumerate}
    \item $\psi$ is ascending if and only if $[\ell]^2\mid \psi \circ \theta \circ \widehat{\psi}$ in $\End(E')$.
    \item $\psi$ is horizontal if and only if $[\ell] \mid \psi \circ \theta \circ \widehat{\psi}$ but $[\ell]^2 \nmid \psi \circ \theta \circ \widehat{\psi}$ in $\End(E')$.
    \item $\psi$ is descending if and only if $[\ell] \nmid \psi \circ \theta \circ \widehat{\psi}$ in $\End(E')$.
\end{enumerate}
\end{proposition}

\begin{proof}
Let $\iota_\theta$ be the orientation on $E$ associated to $\theta$.  
Let $\iota'$ be the induced orientation on $E'$ by $\iota_\theta$ via $\psi$.   Let $\mathcal{O} ,\,\mathcal{O}'\subseteq K$ be two orders such that $\iota_\theta$ is $\mathcal{O}$-primitive and $\iota'$ is $\mathcal{O}'$-primitive. The three cases in the proposition correspond to the cases when $\mathcal{O}\subsetneq\mathcal{O}',\, \mathcal{O}=\mathcal{O}'$ and $\mathcal{O}\supsetneq\mathcal{O}'$, respectively. Therefore, $\psi$ is ascending, horizontal and descending correspondingly. 
\end{proof}

The previous proposition demonstrates that it is enough to check the action of $\psi \circ \theta \circ \widehat{\psi}$ on $E[\ell]$ to determine whether $\psi$ is ascending, horizontal or descending. However, we can also write down the ascending or horizontal endomorphisms directly by analysing the eigenspaces of $\theta$ on $E[\ell]$, as follows.  Note that a version of this for Frobenius is used in CSIDH \cite{CSIDH} to walk horizontally, earlier used in
\cite[Section 3.2]{kieffer2018accelerating} 
and \cite[Section 2.3]{deFeoKiefferSmith_ordinarykeyexch}.

\begin{proposition}
\label{prop:character2_mj}
	Suppose $\theta \in \End(E)$ is $\ell$-suitable and $\ell$-primitive. For each $P\in E[\ell]$ of order $\ell$ let $\psi_P$ denote the degree $\ell$ quotient isogeny induced by $\langle P \rangle$. Let $\lambda_1, \lambda_2 \in \FF_{\ell^2}$ be the eigenvalues of $\theta$ acting on $E[\ell]$. Consider the oriented curve $(E,\iota_\theta)$.
	\begin{enumerate}
	    \item If $\lambda_1,\lambda_2 \in \FF_{\ell^2} \backslash \FF_{\ell}$, then all $\psi_P$'s are descending.
	    \item If $\lambda_1,\lambda_2 \in \FF_\ell$, and
	    \subitem (2a) $\lambda_1 = \lambda_2 = 0$, then there is a unique eigenspace $\langle Q \rangle$ and that gives rise to an ascending isogeny $\psi_Q$; the rest $\psi_P$'s are descending.
	    \subitem (2b) $\lambda_1 = \lambda_2 \neq 0$, then there is a unique eigenspace $\langle Q \rangle$ and that gives rise to a horizontal isogeny $\psi_Q$; the rest $\psi_P$'s are descending.
        \subitem (2c) $\lambda_1 \neq \lambda_2$, then there are two eigenspaces $\langle Q_1\rangle,\,\langle Q_2 \rangle$ that correspond to $\lambda_1,\,\lambda_2$ respectively. The two isogenies $\psi_{Q_1},\,\psi_{Q_2}$ are horizontal, and the rest $\psi_P$'s are descending.
	\end{enumerate}
\end{proposition}

 \begin{proof}
Suppose $\alpha \mapsto \theta$ gives a $K$-orientation of $E$, for $K = \QQ(\alpha)$. Define $\mathcal{O}$ to be $\ZZ[\alpha]$. Let $f(x) \in \ZZ[x]$ denote the minimal polynomial of $\alpha$ over $\QQ$, then $f(x)\amod{\ell}$ is the characteristic polynomial of the action of $\theta$ on $E[\ell]$. From this one can show that Case (2a) appears if and only if $\alpha$ is divisible by $\ell$ as an algebraic integer. Since $\alpha$ is $\ell$-suitable, this is equivalent to $\mathcal{O}$ being non-maximal at $\ell$. Therefore we divide the proof into two cases. In both cases, the statements on the number of descending isogenies follow from the volcano structure as described in Proposition~\ref{prop:local-volcano}.

\textbf{Case I :} $\mathcal{O}$ is not maximal at $\ell$. The eigenspace corresponds to 0 is one-dimensional as otherwise it violates the fact that $\alpha$ is $\ell$-primitive, denote the eigenspace by $\langle Q \rangle$. Then $\langle Q \rangle = E[\mathfrak{l}]$ where $\mathfrak{l}:=(\alpha, \ell)_\mathcal{O}$ is a non-invertible ideal in $\mathcal{O}$. According to~\cite[Proposition 3.5]{onuki2021}, the corresponding isogeny $\psi_Q$ is ascending.

\textbf{Case II :} $\mathcal{O}$ is maximal at $\ell$.
\begin{enumerate}
    \item[$\bullet$] Case (1) is equivalent to $\ell$ being inert in $K$, there are only descending isogenies.
    \item[$\bullet$] Case (2b) is equivalent to $\ell$ ramifying in $K$. In this case, the eigenspace is again one-dimensional, we denote it by $\langle Q \rangle$. Let $\lambda:=\lambda_1=\lambda_2$, then $\langle Q \rangle = E[\mathfrak{l}]$ where $\mathfrak{l}:=(\alpha-\lambda, \ell)_\mathcal{O}$ is an invertible ideal in $\mathcal{O}$. According to~\cite[Proposition 3.5]{onuki2021}, the corresponding isogeny $\psi_Q$ is horizontal.
    \item[$\bullet$] Case (2c) is equivalent to $\ell$ splitting in $K$. In this case, there are two distinct $\FF_\ell$-eigenvalues and two eigenspaces $\langle Q_1 \rangle, \langle Q_2 \rangle$.  For $i = 1$ or $2$, $\langle Q_i \rangle = E[\mathfrak{l}_i]$ where $\mathfrak{l}_i:=(\alpha - \lambda_i, \ell)_\mathcal{O}$ are invertible ideals in $\mathcal{O}$. They give rise to two horizontal isogenies.
\end{enumerate}

\end{proof}

\begin{remark}
\label{rem:matrixrep}
Observe from the proposition that in order to detect which outgoing $\ell$-isogeny at an oriented curve $(E,\theta)$ is ascending or horizontal, we only need to know how $\theta$ acts on $E[\ell]$.  Indeed, we can formalize as follows.  
Let $T_\ell(E)$ have basis $P = (P_n)$, $Q = (Q_n)$, where $P_n, Q_n \in E[\ell^n]$.  	Let $\theta \in \End(E)$ have matrix $M_\theta = \begin{pmatrix} \alpha & \beta \\ \gamma & \delta \end{pmatrix} \in M_2(\ZZ_\ell)$ with respect to that basis.    Let $\phi_a$ have kernel $\langle P_1 - [a]Q_1 \rangle$ for $0 \le a < \ell$ and kernel $\langle Q_1 \rangle$ for $a = \infty$.  We determine a basis $P'$, $Q'$ for the codomain $T_\ell(\phi_a(E))$ as follows:  take any $P'$ satisfying $[\ell]P' = \phi_a(P - [a]Q)$ and take $Q' = \phi_a(Q)$, in the case $a \neq \infty$.  In the case $a = \infty$, we take $P' = \phi_\infty(P)$ and take $Q'$ to be any point satisfying $[\ell]Q' = \phi_\infty(Q)$.
	With the setup as described above, for any $\ell$-isogeny $\phi: E \rightarrow E'$, we have that $\phi = \phi_a$ for some $a \in \{0,1,\ldots, \ell-1, \infty\}$.  Furthermore, for any endomorphism $\theta \in \End(E)$, with respect to bases $P$, $Q$ and $P'$, $Q'$ as described above, $\phi_a \theta \widehat{\phi_a} \in \End(E')$ has $\ell$-adic matrix representation
\[
\begin{pmatrix}
\ell & 0 \\ a & 1 \end{pmatrix}
M_\theta
\begin{pmatrix}
1 & 0 \\ -a & \ell \end{pmatrix}
\in M_2(\ZZ_\ell) 
\quad \mbox{or} \quad
\begin{pmatrix}
1 & 0 \\ 0 & \ell \end{pmatrix}
M_\theta
\begin{pmatrix}
\ell & 0 \\ 0 & 1 \end{pmatrix}
\in M_2(\ZZ_\ell),
\]
depending upon whether $a \neq \infty$ or $a = \infty$ respectively.  Furthermore, as a consequence of Proposition~\ref{prop:character2_mj},
	\begin{enumerate}
		\item Suppose $(E,\theta)$ is not at the rim in the oriented isogeny graph.  Then, the ascending isogeny is given by $\phi_a$ for $a \equiv \alpha/\beta \pmod \ell$ (where $a = \infty$ if $\beta \equiv 0 \pmod \ell$).  
		\item Suppose instead that $(E,\theta)$ is at the rim.  Then, the two horizontal isogenies are given by the two values of $a$ satisfying $\beta a^2 - (\alpha-\delta) a - \gamma \equiv 0 \pmod \ell$, if such exist (if $\beta\equiv 0 \pmod \ell$, the solutions are $a = \infty$ and $a \equiv \gamma/(\delta-\alpha) \pmod \ell$).
	\end{enumerate}
These observations show that one can navigate in the oriented graph, one can perform a Waterhouse transfer (see the next section), divide by $\ell$, and translate by integers, using the matrix representation.  
In fact, the algorithms presented in this paper for finding a path to $j=1728$ can be adapted (using the observations just mentioned) to work for an endomorphism given as an approximate element of $T_\ell(E)$. Note that one loses precision every time one divides by $\ell$, so that one's precision limits the number of steps one can take.  A situation where one may be provided with such an endomorphism is the situation of the cryptographic SIDH problem (the subject of recent attacks \cite{castryck-sidh-attack,maino-attack-sidh}), where an unknown isogeny $\varphi : E \rightarrow \Einit$ to a starting curve gives rise to various endomorphisms $\widehat{\varphi}\theta \varphi$ for $\theta \in \End(\Einit)$ whose action on certain torsion groups is known.
\end{remark}

\section{Representing orientations and endomorphisms}\label{sec:combine_5_7_8_9}
In this section, we will introduce several ways to represent isogenies and endomorphisms and then provide  functionality for each type of representation.
 
 \subsection{Representations and functionality} \label{sec:represent}

We remind the reader that throughout the paper, isogenies and endomorphisms will be assumed separable unless otherwise stated (see Section \ref{ssec:notations}).  In this section, we discuss two types of representations of an endomorphism.  The first is the most basic.

 \begin{definition}
 A \emph{rationally represented} isogeny is an isogeny given by a rational map.  A \emph{rationally represented endomorphism} is an endomorphism which is rationally represented as an isogeny.
 \end{definition}

 We may also represent endomorphisms of large degree (e.g.\ not polynomial in $\log p$) by writing them as a chain of isogenies of manageable degree.
 
 \begin{definition}
 An \emph{isogeny chain} isogeny $\varphi: E_0 \rightarrow E_k$ is an isogeny which is given in the form of a sequence of rationally represented isogenies $( \varphi_i : E_{i-1} \rightarrow E_i )_{i=1}^k$ which compose to $\varphi$, i.e. $\varphi_k \circ \varphi_{k-1} \circ \cdots \circ \varphi_2 \circ \varphi_1 = \varphi$.  
 \end{definition}
 
 Let $B > 0$. Recall that an integer is called \emph{$B$-smooth} (or \emph{$B$-friable}) if its largest prime factor is at most $B$. It is called \emph{$B$-powersmooth} (or \emph{$B$-ultrafriable}) if its largest prime power factor is at most $B$. In order to handle isogeny chain endomorphisms, we will generally \emph{refactor} them, meaning we will replace the chain with another chain representing the same endomorphism, but whose component isogenies have coprime prime power degrees. Moreover, we also fix a powersmooth bound $B$ for the prime power degrees. In Section \ref{sec:chooseb}, we explain our choice of $B$ for the best algorithm runtime.
 
 \begin{definition}
 An isogeny chain whose component isogenies have coprime prime power degrees is called a \emph{prime-power} isogeny chain. Moreover, it is called a \emph{B-powersmooth} prime-power isogeny chain if its component isogenies have coprime prime power degrees  at most $B$.
 \end{definition}

 For isogenies represented in any manner, we will need the following functionality:

  \begin{enumerate}
  \item \textbf{Evaluation at $\ell$-torsion}:  Given $\theta \in \End(E)$, and $P \in E[\ell]$, compute $\theta(P) \in E[\ell]$.  (See Lemma~\ref{lemma:eval-torsion}.)
	 \item \textbf{$\ell$-suitable translation}:  Given $\theta \in \End(E)$, compute $\theta + [t] \in \End(E)$, for some $t \in \ZZ$, so that $\theta +[t]$ is $\ell$-suitable (Definition~\ref{def:l-primitive_suitable}) and again separable.  (See Lemma \ref{lemma:translating} for rational representations and Algorithm~\ref{alg:suitable-chain} for isogeny chains.) Note that for powersmooth prime power isogeny chains, by computing an $\ell$-suitable translation, we always mean that we compute a translate that is a $B$-powersmooth prime power isogeny chain unless otherwise specified. This is exactly what Algorithm~\ref{alg:suitable-chain} does.
     
     \item \textbf{Division by $\ell$}:  Given $\theta \in \End(E)$ such that $\theta = [\ell] \circ \theta'$, compute $\theta' \in \End(E)$.  (See Algorithm~\ref{alg:divisionbyell} for rational representations and Algorithm~\ref{alg:divisionbyell-chain} for isogeny chains.) 
     
     \item \textbf{Waterhouse transfer}:  Given $\theta \in \End(E)$ and $\varphi: E \rightarrow E'$ an $\ell$-isogeny, compute $\varphi \circ \theta \circ \widehat{\varphi} \in \End(E')$.  (See Lemma~\ref{lemma:composing} for rational representations and Algorithm~\ref{alg:refactor-chain} for isogeny chains.)  The terminology is based on \cite{waterhouse1969}.
 \end{enumerate}

 We have endeavoured to write the paper in a modular fashion, so that these two types of representations -- or another unforeseen type of representation, as long as it provides these functionalities -- can be used at will.  In particular, we write our algorithms (Sections~\ref{sec:primitive} onwards) in terms of these functionalities (writing for example $\theta \leftarrow \theta/[\ell]$ for division by $\ell$, to be implemented according to the endomorphism representation chosen).

Although isogeny chain endomorphisms may have large degree, we assume that for any type of endomorphism representation, \textbf{the overall degree, trace and discriminant are polynomially bounded in~$p$}.
 
 As discussed in Section \ref{sec:runtime-lemmata}, it can be rather involved to compute the trace of an endomorphism.  However, the manipulations we perform in our algorithms transform the trace predictably.  Therefore, it is to our advantage to attach the trace data to all endomorphisms under consideration and update it as needed.  For either rationally represented or isogeny chain endomorphisms, our data type will be the following.
 
 \begin{definition}
 A \emph{traced endomorphism} is a tuple of data $(\Eone,\theta,t,n)$ where $\theta \in \End(\Eone)$ is either rationally represented or an isogeny chain, and $t$ and $n$ are the reduced trace and norm (degree) of $\theta$, respectively.
 \end{definition}

\subsection{Functionality for rationally represented endomorphisms}
\label{sec:rationally-rep}
In the case of a rationally represented endomorphism, we can evaluate at $\ell$-torsion directly (Lemma~\ref{lemma:eval-torsion}).  We can translate by an integer by adding the rational maps under the group law (Lemma~\ref{lemma:translating}). 
We can Waterhouse transfer by composing the maps (Lemma~\ref{lemma:composing}).  However, division by $\ell$ requires a dedicated algorithm.  In Section \ref{sec:division_by_ell}, we 
describe 
the algorithm of McMurdy \cite{mcmurdy2014explicit} for exactly this purpose, and analyse its runtime in greater detail. For the completeness of this section, we record here that the runtime of dividing an isogeny $\varphi : E \rightarrow E'$ of supersingular elliptic curves defined over $\FF_{p^2}$
(Algorithm \ref{alg:divisionbyell}) is $O(\deg^2(\varphi) \mulM(p))$.


\subsection{Functionality for isogeny chain endomorphisms} \label{sec:algs-chain} 

An isogeny chain representation of an endomorphism can be more space efficient than its rational representation, and more efficient to compute with. Computing the Waterhouse transfer of an isogeny chain endomorphism is essentially trivial: include the transfer isogenies in the chain. 
To evaluate at $\ell$-torsion, we evaluate the sequence of maps one-by-one (Lemma~\ref{lemma:eval-torsion}); the runtime depends polynomially on the largest degree of their component isogenies.  

In this section, we give algorithms for the more onerous tasks of division-by-$\ell$ and translation by integers.  Their runtimes will depend polynomially on the largest prime power appearing in the degree of the endomorphism, which must therefore be kept small for efficiency.  To address this problem, which arises when translating to something $\ell$-suitable, we use a search step to find a translate of powersmooth degree.  

In order to keep the largest prime power in the degree below a certain bound, we will be interested in $B$-powersmooth prime power isogeny chains. In the last subsection of this section, we balance the runtime considerations by choosing a subexponential powersmoothness bound $B$ for the degree of an isogeny chain endomorphism.  Thus, working with a general such endomorphism is a subexponential endeavour. 

Although our concern is with endomorphisms, both Algorithm~\ref{alg:refactor-chain} and Algorithm~\ref{alg:divisionbyell-chain}  work for isogenies in general.

\subsubsection{Refactoring an isogeny chain}
\label{sec:refactor-chain}

If an endomorphism is not in the prime power isogeny chain form, we can refactor it. To achieve this, one factors the degree, then builds the new chain from scratch kernel-by-kernel, as described in Algorithm \ref{alg:refactor-chain}.  In fact, any endomorphism that can be evaluated at arbitrary points on the curve can be converted to an isogeny chain representation using this algorithm.

\begin{remark} In principle, it is possible to refactor into degrees that are primes as opposed to prime powers. However, this doesn't circumvent the need for powersmoothness (in practice, it would provide some savings, e.g.\ in V\'elu's formulas, but it wouldn't avoid the overall polynomial dependence on the powersmoothness bound). During refactoring, for any prime power factor $q^k$ of the degree, the endomorphism needs to be evaluated on the $q^k$-torsion, which should therefore be defined over a field of manageable size. See \cite[Section 5.2.1]{CharlesGorenLauter} for a nice discussion of this issue in another context. 
\end{remark}

\begin{algorithm}
	\caption{\,Refactoring an isogeny chain }\label{alg:refactor-chain}
    \vspace{.2ex}
    \Input {%
	    A traced endomorphism $(E,\theta,t,n)$ in any form in which it can be evaluated (such as rationally represented or a translation of an isogeny chain), of degree coprime to $p$. }

    \Output {
	    The same traced endomorphism $(E,\theta,t,n) \in \End(E)$ in prime-power isogeny chain form.
        }
    \vspace{.4ex}
    
    $H \leftarrow []$
    \;
    
    $E_0 \leftarrow E$  
    \;
	    Write $n = \prod_{j=0}^u q_j^{k_j}$ by factoring. 

	    \For{$j=0, \ldots, u$ }{ 
	            Compute a basis for $E[q_j^{k_j}]$.
			    \;

			    Compute $G_j = \ker(\theta) \cap E[q_j^{k_j}]$ 
			    by evaluating $\theta$ on $E[q_j^{k_j}]$.\label{line:K}
			    \;

			    Compute a rationally represented isogeny $\varphi_j: E_j \rightarrow E_{j+1}$ given by the kernel $\varphi_{j-1} \circ \ldots \circ \varphi_0(G_j)$, using Velu's formulas. 
			    \label{step:refactor-chain-velu}
			    \;

			    Append $(\varphi_j: E_j \rightarrow E_{j+1})$ to $H$.
			    \;

	    }
                
	    \Return $(E,\theta,t,n)$ where $\theta$ is given by the isogeny chain $H$.

\end{algorithm}

\begin{proposition}
\label{prop:refactor}
	Let $B$ be the largest prime power dividing $\deg \theta$. Then Algorithm~\ref{alg:refactor-chain} is correct and has runtime $O( \log \deg \theta )$ times the maximum of the following three runtimes: $O( B^2 (\log p))$, $O(B^2 (\log B) \mulM(p^{B^2}))$ and the runtime of evaluation of $\theta$ on $O(B)$-torsion. The space requirement of Algorithm~\ref{alg:refactor-chain} is~$O(B^2 \log p)$.  In particular, if $\theta$ is an integer translate of an isogeny chain with $B$-powersmooth degree, then the runtime is $O((\log \deg \theta) B^2 \mulM(p^{B^2}))$. 
\end{proposition}

\begin{proof}
	The \textbf{For} loop builds an isogeny chain for $\theta$.  One can see this by induction:  assuming $\theta = \nu' \circ \nu$ where $\nu := \varphi_{j-1} \circ \ldots \circ \varphi_0$, we have by construction that $\nu(G_j)$ vanishes under $\nu'$.  Hence $\theta$ factors through $\varphi_j \circ \nu$.
	
	To write the factorization of $n$ is at worst $O(B \log^2 B )$ in time (by trial division), but $O(\log n)$ in space. 	For each prime power factor (so at most $\log n$ times), we must do each of the following: \begin{enumerate*}[label=(\roman*)]

	    \item Compute a basis for the torsion subgroup in time and space $O(B^2\log p)$ by Lemma \ref{lem:torsion-basis}.

	    \item Evaluate $\theta$ on the basis 
	\item List the elements of the kernel $G_j$; this involves computing all linear combinations of the basis images and recording those combinations which vanish; and then computing the corresponding linear combinations of the original torsion points, a total of $B^2+B$ linear combinations; by Lemma \ref{lemma:lincomb}, this takes time $O(B^2 (\log B) \mulM(p^{B^2}))$.
	\item Apply V\'elu's formulas in time $O(B \mulM(p^{B^2}))$ by Lemma \ref{lemma:velu}.  Writing down the resulting isogeny takes $O(B)$ coefficients in a subfield of  $\FF_{p^{12}}$ (Lemma~\ref{lemma:isog12}), hence we use $O(B \log p)$ space for each isogeny of the chain.
	
	\end{enumerate*}

	If $\theta$ is a translate of an isogeny chain whose component degrees are bounded by $B$, we can further estimate the time taken to evaluate  $\theta$ on the torsion basis.  This involves one evaluation for each component isogeny (at most $\log n$ such). 
	Each evaluation of a component $\varphi_i$ takes time $O((\deg \varphi_i) \mulM(p^{B^2}))$ by Lemma \ref{lemma:eval-torsion}.  (Evaluation of the integer translation is of smaller runtime by Lemma~\ref{lemma:lincomb}; since the integer is taken modulo the torsion, its size is irrelevant.)
\end{proof}

\begin{remark}
The exponent of the dependence on $B$ can surely be improved here; for example, if $\deg \theta$ is prime, then our bound on the number of linear combinations on which to evaluate $\theta$ is a substantial overestimate. 
\end{remark}

\subsubsection{Division by $\ell$}
\label{sec:divbyell-chain}

In this section, we demonstrate in Algorithm~\ref{alg:divisionbyell-chain} how to divide an isogeny chain endomorphism by $[\ell]$.

\begin{algorithm}
	\caption{\,Dividing-by-$[\ell]$ for an endomorphism given as a prime-power isogeny chain.}\label{alg:divisionbyell-chain}
    \vspace{.2ex}
    \Input {%
	    A traced endomorphism $(E,\theta,t,n)$ in prime-power isogeny chain form
	    , such that $\theta(E[\ell]) = \{ O_E \}$. 
        }
    \Output {
	    A traced endomorphism $(E,\theta',t',n') \in \End(E)$ such that $\theta = [\ell] \circ \theta'$, in prime-power isogeny chain form.
        }
    \vspace{.4ex}

     $i \leftarrow$ the index at which the chain has $\ell$-power degree.
     \;

		Modify the chain for $\theta$ by replacing $\varphi_i$ with $\varphi_i/[\ell]$ using Algorithm~\ref{alg:divisionbyell}.\label{step:divisionbyell-chain-divide} 
		
        $t \leftarrow t/\ell$
                
	    $n \leftarrow n/\ell^2$. 
                
	    \Return $(E,\theta,t,n)$.

\end{algorithm}

\begin{proposition}
\label{prop:divisionbyell-chain}
Let $B$ be an upper bound on the degrees of the prime powers in $\theta$. Then
Algorithm~\ref{alg:divisionbyell-chain} is correct and runs in time 
$O(B^2 \poly(\log p))$. 
\end{proposition}

\begin{proof}
The runtime is negligible except for the call to Algorithm~\ref{alg:divisionbyell}.
By Proposition \ref{prop:comple-alg-divisionbyell}, that algorithm runs in time  $O(\deg^2(\varphi_i)\mulM(p))$ (and we bound $\mulM(p)$ by $\poly(\log p)$ as discussed in Section \ref{ssec:notations}).
\end{proof}

\subsubsection{Finding a $B$-powersmooth $\ell$-suitable translate}

As discussed earlier, we wish to keep the powersmoothness bound $B$ on the degree of an isogeny chain endomorphism low when translating by an integer.  Since our goal is to find $\ell$-suitable endomorphisms, and translation by $\ell$ preserves $\ell$-suitability, we may search amongst nearby translates for one which is $B$-powersmooth for our desired bound $B$. This is done in Algorithm~\ref{alg:suitable-chain}.

\begin{algorithm}
    \caption{\,Computing a $B$-powersmooth $\ell$-suitable translate in prime-power isogeny-chain form.}\label{alg:suitable-chain}
    \vspace{.2ex} 
    \Input {%
	    A traced endomorphism $(E, \theta, t, n)$ in prime-power isogeny chain form, and a powersmoothness bound $B$ (where $B=\infty$ is acceptable). }
    \Output {
        	    A traced endomorphism $(E,\theta',t',n')$ which satisfies $\ZZ[\theta'] = \ZZ[\theta]$ but where $\theta'$ is $\ell$-suitable, and is given as a separable prime-power isogeny chain, with prime powers $\le B$.
        }
    \vspace{.4ex}

      Compute the minimal $\ell$-suitable translate $T$ for $\theta$ (Lemma~\ref{lemma:ell_suitable_translation}).\label{step:suitable-chain-call}

	    Try values $n(b) = n + (T+b\ell)t + (T+b\ell)^2$ for small integers $b$, to find $b$ such that $n(b)$ is $B$-powersmooth and coprime to $p$.\label{step:sieve_for_b}

	  $\theta' \leftarrow$ a refactored prime-power isogeny chain for $\theta + T + b\ell$, using Algorithm~\ref{alg:refactor-chain}.\label{step:ell_suitable_theta_prime}

            $t' \leftarrow t + 2T + 2b\ell$\label{step:ell_suitable_trace} 
                
	    $n' \leftarrow n + (T+b\ell)t + (T+b\ell)^2$.\label{step:ell_suitable_norm} 
                
            \Return{ $(E, \theta',t',n')$ }
             
\end{algorithm}

\begin{proposition}
\label{prop:suitable-chain}
Algorithm~\ref{alg:suitable-chain} is correct, and the runtime is that of Algorithm~\ref{alg:refactor-chain} plus the time taken for Step~\ref{step:sieve_for_b}.
\end{proposition}

\begin{proof}
	The $\ell$-suitability of the output is guaranteed by Lemma~\ref{lemma:ell_suitable_translation}. 
\end{proof}

\subsubsection{Choosing a powersmoothness bound $B$}
\label{sec:chooseb}

In practice, we need to balance the runtimes of the various functionalities of an isogeny chain endomorphism by choosing an appropriate powersmoothness bound $B$.

The number of $B$-smooth and $B$-powersmooth numbers below a bound $X$ is asymptotically the same, provided that $B/\log^2 X \rightarrow \infty$ \cite{MR3356295} (another reference shows they are asymptotically proportional, provided $\log B / (\log \log X) \rightarrow \infty$ \cite[Section 3.1]{smoothbound_CN}).  
In our situation, we expect to handle endomorphisms which may have degree as much as exponential in $\log p$.  Fortunately, we can, at least heuristically, find subexponentially smooth translates in subexponential time \cite[Section 3.1]{smoothbound_CN}.

\begin{heuristic}
\label{heur:translates}
Given integers $n$, $t$, and $T$, values of the function $n(b) = n + (T+b\ell)t + (T+b\ell)^2$, as $b \rightarrow \infty$, are powersmooth with the same probability as randomly chosen integers of the same size.
\end{heuristic}

This is the powersmooth analogue of the heuristic assumption underlying the quadratic sieve; see \cite{CGPT}.

\begin{proposition}
	\label{prop:runtimeB}
	Assume Heuristic \ref{heur:translates}.  Let $\theta \in \End(E)$ have degree $d$ such that $L_d(1/2) > \poly(\log p)$, and assume that its trace $t$ is polynomial in $d$.  Then Algorithm~\ref{alg:suitable-chain} produces an $L_d(1/2)$-powersmooth prime power isogeny chain of total degree $O(d)$.  Furthermore, on $L_d(1/2)$-powersmooth prime power isogeny chains of total degree $O(d)$, the maximum runtime of Algorithm~\ref{alg:refactor-chain},  Algorithm~\ref{alg:divisionbyell-chain} and Algorithm~\ref{alg:suitable-chain} is $L_d(1/2)$, and the output of these algorithms is again an $L_d(1/2)$-powersmooth prime power isogeny chain of total degree $O(d)$.
\end{proposition}

\begin{proof}
    We have seen that all the runtimes in Algorithms~\ref{alg:refactor-chain} through \ref{alg:suitable-chain}
    are polynomial in $B$, $\log d$ ($=O(\log p)$ by the assumptions of Section~\ref{sec:represent}), and $\log p$, with the exception of Step~\ref{step:sieve_for_b} in Algorithm~\ref{alg:suitable-chain}.  Hence, taking $B = L_d(1/2)$, the runtime (except for this step) will be $L_d(1/2)$.  
    
    As far as Step~\ref{step:sieve_for_b}, under Heuristic~\ref{heur:translates}, we can call on \cite[Section 3.1]{smoothbound_CN} (note that the $L$-notation in the reference differs from ours here). According to \cite[Section 3.1]{smoothbound_CN}, the probability that a random integer between $1$ and $d$ is $B$-powersmooth is $1/L_d(1/2)$. Testing values of $b$ between $1$ and $L_d(1/2)$, we do indeed have $n(b) < d$.  Thus, we expect to find a $B$-powersmooth integer, by Heuristic~\ref{heur:translates}. For each $b$-value, to see whether $n(b)$ is $B$-powersmooth, we use na\"ive division in time $O(B\log^2 B)$. Therefore, in total, one will find $L_d(1/2)$-powersmooth integers in time $L_d(1/2)$.
    In Step~\ref{step:ell_suitable_norm}, $n' = n + O( b^2\ell^2)$ (since $|t+2T| \le 1$), so the total degree of the output is $O(d)$.
\end{proof}

A few important notes for the remainder of the paper: \textbf{we will assume
$B = L_{\deg \theta}(1/2)$, where~$\theta$ is the initial input endomorphism, when dealing with isogeny chains, and that whenever we perform an $\ell$-suitable translation on an isogeny chain, we choose a $B$-powersmooth prime power $\ell$-suitable translate.}

\begin{example}[\textbf{Computing an $\ell$-suitable translation} via Algorithm~\ref{alg:suitable-chain}]\label{ex:suitable-chain}
We continue with our running example, computing an $\ell$-suitable translate of a degree $47$ endomorphism $\theta$ on the curve $E_{1728}: y^2 = x^3 - x$ for $\ell = 2$.
Here $\theta$ is given as a rational map:
\[\theta(x, y)=\left(\frac{99 x^{47} + 22 x^{46} +\cdots + 77}{   x^{46} + 40 x^{45} +\cdots + 77}, \frac{113 \Fpi x^{69 }+ 157 \Fpi x^{68} +  \cdots + 63 \Fpi}{   x^{69} + 60 x^{68} \cdots + 158}y \right).
\]
The traced endomorphism is $(E_{1728}, \theta, 0, 47)$. In Step~\ref{step:suitable-chain-call}, we compute the minimal $2$-suitable translate $T$ using Lemma~\ref{lemma:ell_suitable_translation}. From the traced endomorphism, we compute $\Delta_\theta = t^2 - 4n = 0^2 - 4\cdot 47= -188$. This implies that the fundamental discriminant is $-47$ and the conductor is $2$. Therefore, the $2$-suitable translates are of the form $\theta + T$ for $T$ in $1+2\ZZ$, and the minimal $2$-suitable translate is obtained for $T = 1$. 
In Step~\ref{step:sieve_for_b}, we find $b = 0$ produces $n(b) = 2^4\cdot3$, which is $B$-powersmooth for $B = 50$. In Step~\ref{step:ell_suitable_theta_prime}, we factor $\theta + 1$ into an isogeny chain $\theta'=\varphi_{171} \circ \varphi_{1728}$ where $\deg(\varphi_{1728})=16$ and $\deg(\varphi_{171})=3$, which requires a call to Algorithm~\ref{alg:refactor-chain}. Here,
\[\varphi_{1728}(x, y)=\left( \frac{  x^{16} + (156    \Fpi+ 63)     x^{15} + \cdots  + 56    \Fpi+ 36}{  x^{15} + (156    \Fpi+ 63)     x^{14} +  \cdots  +  10    \Fpi+ 71},  \frac{  x^{23} + (55    \Fpi+ 95)     x^{22} + \cdots  + 105    \Fpi+ 82}{  x^{23} + (55    \Fpi+ 95)     x^{22} + \cdots + 26    \Fpi+ 87} y\right) \]
and
\[
\resizebox{\displaywidth}{!}{%
$\displaystyle \varphi_{171}(x, y)=
\left(\frac{x^3 + (102  \Fpi+ 30) x^2 + (31  \Fpi+ 74) x + 10  \Fpi+ 158}{x^2 + (102  \Fpi+ 30) x + 98  \Fpi+ 130}, \frac{x^3 + (153  \Fpi+ 45) x^2 + (3  \Fpi+ 88) x + 102  \Fpi+ 108}{x^3 + (153  \Fpi+ 45) x^2 + (115  \Fpi+ 32) x + 45  \Fpi+ 174} y\right).
$}
\]
The quantities in Steps~\ref{step:ell_suitable_trace} and \ref{step:ell_suitable_norm} can be computed immediately from the values of $t, n, T, b$, and $\ell$, yielding $t' = 2$ and $n' = 48$. The algorithm returns $(E_{1728},\theta', t', n')$.

\end{example}

\subsection{Poly-rep runtime}
\label{sec:polyrep}

In the last two sections, we computed the runtimes of the basic operations for rationally represented and isogeny chain endomorphisms.  We summarize here.

\begin{proposition}
\label{prop:runtime-summary}
Suppose $\theta$ is an isogeny whose trace $t$, norm $n$ and discriminant $\Delta$ are all at most polynomial in $p$.  If $\theta$ is rationally represented, then:
\begin{enumerate}
\item Evaluating at $\ell$-torsion takes time $O(n \poly(\log p) )$ (Lemma~\ref{lemma:eval-torsion}).
    \item Waterhouse transfer by an $\ell$-isogeny takes time $\widetilde{O}( n \poly(\log p) )$ (Lemma~\ref{lemma:composing}).
    \item Dividing by $\ell$ takes time $O(n^2 \poly(\log p) )$ (Proposition~\ref{prop:comple-alg-divisionbyell}).
    \item Computing an $\ell$-suitable translate takes time $\widetilde{O}(\max \{ n, t^2 \} \poly(\log p) )$ (Lemma~\ref{lemma:translating}).
\end{enumerate}
If $\theta$ of degree $d$ is represented as a $B$-powersmooth prime power isogeny chain with $B = L_d(1/2)$ as described in Section~\ref{sec:chooseb}, then, assuming Heuristic~\ref{heur:translates} (see Proposition~\ref{prop:runtimeB}):
\begin{enumerate}
\item Evaluating at $\ell$-torsion takes time $L_d(1/2)$ (Lemma~\ref{lemma:eval-torsion}).
    \item Waterhouse transfer takes time $L_d(1/2)$ (Proposition~\ref{prop:refactor}).
    \item Dividing by $\ell$ takes time $L_d(1/2)$ (Proposition~\ref{prop:divisionbyell-chain}).
    \item Computing a $B$-powersmooth $\ell$-suitable translate takes time $L_d(1/2)$ (Proposition~\ref{prop:suitable-chain}).
\end{enumerate}
\end{proposition}

Of course, in individual situations, these runtimes may be much lower (for example, dividing an isogeny chain by $[\ell]$ may depend only on the power of $\ell$ if no refactoring is necessary).

In the following algorithms, we will need to call all of these operations many times.  It will be convenient to set the following definition. 

\begin{definition}
\label{defn:poly-rep}
	We define the \emph{representation runtime} of a given representation (rationally represented or isogeny chain) to be the maximum runtime of implementing the following operations:  evaluating at $\ell$-torsion, $\ell$-suitable translation, division-by-$\ell$, and Waterhouse transfer by an $\ell$-isogeny. 
	We say that an algorithm has \emph{poly-rep runtime} if its runtime is bounded above by a constant power of $\log p$ times the relevant representation runtime.
\end{definition}

Note that our definition above means that, \textbf{throughout the paper $\poly(\log p) \le $ poly-rep}.

\section{Orientation-finding for $j=1728$}
\label{sec:1728}

For many cryptographic applications, a supersingular elliptic curve with known endomorphism ring is assumed. Most commonly used is the curve with $j=1728$, which is supersingular when $p \equiv 3 \pmod 4$.  For simplicity, this is the curve we will consider here, but our algorithm can be modified to suit other situations (see Section~\ref{sec:otherinitial}).  We will use the model given by $\Einit: y^2 = x^3 - x$, which has endomorphism ring with a $\ZZ$-basis
\[
	\left\langle 1, \mathbf{i}, \frac{\mathbf{i}+\mathbf{k}}{2}, \frac{1+\mathbf{j}}{2} \right\rangle, \quad \mathbf{i}^2 = -1, \ \mathbf{j}^2 = -p, \ \mathbf{k} = \mathbf{i}\mathbf{j}.
\]
In particular, $\mathbf{i}$ is given by $(x,y) \mapsto (-x, \sqrt{-1}\,y)$ and $\mathbf{j}$ is the Frobenius endomorphism\footnote{Note that some papers use the model $y^2 = x^3 + x$, such as \cite[Section 5.1]{EHLMP_reductions}; this model is a quartic twist of ours and under the induced isomorphism of the endomorphism rings, the element which is realized as Frobenius is not preserved.  The model we choose for this paper has $2$-torsion conveniently defined over $\mathbb{F}_p$. See \cite{kate1728note}.} $(x,y) \mapsto (x^p,y^p)$.

Let $\mathcal{O}$ be an imaginary quadratic order of conductor coprime to $\ell$ such that $\mathcal{O}$ embeds into $B_{p,\infty}$.  In this section we give an algorithm for finding an endomorphism $\theta \in \End(\Einit)$, generating a suborder $\mathcal{O}' \subseteq \mathcal{O}$ of discriminant $\ell^{2r} \Delta_\mathcal{O}$ for the minimal possible $r$.  In other words, we wish to find an $\ell$-primitive orientation by a suborder $\mathcal{O}'$ of $\mathcal{O}$.  Or, rephrased again, we want to find an orientation for $\Einit$ placing it as near to the rim as possible in the oriented supersingular isogeny graph cordillera with rims at $\mathcal{O}$.  Alternatively, the algorithm can be run continuously, to return all $\ell$-primitive orientations by suborders of $\mathcal{O}$ in order of increasing $r$.  

The algorithm we provide (Algorithm~\ref{alg:1728orientation}) has similarities to \cite[Integer Representation, Section 3.2]{KLPT}, where the difference arises because we seek a given discriminant instead of a given norm.  In fact, this algorithm applies more generally to curves over $\mathbb{F}_p$ satisfying the hypotheses of \cite[Section 3.2]{KLPT}; in Section~\ref{sec:otherinitial} we make some comments on adapting this algorithm for other initial curves of known endomorphism ring.

An algorithm for a similar problem appears in \cite[Section 4.3]{WesolowskiOrientations}.  However, that algorithm finds the `smallest' quadratic order only:  it requires the discriminant be bounded above by $2 \sqrt{p}-1$.  We wish to find orientations by more general orders.

\subsection{In terms of $1$, $\mathbf{i}$, $\mathbf{j}$, $\mathbf{k}$}

The goal of Algorithm~\ref{alg:1728orientation} is to find such an endomorphism as a linear combination of $1$, $\mathbf{i}$, $\mathbf{j}$, $\mathbf{k}$.

The idea is to solve a norm equation for $\Einit$ under extra conditions that guarantee that the result is an element of the desired quadratic order. 
The algorithm depends on Cornacchia's algorithm, which is discussed in \cite[Section 1.5.2]{Cohen_CourseInCompAlgNT} and \cite[Section 3.1]{FiteSutherlandSatoTateGroups}.  It solves the equation $x^2 + y^2 = n$ when a square root of $-1$ modulo $n$ is known (e.g., such a square root can be found if $n$ is factored).

\begin{algorithm}
    \caption{\,Computing an orientation for the initial curve.} \label{alg:1728orientation}
    \vspace{.2ex}
    \Input {%
    A discriminant $\Delta_\mathcal{O}$ coprime to $p$, which is the discriminant of an $\ell$-fundamental quadratic order $\mathcal{O}$ that embeds into $B_{p,\infty}$.
           }
    \Output {%
             $(\theta, r)$ where $\theta \in \End(\Einit)$ is represented as a linear combination of $1$, $\mathbf{i}$, $\mathbf{j}$, $\mathbf{k}$, with $\ZZ[\theta] = \mathcal{O}' \subseteq \mathcal{O}$ where $[\mathcal{O}:\mathcal{O}'] = \ell^{r}$.  Furthermore, $\theta$ is $\ell$-primitive.  (Here $\Einit$ and $\mathbf{i}$, $\mathbf{j}$ and $\mathbf{k}$ are as in the introduction to this section, namely the specified model of $j=1728$.) 
        }
    \vspace{.4ex}
    
            $r \leftarrow -1$. 
            \;
            \Repeat{$\theta$ is defined}{
                    $r \leftarrow r + 1$.\label{step:alg:1728orientation_incrementr} 
                    \;
                    
                    Find the smallest positive $x$ such that $x^2 \equiv -\Delta_\mathcal{O}\ell^{2r} \pmod p$. \label{step:alg:1728orientation_findx}\; 
                    \While{$x < \sqrt{-\Delta_\mathcal{O} \ell^{2r}}$\label{step:1728orientation-while}}{
                    
                    $D \leftarrow (-\Delta_\mathcal{O} \ell^{2r} - x^2)/p$.\label{step:1728orientation-D}
                    \;
                    
                    \If{ $D \equiv 1 \pmod 4$ }{
                      \If{ $D$ is prime }{

                    Find a square root of $-1$ modulo $D$.
                    \;
                    
                    Use the output of Step 9 and Cornacchia's algorithm to find $y$ and $z$ such that $y^2 + z^2 = D$.
                                    \;
                                    \If{$y$ is odd}{
                                    Swap $y$ and $z$.
                                    }
                    
			    \If{$x$ is even 
			    }{

                                        $\theta \leftarrow \frac{1}{2} + \frac{x}{2}\mathbf{i} + \frac{z}{2}\mathbf{j} + \frac{y}{2}\mathbf{k}$.
                                        \;
                                        
                        }\Else{

                                $\theta \leftarrow \frac{x}{2}\mathbf{i} + \frac{y}{2}\mathbf{j} + \frac{z}{2}\mathbf{k}$.

                        }
                                                                \Break{the \textbf{While} loop}
                    }
                    }
                    $x \leftarrow x + p$
                    
                    }
               
                }

            $c \leftarrow 0$
            \;
                                               \While{ $c < r$}{

                                   Translate $\theta$ to be minimally $\ell$-suitable (Lemma~\ref{lemma:ell_suitable_translation}).

                                \If{ $\theta/\ell \in \End(\Einit)$ }{
                                   $\theta \leftarrow \theta/\ell$.
                                   \;
                                   
                                   $c \leftarrow c + 1$
                                   \;
                                   }\Else{\Break{the \textbf{While} loop}}
}
                                        \Return{ $\theta$ as a linear combination, $r-c$ }
\end{algorithm}

\begin{remark}\label{remark:continuously}
	Algorithm~\ref{alg:1728orientation} can be adapted to run continuously, finding many $K$-orientations of $1728$.  Simply continue the loops instead of breaking them, returning an endomorphism $\theta$ every time one is found.  
	\end{remark}
	
	\begin{remark}
	If one wishes to find \emph{all} possible solutions, remove the requirements that $D$ be a prime congruent to $1 \pmod 4$, although this will adversely affect runtime (Cornacchia's algorithm will require factoring $D$).  Furthermore, we must make sure Cornacchia's algorithm returns \emph{all} solutions, and we must include solutions obtained by changing the sign of~$x$ on each solution already obtained.  We must also be aware that later solutions may fail to be $\ell$-primitive; these can be discarded.  With these adjustments, every orientation of the form specified will eventually be found by the algorithm (not every $\theta$, but every embedding of $\mathcal{O}'$ into $\End(\Einit)$ for all $\mathcal{O}'$) -- see the proof of Proposition~\ref{prop:alg:1728orientation} for relevant details.
\end{remark}

Because of the primality testing step, the algorithm terminates only heuristically.  We separately prove its correctness (if it returns) and then give a heuristic runtime.

In what follows, write $\Delta := \Delta_\mathcal{O}$ for convenience.

\begin{proposition}
	\label{prop:alg:1728orientation}
Any output returned by Algorithm~\ref{alg:1728orientation} is correct.
\end{proposition}

\begin{proof}
We attempt to find an endomorphism $\theta$ for each fixed $r$ increasing from $r=0$.

If the order $\mathcal{O}'$ of index $\ell^r$ in $\mathcal{O}$ has even discriminant (namely $\Delta \ell^{2r}$), then we seek an element of reduced trace zero and reduced norm $-\Delta \ell^{2r}/4$.   Such an element must generate $\mathcal{O}'$, 
and $\mathcal{O}'$ must contain a generator of this form.  Write the element as $\theta = \frac{x}{2}\mathbf{i} + \frac{y}{2}\mathbf{j} + \frac{z}{2}\mathbf{k}$.  Then, simplifying the equation, the norm condition is 
\[
x^2 + py^2 + pz^2 = - \Delta \ell^{2r}.
\]
Any solutions must have $x^2 < \sqrt{-\Delta\ell^{2r}}$, and for a valid $x$, solutions $y$ and $z$ are found by Cornacchia's algorithm applied to 
\[
y^2 + z^2 = (-\Delta\ell^{2r} - x^2)/p.
\]
In order to be contained in $\End(\Einit)$, we require $x \equiv z \pmod{2}$ and $y$ is even.  The variable $r$ is incremented if no solution exists, or if Cornacchia's algorithm is not applied because $D$ is not a prime congruent to $1 \pmod 4$ (in which case we may miss solutions).

If $\Delta \ell^{2r}$ is odd, we instead seek an element of reduced trace $1$ and reduced norm $(-\Delta\ell^{2r} + 1)/4$.  Such an element will again necessarily generate $\mathcal{O}'$, and $\mathcal{O}'$ must contain a generator of this form.  Writing the element as $\theta = \frac{1}{2} + \frac{x}{2} \mathbf{i} + \frac{y}{2}\mathbf{j} + \frac{z}{2}\mathbf{k}$, after slightly simplifying the norm equation, we must solve the same equation as before:
\[
x^2 + py^2 + pz^2 = -\Delta\ell^{2r}.
\]
However, in order to lie in $\End(\Einit)$, such an element must satisfy the conditions that $x \equiv z \pmod 2$ and $y$ is \emph{odd} (note the parity difference).  The rest of this case is as above.

If $\theta$ is not $\ell$-primitive, the algorithm will translate and divide by $\ell$ until it is.
\end{proof}

For the runtime analysis, and the assertion that the algorithm returns an output at all, we need a heuristic similar to that used for torsion-point attacks \cite[Heuristic 1]{deQuehenEtAl_ImprovedTorPt} and the KLPT algorithm \cite[Section 3.2]{KLPT}.

\begin{heuristic}
\label{heur:prime}
Fix integers $D > 0$, $b> 0$, and a prime $p$ coprime to $D b$ that splits in the real quadratic field $\QQ(\sqrt{D})$. 
Ranging through pairs
\[
\left\{ (r,x) : 0 < x, x^2 < D b^{2r}, 0 \le r,  D b^{2r} - x^2 \equiv 0 \pmod p \right\},
\]
consider the value
\[
N(r,x) = 
\frac{D b^{2r} - x^2}{p}.
\]
The probability that $N(r,x)$ is a prime congruent to $1$ modulo $4$ is at least $O(1/(\log D \log N(r,x)))$, where the implied constant is independent of $p$, $D$, and $b$.
\end{heuristic}

We now give a brief justification for this heuristic by passing to the real quadratic field $\QQ(\sqrt{D})$.  Write $D = f^2d$ where $d > 0$ is squarefree.  We have $N(r,x) = q$ if and only if $\pm pq = N(x + fb^r\sqrt{d})$.  Hence we need to estimate the probability, given that $N(x + fb^r \sqrt{d})$ is divisible by $p$, that it is of the form $\pm pq$ for some other prime $q$.  We analyse instead the probability, for $\alpha \in \mathcal{O}_{\QQ(\sqrt{d})}$ (having no assumptions on the form of~$\alpha$), given that $N(\alpha)$ is divisible by $p$, that it is of the form $\pm pq$ for some prime $q$.  Heuristically, we assume that this will be the same probability.

Given that $p$ splits, there is a prime ideal $\mathfrak{p}$ above $p$ in the maximal order of $\QQ(\sqrt{d})$.  Hence $N(\alpha)$ has the form $\pm pq$ if and only if there is a prime ideal $\mathfrak{q}$ of norm $q$ satisfying $\mathfrak{p}\mathfrak{q} = (\alpha)$ (or $\overline{\mathfrak{p}}\mathfrak{q} = (\alpha)$).  If $p \mid N(\alpha)$, then replacing $\mathfrak{p}$ with $\overline{\mathfrak{p}}$ if necessary, this occurs if and only if the integral ideal $(\alpha)\mathfrak{p}^{-1} \in [\mathfrak{p}]^{-1}$ has norm $q$.

Therefore, we estimate the probability that integral elements in $[\mathfrak{p}]^{-1}$ of size $X$ have prime norm.  This is bounded below by the probability that integers of size $X$ have a norm which is a prime represented by the class $[\mathfrak{p}]^{-1}$.  This in turn is bounded below by $\frac{1}{h \log X}$ where $h$ is the class number of $\QQ(\sqrt{d})$.  We apply this estimate with $X = N(r,x)$.

Finally, following the Cohen-Lenstra heuristics for real quadratic fields, it may be reasonable to expect the class number $h_{\QQ(\sqrt{d})}$ to have an expected value bounded by $O(\log d)$, since the number of prime factors of $d$ is around $\log \log d$ (see \cite{Cohen_Lenstra_Heuristics} for a result for prime discriminants and recall that the $2$-part of the class group is controlled by the number of prime factors of $d$).

Heuristic \ref{heur:prime} has been confirmed numerically in some small cases; we will consider this heuristic in more detail in \cite{papertwo}.  The corresponding heuristic, in the case of the KLPT norm equation, has been verified by Wesolowski \cite{Wesolowski_IsogPathandEndoRing}; it would be nice to know if similar methods apply here.

\begin{proposition}
\label{prop:1728returns}
Suppose Heuristic \ref{heur:prime} holds and $\Delta$ is coprime to $p$.  If $|\Delta| < p^2$, then
	Algorithm~\ref{alg:1728orientation} returns a pair $(\theta,r)$ of norm at most  $ p^2 \log^{2 + \epsilon} (p)$   with $r = O(\log p)$ in time $O( \log ^{6 + \epsilon}(p) )$.  If instead $|\Delta| > p^2$, then the algorithm will return a pair $(\theta,r)$ with $r = O(1)$ and norm $O(|\Delta|)$ in time $O( \sqrt{|\Delta|} \log^{4+\epsilon}(\Delta) (\log p) p^{-1})$.
	
	    Running the algorithm continuously, subsequent pairs $(\theta,r)$ should be found in the same runtime, with $r$ expected to increase by $1$, and their norms expected to increase by a constant factor of $\ell^2$ at each subsequent pair.
\end{proposition}

\begin{proof}
Suppose $r \leq u \log_\ell p$, where $u$ is positive (otherwise $r$ is not positive).  Then $\sqrt{-\Delta \ell^{2r}} \le |\Delta|^{1/2}p^{u}$.  Thus, we expect to iterate the \textbf{While} loop at Step \ref{step:1728orientation-while} at most $X(\Delta,u) := \lceil |\Delta|^{1/2}p^{u-1} \rceil+1$ times.
Each time we enter the loop, we obtain a value $D= (-\Delta \ell^{2r} - x^2)/p$ of size $\le p X(\Delta,u)^2$.  The probability that $D$ is prime and $1 \pmod 4$ is heuristically $1/(4\log (p^{1/2} X(\Delta,u)))$ (Heuristic \ref{heur:prime}).  Hence we expect to reach Cornacchia's algorithm once $u$ is large enough such that
\[
X(\Delta,u) \ge 4 \log (p^{1/2} X(\Delta,u)) > 1.
\]
Reaching it will terminate the algorithm.
This is a mild condition, satisfied asymptotically when $X(\Delta,u) \ge (\log p)^{1+\epsilon}$.
In fact, it suffices to take $\sqrt{|\Delta|} p^u \ge p \log ^{1+\epsilon}(p)$, or equivalently, 
\begin{equation}
    \label{eqn:ulogp}
u \log p \ge \log p - \frac{1}{2}\log |\Delta| + (1+\epsilon) \log \log p.
\end{equation}
In particular, $u>1$ is always enough, and if $|\Delta| > p^{2+\epsilon}$, then any positive value for $u$ will suffice.  (An informal explanation of this behaviour: even for a volcano with a trivial rim, distance $(1+\epsilon)\log p$ down its sides is enough to capture all $j$-invariants.  At the same time, if $\Delta$ is large enough that the rim likely captures all $j$-invariants, then we needn't descend the volcano at all.) This shows that the algorithm needs to increase~$r$ at most $O(\log p)$ times before it reaches Cornacchia's algorithm.   

For $|\Delta| \le p^{2+\epsilon}$, the optimal value of $u$ is given by \eqref{eqn:ulogp}. However, since $u$ cannot be negative, when $|\Delta| > p^{2+\epsilon}$, the optimal value of $u$ is $0$.  (Again, informally: the class group will be of size $\approx \sqrt{|\Delta|}>p$, and we will find all $\approx\frac{p}{12}$ supersingular $j$-invariants already on the rim of an isogeny volcano.)

We first determine the overall runtime in terms of $X(\Delta,u)$ and $p$.  The primality test can be run in time $O( \log^{4+\epsilon} D)$ for example, using the Miller-Rabin algorithm
\cite[Section 2]{Schoof_PrimalityTesting}. 
This algorithm is probabilistic, so there is a negligible possibility that Cornacchia's algorithm may fail on false positives.

    Once $D$ is a prime congruent to $1 \amod 4$, we must find a square root of $-1$ with which to run Cornacchia's algorithm.  There is a nice analysis of this exact situation in \cite[Section 3.1]{FiteSutherlandSatoTateGroups}, which concludes that it takes probabilistic time $\widetilde{O}( \log^2 D)$, which is negligible compared to the primality testing.
    
    Thus, for the final runtime, we increment $r$ at most $O(\log p)$ times, running a primality test of cost $O( \log^{4+\epsilon} D)$  at most $O(X(\Delta,u))$ times  for each $r$, before reaching a point where Cornacchia's algorithm is invoked.  Using $D \le p X(\Delta,u)^2$, this gives  runtime $O(X(\Delta,u) (\log p) ( \log p + 2 \log X(\Delta,u))^{4+\epsilon} )$.
    
    In the case of large $|\Delta| > p^{2+\epsilon}$, we put $u=0$ and obtain $X(\Delta,u) = O(\sqrt{|\Delta|}/p)$ and asymptotically $X(\Delta,u) > p^\epsilon$.   This yields a runtime of $O( \sqrt{|\Delta|} \log^{4+\epsilon}(\Delta) (\log p) p^{-1})$. In this case $r=O(1)$ and the norm of the solution found by Cornacchia's algorithm is bounded by $O(|\Delta|)$.
    
    In the case of small $|\Delta| \le p^2$, we optimize $u$ according to \eqref{eqn:ulogp} and obtain $X(\Delta,u) = O(\log^{1+\epsilon} (p))$ and asymptotically $X(\Delta,u) < p$. This gives $O(\log^{6+\epsilon}( p))$.  At the same time, the norm of the solution found is bounded by $|\Delta| \ell^{2r} \le p^2 X(\Delta,u)^2 \le p^2 \log^{2 + 2\epsilon}(p)$.
    
    Once $r$ has reached $O(\log p)$, we expect solutions for each $r$ with high probability.  Therefore, running the algorithm continuously, subsequent solutions should be found in the same runtime as the first, and their sizes should be increasing by an expected constant factor of $\ell^2$ at each subsequent solution.
\end{proof}

\begin{example}[\textbf{Computing an orientation for the initial curve} via  Algorithm~\ref{alg:1728orientation}]\label{ex:1728orientation}

We return to our working example $p = 179$, $\Delta = -47$, $\ell = 2$, and $E_{1728}: y^2 = x^3 - x$.  Note that $\log_\ell(p) \approx 7.48$, so we expect the algorithm to succeed reliably once $r=7$ or $8$, if not earlier.
Beginning with $r = 0$, in Step~\ref{step:alg:1728orientation_findx} we compute the smallest positive $x$ such that $x^2 \equiv 47\amod{179}$, namely $x = 88$. As $x = 88$ exceeds $\sqrt{47}\approx 6.9$, we return to Step~\ref{step:alg:1728orientation_incrementr} and increment $r$ to $r = 1$. This reflects the fact that the curve $E_{1728}$ does not admit a $\mathbb{Q}(\sqrt{-47})$-orientation on the rim. 
Continuing, we find the smallest positive integer $x$ such that $x^2 \equiv 188\amod{179}$, namely $x = 3$. As $x = 3 <\sqrt{47\cdot4}\approx 13.7$, we define $D = (47\cdot4 - 3^2)/179 = 1$ in Step~\ref{step:1728orientation-D}. 
Cornacchia's algorithm returns $1^2 + 0^2 = 1$. 
We obtain the element $\frac{3\mathbf{i} + \mathbf{k}}{2} \in \End(E_{1728})$. This indicates (correctly) that $E_{1728}$ admits an orientation %
at $r = 1$ of the $\mathbb{Q}(\sqrt{-47})$-oriented $2$-isogeny volcano, see the node with $j$-invariant  $1728$ in Figure \ref{fig:path_to_Eint}. 
If we continue to run the algorithm, looking for pairs $(r,\theta)$ for~$r$ up to~$8$, it returns three more pairs:
\[\left(r =7, \theta = \frac{371}{2} \mathbf{i} + 29\mathbf{j} + \frac{13}{2}\mathbf{k}\right),
  \left(r = 8, \theta= \frac{153}{2} \mathbf{i} + 27\mathbf{j} + \frac{119}{2}\mathbf{k}\right),
  \left(r = 8, \theta = \frac{511}{2} \mathbf{i} + 41\mathbf{j} + \frac{95}{2}\mathbf{k}\right).
  \]
\end{example}

We now formalize a heuristic about the behaviour of Algorithm~\ref{alg:1728orientation} needed for what follows.  This is a version of Heuristic~\ref{heur:uniform-volcanoes} specific to the algorithm we use.
\begin{heuristic}
\label{heuristic:uniform-cordillera}
 Let $\mathcal{O}$ be a quadratic order.
Let $\SSOnotprim$ be the finite union of $\mathcal{O}'$-cordilleras where $\mathcal{O}' \supseteq \mathcal{O}$. Write $R_{\SSOnotprim}$ for the sum of the number of descending edges from all rims of $\SSOnotprim$. Fix a volcano $\mathcal{V}$ having $R_\mathcal{V}$ edges descending from its rim.   Then
Algorithm~\ref{alg:1728orientation} running continuously will
\begin{enumerate*}[label=(\roman*)] \item \label{hA} eventually produce solutions on every volcano of $\SSOnotprim$, and \item \label{hB} produce solutions on the fixed volcano $\mathcal{V}$ with probability approaching $R_\mathcal{V}/R_{\SSOnotprim}$.
\end{enumerate*}
\end{heuristic}
If $\SSOnotprim$ has only one volcano, this heuristic is immediate as long as the algorithm produces infinitely many solutions (which happens by Proposition \ref{prop:1728returns}, under heuristic assumptions from Section \ref{sec:graph-heuristics}).  
If Algorithm~\ref{alg:1728orientation} returned \emph{all} orientations of $1728$, then this heuristic would follow directly from Heuristic \ref{heur:uniform-volcanoes}.
The difficulty is that it finds only those solutions where the primality testing step succeeds.  In other words, we cannot rule out the unlikely possibility that the primality condition causes all the orientations of $1728$ to be missed on some individual volcano.  Thus, we seem to require a version of Heuristic \ref{heur:prime} which asserts that the primality is independent of whether the eventual solution is on any fixed volcano of the cordillera.  We consider Heuristic \ref{heuristic:uniform-cordillera} more closely in the companion paper \cite{papertwo}.

\subsection{As an isogeny chain endomorphism}
\label{sec:1728-orient-isogeny-chain}
Since $\mathbf{i}$ and $\mathbf{j}$ are known endomorphisms which can be evaluated at points, any combination of these can also be evaluated at points. Therefore the output of Algorithm~\ref{alg:1728orientation} can be input into Algorithm~\ref{alg:suitable-chain}, and an $\ell$-suitable isogeny chain endomorphism will result.  Thus, in poly-rep time (that is, depending on $B$, the powersmoothness bound), we can obtain the output of Algorithm~\ref{alg:1728orientation} as an isogeny chain endomorphism.

\subsection{Curves other than $j=1728$}
\label{sec:otherinitial}

Algorithm~\ref{alg:1728orientation} 
can be adapted to work for certain curves $\Einit$ other than the curve with $j=1728$.  In particular, if the endomorphism ring $\End(E)$ of a curve $E$ defined over~$\mathbb{F}_p$ is of the form $\mathcal{O} + \mathbf{j}\mathcal{O}$, where $\mathbf{j}$ is the Frobenius endomorphism and $\mathcal{O}$ is a quadratic order, then the adaptation of Algorithm~\ref{alg:1728orientation} is clear, where we use the principal norm form of $\mathcal{O}$ in place of $x^2 + y^2$.  As before, this will reduce to Cornacchia's algorithm.  Instead of primes that are $1 \amod 4$, we seek primes that split in the field and are coprime to the conductor of $\mathcal{O}$; this requires a Legendre symbol computation.  The runtime is essentially unchanged provided that $\Delta_\mathcal{O} < p$ (so Cornacchia's applies; see \cite[Section 3.1]{FiteSutherlandSatoTateGroups}).
This adaptation follows the discussion in \cite[Section 3.2]{KLPT}, which also discusses good choices for $\Einit$ and $\mathcal{O}$.

\section{Supporting algorithms for walking on oriented curves}
\label{sec:algs}
Given a suitable endomorphism, we will present algorithms for walking on an oriented $\ell$-isogeny graph.

\subsection{Computing an $\ell$-primitive endomorphism} \label{sec:primitive}

Recall from Definition~\ref{def:l-primitive_suitable} that an endomorphism $\theta$ is $\ell$-primitive if the associated orientation is $\ell$-primitive.  If $\theta$ is chosen to be $\ell$-suitable, then equivalently, $\theta$ is $\ell$-primitive if and only if it is not divisible by $[\ell]$ in $\End(E)$ (Lemma~\ref{lemma:suitable}).  Therefore, given $\theta$, we can translate it to become $\ell$-suitable and then divide by $[\ell]$ as often as possible to obtain an $\ell$-primitive endomorphism.

\begin{algorithm}
    \caption{\,Computing an $\ell$-primitive endomorphism given an endomorphism.} \label{alg:primitive}
    \vspace{.2ex}
    \Input {%
            A traced endomorphism $(E, \theta, t, n)$ providing the functionality of Section~\ref{sec:represent}.
        }
    \Output {
         A traced endomorphism $(E,\theta',t',n')$ which is $\ell$-primitive, and the $\ell$-valuation of the index $[\ZZ[\theta']:\ZZ[\theta]]$.
        }
    \vspace{.4ex}
          
                \If{$t^2-4n$ is $\ell$-fundamental}{\label{step:alg_primitive_does_ell_divide-init}
                    \Return $(E,\theta,t,n)$ and $0$.
                }
                
		$(E,\theta,t,n) \leftarrow \mbox{ an $\ell$-suitable translate of $(E,\theta,t,n)$}$

                $c \leftarrow 0$

		\While{$ [\ell] \mid \theta$}{\label{step:alg_primitive_while_loop}
			$(E, \theta, t, n) \leftarrow (E, \theta/[\ell], t/\ell, n/\ell^2 )$    

                    $c \leftarrow c + 1$
                    
                    \If{$t^2-4n$ is $\ell$-fundamental }{\label{step:alg_primitive_does_ell_divide}
                        \Return $(E,\theta,t,n)$ and $c$.
                    }
                    
		$(E,\theta,t,n) \leftarrow \mbox{ an $\ell$-suitable translate of $(E,\theta,t,n)$}$

                }
                
                \Return $(E,\theta,t,n)$ and $c$.
           
\end{algorithm}

\begin{proposition}
	Algorithm~\ref{alg:primitive} is correct, and runs in poly-rep time (see Definition \ref{defn:poly-rep}).
\end{proposition}

\begin{proof}
	If $t^2 - 4n$ is $\ell$-fundamental, then the conductor of the quadratic order generated by $\theta$ is not divisible by $\ell$; in this case $\theta$ is already $\ell$-primitive.  In order to check if any order of superindex $\ell$ contains $\ZZ[\theta]$ within $\End(E)$, we first translate $\theta$ to be $\ell$-suitable, and then check whether it is divisible by $[\ell]$ within $\End(E)$.  If it is, we divide it by $\ell$ and repeat.

	For runtime, the algorithm translates to an $\ell$-suitable translate, tests for divisibility by~$\ell$, and divides by $\ell$, at most a polynomial number of times (since we assume that the discriminant of $\ZZ[\theta]$ is bounded by a power of $p$; see Section \ref{sec:represent}).
\end{proof}

\begin{example}[\textbf{Computing an $\ell$-primitive endomorphism} via Algorithm~\ref{alg:primitive}]
\label{sec:primitive-orientation}
  We apply Algorithm~\ref{alg:primitive} to the output of Example~\ref{ex:suitable-chain}, namely $(E_{1728},\theta', t', n')$ where $\theta'=\varphi_{171} \circ \varphi_{1728}, t'=2, n'=48$.  This is not at the rim, but is already $\ell$-suitable.  We find $[2] \nmid \theta'$ by evaluating on $E_{1728}[2]$; hence we return the input unchanged.
\end{example}

\subsection{Rim walking via the class group action}
\label{sec:walkrim}

In the case that an orientation is available, one can walk the rim of the oriented $\ell$-isogeny volcano using the class group action. 
Walking a cycle generated by the class group action was first described in Br\"oker-Charles-Lauter \cite{BrokerCharlesLauter_EvalLargeDeg} in the case of ordinary curves, which carry an orientation by Frobenius.  This was later used in CSIDH \cite{CSIDH}, and it was remarked that it extends to orientations by $\QQ(\sqrt{-np})$ in Chenu-Smith \cite{chenu2021higherdegree}.  In this section we provide a generalization of the same algorithm to arbitrary orientations.  
The algorithm walks the rim from a specified start curve in an arbitrary direction until it encounters a specified end curve.   This path is computed using the action of the class group on the \emph{oriented} curves in the rim of the \emph{oriented} volcano. As such, it requires knowledge of the orientation, so the steps of the algorithm must pull the orientation (i.e.\ the endomorphism) along with them.

More precisely, the ideal we wish to apply to $(E,\theta)$ is given in terms of $\theta$, so that one can use the methods of Br\"oker-Charles-Lauter \cite[Section 3]{BrokerCharlesLauter_EvalLargeDeg} with $\theta$ in place of Frobenius.  One can apply the Waterhouse transfer of~$\theta$, and divide by $\ell$ to carry along $\theta$ in the computation.

The algorithm works by applying the action of $\Cl(\mathcal{O})$ to a rim of elements primitively oriented by a quadratic order $\mathcal{O}$.  In fact, using $\Cl(\mathcal{O})$ works just as well if the rim is primitively oriented by $\mathcal{O}' \supseteq \mathcal{O}$, where $\ell \nmid [ \mathcal{O}' : \mathcal{O} ]$.  This allows us to walk on any rim associated to an $\ell$-fundamental discriminant $\Delta$, without knowing for sure that the orientation is primitive with respect to $\Delta$.  See Proposition \ref{prop:non-prim-action}.

\begin{algorithm}
    \caption{\,Walking along the rim of the oriented supersingular $\ell$-isogeny graph} \label{alg:cycleCGAction}
    \vspace{.2ex}
    \Input {%
	   An $\ell$-primitive traced endomorphism $(E_1, \theta_1, t_1, n_1)$ providing the functionality of Section~\ref{sec:represent}, and a target curve $E_2$.
           }
    \Output {%
	    If $E_1$ and $E_2$ are on the same volcano rim in the oriented isogeny graph for the field $\QQ(\theta)$, with discriminant coprime to $\ell$, the algorithm returns a path of oriented horizontal $\ell$-isogenies from $(E_1,\theta_1,t_1,n_1)$ to a vertex with curve $E_2$. Otherwise returns FAILURE.
        }
    \vspace{.4ex}
	    \If{ $\ell \mid t^2 - 4n$ }{
		    \Return FAILURE.
	    }
            $H \leftarrow []$. \label{step:cycleCGaction_initpath}
            \;
	    \If{ $j(E_1) = j(E_2)$ }{
		    \Return $H$.
	    }

	    Compute $\mathcal{O} \cong \ZZ[\theta_1]$, the quadratic order generated by $\theta_1$ (using trace and norm), together with an explicit isomorphism given in the form of $\alpha_{\theta_1} \in \mathcal{O}$ corresponding to $\theta_1$.\label{step:walking_along_rim_ComputeO}
	    \;
	    \If{ $\ell$ is inert in $\mathcal{O}$ }{
		    \Return FAILURE.
	    }
	    Compute $\tau \in \mathcal{O}$ such that $\mathfrak{l} = ( \ell,\tau )_\mathcal{O}$  is a prime ideal of $\mathcal{O}$ above $\ell$. \label{step:cycleCGaction_primeideal}
	    \;
	    Compute $a, b \in \ZZ$ so that $\tau = a + b\alpha_{\theta_1}$.
            \;
            
		    $(E,\theta,t,n) \leftarrow (E_1,\theta_1,t_1,n_1)$. \label{step:cycleCGaction-init}
		    \;

		    \Repeat{$(j(E),\theta,t,n) = (j(E_1),\theta_1,t_1,n_1)$ or $j(E) = j(E_2)$} 
                {
                Compute $E[\ell]$. \label{step:cycleCGaction_elltorsion}
                \;
		Compute $E[\mathfrak{l}] \leftarrow E[\ell] \cap \ker(a+b\theta)$ by evaluating $a+b\theta$ on $E[\ell]$. \label{step:cycleCGaction_kernel}
                \;
                Use V\'elu's algorithm to compute the $\ell$-isogeny $\nu: E \rightarrow E'$ with kernel $E[\mathfrak{l}]$. \label{step:cycleCGaction_isogeny}
                \;
		$(E,\theta,t,n) \leftarrow (E', \nu \circ \theta \circ \widehat{\nu}, t\ell, n\ell^2)$. \label{step:cycleCGaction_orientation} 
		\;
		$(E, \theta, t, n) \leftarrow (E, \theta/[\ell], t/\ell, n/\ell^2)$.
                \;
		Append $(\nu, (E,\theta,t,n))$ to $H$. \label{step:cycleCGaction_append}
                \;
                }
          
		\If{ $j(E) = j(E_2)$}{
		    \Return $H$
	    } \Else {
	    	\Return FAILURE
	}
\end{algorithm}

Calling Algorithm~\ref{alg:cycleCGAction} without lines 4 and 5 on identical input curves (i.e.\ $(E_1, \iota_1) = (E_2, \iota_2)$
yields the entire rim of the $\ell$-oriented isogeny graph.

\begin{proposition}
\label{prop:cycleCGAction}
	Algorithm~\ref{alg:cycleCGAction} is correct.  
	Each step of the rim walk has poly-rep runtime.  The number of steps is bounded $O(\hO)$. 
	Furthermore, if $\theta$ is in prime power isogeny chain form with any powersmoothness bound $B$,
	then each step of the rim walk has runtime polynomial in $B$.
\end{proposition}

\begin{proof}
	If $\ell \mid t^2 - 4n$, then either we are not at the rim, or the field discriminant is not coprime to $\ell$.  If $j(E_1)=j(E_2)$, we have already completed our task.  Assuming neither of those cases, we compute the quadratic order $\mathcal{O}$ generated by $\theta$ using its minimal polynomial, and associate an element $\alpha_\theta$ to $\theta$.   
	The volcano rim in question  is contained in $\operatorname{SS}_{\mathcal{O}'}$ for some $\mathcal{O}' \supseteq \mathcal{O}$, where the relative index $f = [ \mathcal{O}' : \mathcal{O} ]$ is coprime to $\ell$ (by $\ell$-primitivity).
	If $\ell$ is inert in $\mathcal{O}$, then it is also inert in $\mathcal{O}'$. Hence the rim of the associated volcano is trivial; since $j(E_1) \neq j(E_2)$, this indicates there is no valid path to be found.  
	Otherwise, $\ell$ is split or ramified in $\mathcal{O}$, so we factor it and compute $a$ and $b$ and $\tau$ as in the algorithm.  Namely, we have the factorization $\ell\mathcal{O} = (\ell, \tau)_\mathcal{O}(\ell, \overline{\tau})_\mathcal{O}$ in $\mathcal{O}$.  Then $\ell\mathcal{O}' = (\ell, \tau)_{\mathcal{O}'}(\ell, \overline{\tau})_{\mathcal{O}'}$ in $\mathcal{O}'$. 
	Therefore, the isogeny computed is the action of the ideal $\mathfrak{l}$ lying above $\ell$ in $\mathcal{O}'$ on $\SS_{\mathcal{O}'}$ as desired, which is thus a horizontal isogeny.
	The \textbf{repeat} clause walks the rim step by step. 
	
	We stop if we meet $E_2$ or return to our (oriented) starting point.  The latter occurs only if we have walked the entire rim, which means $E_2$ was not on that rim.

	For runtime, all individual steps are polynomial, except for calls to evaluate at $\ell$-torsion points, Waterhouse transfer and divide by $\ell$.  The number of repeats is equal to the path length from $E_1$ to $E_2$ along the rim.  The size of the rim is $O(\hO)$ (Section~\ref{sec:volcano-structure}).
	
	For the final statement of the proposition, note that no $\ell$-suitable translation is needed in the algorithm.  In fact, the norm of the endomorphism remains constant as one walks the rim.
\end{proof}


\begin{example}[\textbf{Walking along the rim of the oriented supersingular $\ell$-isogeny graph} via Algorithm \ref{alg:cycleCGAction}]
\label{ex:cycleCGAction}
As before, we have $K = \mathbb{Q}(\sqrt{-47})$. We use Algorithm \ref{alg:cycleCGAction} on input $\ell = 2$, $(E_{22}, \theta_{22},t_{22},n_{22})$ and target curve $E_{22}$ to compute the entire rim of the oriented 2-isogeny volcano for purposes of demonstration. The endomorphism $\theta_{22}$ is a primitive $\mathcal{O}_K$-orientation, so the curve $E_{22}$ lies on the rim of an $\mathcal{O}_K$-oriented isogeny volcano. Step \ref{step:cycleCGaction_primeideal} computes the prime ideal $\mathfrak{l} = (2, \omega)_{\mathcal{O}_K}$. In Step~\ref{step:cycleCGaction_elltorsion}, we compute $E_{22}[2] = \{\mathcal{O}_{E_{22}},(2,0),(156\Fpi+178,0), (23\Fpi + 178,0)\}$. We obtain $E_{22}[\mathfrak{l}] = \langle (156\Fpi+178,0) \rangle$ in Step~\ref{step:cycleCGaction_kernel}. Velu's formulas in Step~\ref{step:cycleCGaction_isogeny} compute the isogeny $\varphi_{22} : E_{22} \rightarrow E_{99\Fpi+ 107}$. The codomain of $\varphi_{22}$ is $ E_{99\Fpi+107} : y^2 = x^3 + (26\Fpi +88)x + (141\Fpi+104)$.
In Step~\ref{step:cycleCGaction_orientation}, we compute the traced endomorphism $(E_{99\Fpi+107},\theta_{99\Fpi + 107}, t_{99\Fpi+107}, n_{99\Fpi+107})$ with $\theta_{99\Fpi + 107}:= \frac{1}{2} \, \varphi_{22} \circ \theta_{22} \circ \hat{\varphi}_{22}$, an endomorphism of degree 12. Step~\ref{step:cycleCGaction_append} appends the isogeny $\varphi_{22}$ and the traced endomorphism $(E_{99\Fpi+107},\theta_{99\Fpi + 107}, t_{99\Fpi+107}, n_{99\Fpi+107})$ to $H$.

In the next rim step, starting with $(E_{99\Fpi+107},\theta_{99\Fpi + 107}, t_{99\Fpi+107}, n_{99\Fpi+107})$, we compute the isogeny $\varphi_{99\Fpi+107} : E_{99\Fpi+107} \rightarrow E_{5\Fpi+109}$. The isogeny $\varphi_{99\Fpi+107}$ and traced endomorphism  $(E_{5\Fpi+109}, \theta_{5\Fpi+109}, t_{5\Fpi + 109}, n_{5\Fpi + 109})$ are appended to $H$ in Step~\ref{step:cycleCGaction_append}. 

In the next rim step, we find the isogeny $\varphi_{5\Fpi+109} : E_{5\Fpi+109} \rightarrow E_{174\Fpi+109}$ and corresponding traced endomorphism $(E_{174\Fpi+109}, \theta_{174\Fpi+109}, t_{174\Fpi+109},n_{174\Fpi+109})$
with $\theta_{174\Fpi+109} = \frac{1}{2} (\varphi_{5\Fpi+109}) \circ \theta_{5\Fpi+109} \circ \hat{\varphi}_{5\Fpi+109}$.

A fourth step along the rim produces the isogeny $\varphi_{174\Fpi+109} : E_{174\Fpi+109} \rightarrow E_{80\Fpi+107}$ and traced endomorphism $(E_{80\Fpi+107}, \theta_{80\Fpi+107}, t_{80\Fpi+107}, n_{80\Fpi+107})$.

The final step along the rim produces the isogeny $\varphi_{80\Fpi+107}: E_{80\Fpi + 107} \rightarrow E_{22}'$ with codomain $E_{22}' : y^2 = (125\Fpi+98)x + (84\Fpi+152)$ and induced traced endomorphism $(E_{22}', \theta_{22}', t_{22}', n_{22}')$. The codomain $E_{22}'$ is isomorphic to $E_{22}$ via an isomorphism $\rho$, and we use the same isomorphism $\rho$ to confirm that $E_{22}'$ and $E_{22}$ are in fact isomorphic as oriented curves by computing $\theta_{22}' = \rho \circ \theta_{22} \circ \rho^{-1}$. 

Algorithm \ref{alg:cycleCGAction} terminates and returns the rim cycle 
\[E_{22}\xrightarrow[]{\hspace*{0.2cm} \varphi_{22} \hspace*{0.2cm}} E_{99\Fpi + 107} \xrightarrow[]{\hspace*{0.2cm} \varphi_{99\Fpi + 107} \hspace*{0.2cm}} E_{5\Fpi + 109} \xrightarrow[]{\hspace*{0.2cm} \varphi_{5\Fpi + 109} \hspace*{0.2cm}} E_{174\Fpi + 109} \xrightarrow[]{\hspace*{0.2cm} \varphi_{174\Fpi + 109} \hspace*{0.2cm}}
 E_{80\Fpi+107} 
 \xrightarrow[]{\hspace*{0.2cm} \varphi_{80\Fpi + 107} \hspace*{0.2cm}}
E_{22}'\cong E_{22}\]
of length 5 (see the green rim cycle in Figure \ref{fig:path_to_Eint}). Indeed, $K$ has class number 5, and the ideal class of $\mathfrak{l}$ generates the class group of $K$.
\end{example}

\subsection{Ascending to the rim using an orientation}

The other major component of navigating the supersingular $\ell$-isogeny graph using an orientation is to walk to the rim.  
We can use Proposition \ref{prop:character2_mj} to determine the ascending direction and walk up.  This is described in Algorithm~\ref{alg:walktorim}.  
	The number of steps to the rim is expected to be $\log(p)$ in general; see Section~\ref{sec:graph-heuristics}.

\begin{algorithm}
    \caption{\,Walking to the rim of the oriented $\ell$-isogeny graph.} \label{alg:walktorim}
    \vspace{.2ex}
    \Input {%
            An $\ell$-primitive traced endomorphism $(E, \theta, t, n)$ providing the functionality of Section~\ref{sec:represent}.
        }
	\Output {The shortest path from $(E, \theta,t,n)$ to the rim of the oriented $\ell$-isogeny volcano upon which $(E,\theta,t,n)$ lies.}
    \vspace{.4ex}
            $H \leftarrow []$.  \label{step:walktorim_cycle-init}
  \; 
                        
        $k\leftarrow \left\lfloor \frac{\nu_\ell(t^2-4n)}{2} \right\rfloor.$ 
        \label{step:walktorim-k}
        
        \If{ $\ell=2$ and $(t^2-4n)/2^{2k} \not\equiv 1 \pmod 4$ }{ $k \leftarrow k-1$}
        
        \For{$j = 1,\dots,k$}{
        
        Compute $E[\ell]$.
        
        $(E, \theta, t, n) \leftarrow $ an $\ell$-suitable translate of $(E, \theta, t, n)$. 
        
        Compute a generator $P$ for $E[\ell] \cap \ker(\theta)$.
        
        Use V\'elu's algorithm to compute the $\ell$-isogeny $\nu : E \rightarrow E'$ with kernel $\langle P \rangle$. \label{step:walktorim_Velu}
       
	$(E,\theta,t,n) \leftarrow (E',\nu \circ \theta \circ \hat{\nu}, t\ell, n\ell^2)$
	\;

	$(E,\theta,t,n) \leftarrow (E, \theta/[\ell^2], t/\ell^2, n/\ell^4)$ \label{step:walktorim_primitive}
        
	Append $(\nu, (E,\theta,t,n))$ to $H$.
    
    }
    
    \Return $H$
\end{algorithm}

\begin{proposition}
\label{prop:walktorim-chain}
	Algorithm~\ref{alg:walktorim} is correct and has poly-rep runtime times the distance to the rim.
\end{proposition}

\begin{proof}
	The number of steps to the rim is given by the number of times $\ell^2$ divides the discriminant of $\theta$ (we assume $\theta$ is $\ell$-primitive); this is $k$ in Step \ref{step:walktorim-k}.  We translate $\theta$ to be $\ell$-suitable, which implies that $\nu \circ \theta \circ \widehat{\nu}$ can be divided by $[\ell]$ twice when $\nu$ is ascending.  Since there is no horizontal direction (by the choice of $k$ in Step~\ref{step:walktorim-k}), there exists a non-trivial $P \in E[\ell] \cap \ker(\theta)$.  This gives the ascending isogeny by Proposition \ref{prop:character2_mj}.  Once we have found the ascending isogeny, we divide the Waterhouse transfer of $\theta$ by $[\ell]^2$ (Step \ref{step:walktorim_primitive}),  and the result is $\ell$-primitive, in preparation for the next loop iteration.
For each iteration of the \textbf{For} loop, the work is clearly poly-rep.
\end{proof}

\begin{example}[\textbf{Walking to the rim of the oriented $\ell$-isogeny graph for rationally represented endomorphisms} via Algorithm \ref{alg:walktorim} ]\label{ex:walktorim_rationalrep}
We apply Algorithm~\ref{alg:walktorim} to the output of Step~\ref{step:walkto1728-up1} of Example~\ref{ex:pathto1728}, namely $E_{120}$ and $\theta_{120}$ having $t_{120}=0$, $n_{120}=188$. We find that we expect to take two steps to the rim. Since $\theta_{120}$ is already $2$-suitable,  we evaluate it on $E_{120}[2]$ and obtain the kernel $\langle (121i + 4, 0) \rangle$ for the ascending isogeny.  The codomain is $E_{171}$.  Computing the Waterhouse transfer and dividing by $[2]$ twice, we obtain an endomorphism~$\theta'$ which is not $2$-suitable, but  Lemma~\ref{lemma:ell_suitable_translation} shows that $\theta_{171} := \theta' + [1]$ is $2$-suitable.  The second ascending step is similar; this has kernel $\langle(121i + 131, 0) \rangle$ and codomain $E_{5i + 109}$. The two ascending steps are in blue in Figure \ref{fig:path_to_Eint}.
\end{example}

\begin{example}[\textbf{Walking to the rim of the oriented $\ell$-isogeny graph for isogeny chain endomorphisms} via Algorithm \ref{alg:walktorim} ]\label{ex:walktorim_chain}
We begin with input  $(E_{1728}, \varphi_{171}\circ\varphi_{1728},2,48)$, from Step~\ref{step:walkfrominit} of Example~\ref{ex:pathto1728}.  This will require one step to the rim and is already $[2]$-suitable.  Evaluating on $E_{1728}[2]$, we obtain a kernel of $\langle (178, 0) \rangle$ for the ascending isogeny; the codomain is $E_{22}$.  Waterhouse transfer yields an isogeny-chain which is not prime-power refactored, namely $\varphi'_{1728} \circ \varphi_{171} \circ \varphi_{1728} \circ \widehat{\varphi'}_{1728}$ having component degrees $2$, $3$, $16$, $2$, respectively.  We could apply Algorithm~\ref{alg:refactor-chain}, but we proceed in a slightly more expedient manner.  We rewrite $\varphi'_{1728} \circ \varphi_{171}$, having degrees $2$ and $3$, respectively, in a form having degrees $3$ and $2$, respectively.  Thus, we evaluate $\varphi'_{1728} \circ \varphi_{171}$ on the $2$-torsion to obtain the kernel $\langle (29i + 50, 0) \rangle$ determining $\varphi'_{171} : E_{171} \rightarrow E_{174i + 109}$.  Then we apply $\varphi'_{171}$ to the generator of $\ker(\varphi'_{1728} \circ \varphi_{171}) \cap E_{171}[3] = \langle (128\Fpi + 164, 28\Fpi + 90) \rangle$ to obtain a kernel for which V\'elu gives $\varphi_{174i + 109} : E_{174i + 109} \rightarrow E_{22}$.  We obtain the refactored isogeny chain $\varphi_{174\Fpi+109}\circ \varphi_{171}'\circ \varphi_{1728}\circ\widehat{\varphi'}_{1728}$.  We can then divide the 2-power degree component $\varphi_{171}'\circ \varphi_{1728}\circ\widehat{\varphi'}_{1728}$ by $[2]$ twice and let $\varphi_{22}' := \varphi_{171}'\circ \varphi_{1728}\circ\widehat{\varphi'}_{1728}/[4]$.  Replacing this in our isogeny chain above, we now have an isogeny that gives the one step up to the rim (see the red step in Figure \ref{fig:path_to_Eint}):
\[(E_{1728},\varphi_{171}\circ\varphi_{1728},2,48) \xrightarrow[]{\hspace*{0.2cm} \varphi_{1728}' \hspace*{0.2cm}} (E_{22}, \varphi_{174\Fpi + 109}\circ\varphi_{22}', 1, 12).\]
 \end{example}

\subsection{Ascending and walking the rim using the endomorphism ring}
\label{sec:walk-end}

When we find an orientation of $j=1728$, we have more information than just the specified orientation:  we also know the endomorphism ring.  This extra information allows us to navigate the oriented graph in polynomial time using known algorithms.

Specifically, with Algorithm~\ref{alg:walk-end} given here, we can walk up the volcano and traverse the rim (being careful not to back-track by comparing to our previous steps), where each step is polynomial in $\log p$ and the length of the representation of $\theta$.
To get started, we use $\Einit$ as the curve defining $B_{p,\infty}$ as in \cite{Wesolowski_IsogPathandEndoRing}, and take the path $P$ to be the trivial path.

\begin{algorithm}
    \caption{\,Extending a path from $\Einit$ by an ascending or horizontal step.
    } \label{alg:walk-end}
    \vspace{.2ex}
    \Input {%
    A fixed endomorphism $\theta \in \End(\Einit)$.
    An elliptic curve $E$ and path $P$ from $\Einit$ to $E$, with no descending steps, and $s$ equal to the number of ascending steps in the path $P$.
        }
	\Output {For each of the available horizontal or ascending steps $E \rightarrow E'$ (with regards to the orientation induced by $\theta$), returns the data $(E', P', s')$, where $P'$ is the path obtained from $P$ by extending it by the extra step, and $s'$ is the number of ascending steps in the path $P'$.}
    \vspace{.4ex}

   $H \leftarrow []$
     \;
        
       \For{each $\ell$-isogeny $\nu : E \rightarrow E'$ departing $E$}{
       
           $P' \leftarrow$ the path formed by appending  $\nu$ to $P$.
           
           $(\varphi: \Einit \rightarrow E') \leftarrow$ the isogeny associated to the path $P'$. 
        
        Compute a $\mathbb{Z}$-basis of the maximal quaternion order $\mathfrak{O}$ of $E'$ and connecting ideal $I$ between $\Einit$ and $E'$ using \cite[Algorithm 3]{Wesolowski_IsogPathandEndoRing} from the path $P'$.
        
        Compute $\End(E')$ together with an isomorphism $\Psi: \End(E') \rightarrow \mathfrak{O}$, using \cite[Algorithm 6]{Wesolowski_IsogPathandEndoRing}.
        
        $\beta \leftarrow \Psi(\varphi \circ \theta \circ \widehat{\varphi})$  (The ability to evaluate $\Psi( \varphi \circ \theta \circ \widehat{\varphi})$ for $\theta \in \End(\Einit)$ is also obtained when \cite[Algorithm 6]{Wesolowski_IsogPathandEndoRing} is performed in the last step.) 
        \;
        $\beta \leftarrow \beta+T$ where $T \in \ZZ$ is chosen so that $\beta+T$ is the minimal $\ell^s$-suitable translate of $\varphi \circ \theta \circ \widehat{\varphi}$ using Lemma~\ref{lemma:ell_suitable_translation}. 
     
        \If{ $\beta/\ell^{s+1} \in \mathfrak{O}$         }
        { 

         $s' \leftarrow s$
         \;
        \If{$\beta/\ell^{s+2} \in \mathfrak{O}$  }{
         $s' \leftarrow s'+1$ 
        }
         
        Append $(E', P', s')$ to $H$.
        }
        }
 \Return  H.
\end{algorithm}

\begin{proposition}
\label{prop:walk-end}
Under GRH, Algorithm~\ref{alg:walk-end} is correct and runs in expected polynomial time in the following quantities: $\log p$, the size of the representation of $\theta$, and the length of the path $P$.
\end{proposition}

\begin{proof}
Each of the cited algorithms runs in the time specified under GRH.  We determine which steps are ascending or horizontal by testing whether $\beta/\ell^{s+1}, \beta/\ell^{s+2} \in \mathfrak{O}$, by Proposition~\ref{prop:character_mj}.  Since $\beta$ is represented as a linear combination of a basis of $\End(E')$, this involves dividing the coefficients, which is polynomial time. 
\end{proof}

\section{Classical path-finding to $j=1728$}\label{sec:path-1728}

We now present an algorithm which, given a suitable endomorphism on a curve in the supersingular graph, will find a path to the initial curve, under heuristic assumptions.  An illustration of the method is given in Figure~\ref{fig:path_to_Eint}:  we walk from the initial endomorphism to its rim; find an orientation of $E_{1728}$ and walk from that orientation of $E_{1728}$ to its rim; and hope to collide on the same rim.

If one wishes to adapt this algorithm to find a path to a more general initial curve, one would need a replacement to Algorithm~\ref{alg:1728orientation} that works for that initial curve (see Section~\ref{sec:otherinitial} for a discussion of how this may be done).  For this reason, we restrict ourselves to considering the $j=1728$ curve.

\begin{algorithm}
    \caption{\,Finding a path to $E_{1728}$.} \label{alg:pathto1728}
    \vspace{.2ex}
    \Input {%
           A traced endomorphism $(\Eone,\theta,t,n)$ providing the functionality of Section~\ref{sec:represent}, where the discriminant of $\theta$ is coprime to $p$.
           }
    \Output {%
             A path in the $\ell$-isogeny graph between $\Eone$ and $E_{1728}$.
                    }
    \vspace{.4ex}
	$(\Eone,\theta,t,n) \leftarrow (\Eone,\theta/[\ell^k], t/\ell^k, n/\ell^{2k})$ which is $\ell$-primitive, using Algorithm \ref{alg:primitive}.   \label{step:pathhome_make_primitive}
        \;
	$\Delta_\theta \leftarrow t^2 - 4n$.
    \label{step:disc1} 
    \;
    $\Delta \leftarrow$ the $\ell$-fundamental part of $\Delta_\theta$.
    \label{step:l-free}
   \;

    Call Algorithm \ref{alg:walktorim} on input $(\Eone, \theta, t, n)$ to produce an ascending path $H_2$ from $(\Eone, \theta,t,n)$ to $(E_1, \theta_1, t_1, n_1)$ on the rim, i.e.\ where $\ZZ[\theta_1] \subseteq \End(E_1)$ is $\ell$-fundamental.  \label{step:walkto1728-up1}
        \label{step:pathhome_ascending_path}
        \;
            Call Algorithm \ref{alg:cycleCGAction} on input $(E_1,\theta_1,t_1,n_1)$ to walk the rim until we encounter $E_1$ again, storing the $j$-invariants encountered as a list $L$.\label{step:getL}
    \;
    \Repeat{ $E_0 \in L$ or $E_0^{(p)} \in L$}{
    
            Call Algorithm \ref{alg:1728orientation} on input $\Delta$, to obtain a new solution $\theta_{1728} = a + b\mathbf{i} + c\mathbf{j} + d\mathbf{k}$.  (Algorithm~\ref{alg:1728orientation} can be suspended and then resumed to find subsequent solutions; see Remark~\ref{remark:continuously})\label{step:1728orient}
            \;
            
            Using the methods of Section~\ref{sec:walk-end}, produce an ascending path $H_1$ from $E_{1728}$ with endomorphism $\theta_{1728}$  up to the rim, i.e.\ to a traced endomorphism $(E_0, \theta_0, t_0, n_0)$ having $\ell$-fundamental order $\ZZ[\theta_0]$ contained in $\End(E_0)$.\label{step:walkfrominit} 
    \; 
        
}
Compute $H_{rim}$, the path from $E_1$ to $E_0$ or $E_0^{(p)}$, using $L$.
\;
    \If{ $H_{rim}$ joins $E_1$ to $E_0$ }{
	    $H \leftarrow H_2 H_{rim}^{-1} H_1^{-1}$, a path from $E_{1728}$ to $\Eone$.
    }\Else{
    From $H_1$, compute the conjugate path $H_1^{(p)}$ from $E_{1728}$ to $E_0^{(p)}$.
    \;
    $H \leftarrow H_2 H_{rim}^{-1} (H_1^{(p)})^{-1}$, a path from $E_{1728}$ to $\Eone$.
    }
\end{algorithm}

\begin{proposition}
\label{prop:walkto1728}
Assume GRH, Heuristic~\ref{heur:prime}, and the assumptions of Section \ref{sec:represent}.
	Consider an endomorphism $\theta \in \End(E)$ in rationally-represented or prime-power isogeny-chain form as described in Section~\ref{sec:polyrep}, whose discriminant is coprime to $p$ and has $\ell$-fundamental part $\Delta$ satisfying $|\Delta| < p^2$.  Write $\mathcal{O}_\Delta$ for the order of discriminant $\Delta$. 
	 Algorithm~\ref{alg:pathto1728} produces a path of length $O(\log p+ h_{\mathcal{O}_\Delta})$  to $E_{1728}$ in the supersingular $\ell$-isogeny graph, under Heuristic \ref{heuristic:uniform-cordillera} part \ref{hA}.
	 The runtime is expected poly-rep 
	 times $O(h_{\mathcal{O}_\Delta} )$, under Heuristic \ref{heuristic:uniform-cordillera} part \ref{hB}.  
	 Furthermore, the following hold:
	 \begin{enumerate}
	 \item If $\ell$ is inert in $K$, then the runtime improves to $h_{\mathcal{O}_\Delta} \poly(\log p) + $poly-rep, and the path length improves to $O(\log p)$. 
	 \item If $\ell$ is inert in $K$ and the discriminant of $\theta$ is already $\ell$-fundamental, then the runtime improves to $h_{\mathcal{O}_\Delta} \poly(\log p)$ and the path length improves to $O(\log p)$.
	     \item If $\Delta$ is a fundamental discriminant, $\ell$ is split in $K$ and a prime above $\ell$ generates the class group $\Cl(\mathcal{O}_\Delta)$, 
	 then the dependence on Heuristic \ref{heuristic:uniform-cordillera} is removed.
	 \end{enumerate}
	
\end{proposition}

\begin{proof}
Let $\theta$ be the input to the algorithm.  The pair $(E,\iota_\theta)$, where $\iota_\theta: K \rightarrow \End(E)$ is the orientation given by $\theta$, lies somewhere on the oriented $\ell$-isogeny graph 
associated to $K$.  More specifically, it lies on a volcano of the $\mathcal{O}$-cordillera for some order $\mathcal{O}$ whose discriminant divides the $\ell$-fundamental discriminant $\Delta$ computed in Step~\ref{step:l-free}.  In other words, if we write $\mathcal{O}_\Delta$ for the order of discriminant $\Delta$, then $\mathcal{O} \supseteq \mathcal{O}_\Delta$.
Since all endomorphisms throughout the paper are taken to have norm and discriminant  at worst polynomial in~$p$, the distance of $(E, \iota_\theta)$ to the rim is at worst polynomial in $\log p$, and so walking to the rim (Step \ref{step:walkto1728-up1}) is poly-rep by Proposition~\ref{prop:walktorim-chain}.
Next, we walk around the rim; the runtime depends on the size of the rim and we defer that question to later in the proof.

When $\Delta$ is passed on to Algorithm~\ref{alg:1728orientation} in Step~\ref{step:1728orient}, the result (which is returned in polynomial time by Proposition~\ref{prop:1728returns} under Heuristic~\ref{heur:prime}) is an endomorphism of $\End(E_{1728})$ which gives an oriented elliptic curve lying somewhere on a volcano in an $\mathcal{O}'$-cordillera, where again $\mathcal{O}' \supseteq \mathcal{O}_\Delta$.   (We do not necessarily have $\mathcal{O} = \mathcal{O}'$.)  This has norm polynomial in $p$ by Proposition~\ref{prop:1728returns}. 
By Proposition~\ref{prop:1728returns} again, the distance to the rim is $O(\log p)$, so walking to the rim is 
expected polynomial time by Proposition~\ref{prop:walk-end}.
  Hence each \textbf{repeat} iteration has expected polynomial time.
 
 Walking to the rim in Step~\ref{step:walkfrominit}, $E_0$ lies on the rim of a volcano.  This volcano is somewhere in the set of volcanoes $\SSOnotprim$ defined as the finite union of the $\mathcal{O}$-cordilleras for all $\mathcal{O} \supseteq \mathcal{O}_\Delta$ in Heuristic~\ref{heur:uniform-volcanoes}.  Note that its conjugate $E_0^{(p)}$ also lies on a rim in $\SSOnotprim$.  Now $E_1$ also lies on a rim of $\SSOnotprim$.  If $E_0$ (or $E_0^{(p)}$) and $E_1$ lie on the same rim, the algorithm will discover this.  If not, then one continues the calls to Algorithm~\ref{alg:1728orientation}, and another endomorphism will be found.  Under Heuristic~\ref{heuristic:uniform-cordillera} part \ref{hA}, eventually one of these will produce $E_0$ or $E_0^{(p)}$ on the same rim as $E_1$.  The algorithm will then succeed.  

 Let $R$ denote the number of descending edges from the rim containing $E_0$, referred to in this paragraph as the \emph{adjusted rim size} (which is bounded above and below by a constant multiple of the rim size). The sum of the adjusted rim sizes of all rims of $\operatorname{SS}_{\mathcal{O}_\Delta}$  
 is $O(H_{\mathcal{O}_\Delta})$, with $H_{\mathcal{O}_\Delta}$ given by \eqref{classnosum}
 (Equation~\eqref{eqn:cl-frob} and Proposition~\ref{prop:sso-r}).  
 By Lemma~\ref{lem:hurwitz}, this is $O(h_{\mathcal{O}_\Delta} (\log \log |\Delta|)^2)) = O(h_{\mathcal{O}_\Delta})(\log \log p)^2$ (using $|\Delta|< p^2$).
 By Heuristic~\ref{heuristic:uniform-cordillera} part \ref{hB}, the number of times we must \textbf{repeat} is therefore $O(h_{\mathcal{O}_\Delta}/R)(\log\log p)^2$. 
Each iteration performs Steps~\ref {step:1728orient} and \ref{step:walkfrominit} and then checks membership in~$L$. By Proposition \ref{prop:1728returns}, under GRH, Step \ref {step:1728orient} runs in polynomial time in $\log p$ and provides a solution $\theta_{\operatorname{init}}$ of norm at most $p^2 \log^{2 + \epsilon} p$. Then $\theta_{\operatorname{init}}$ can be written as a linear combination of the $\mathbb{Z}$-basis of $\End(E_{1728})$ with integer coefficients of size $O(\log p)$. Hence Step \ref{step:walkfrominit} requires a runtime polynomial in $\log p$ by Proposition \ref{prop:walk-end}; we store the $j$-invariant of the output for comparison to~$L$.
Thus, each iteration takes expected polynomial time times $O(R)$ (to check membership in $L$). 
The walk to produce $L$ in Step \ref{step:getL} takes at most $O(R)$ steps, each of which is poly-rep.  Hence the runtime is poly-rep (for Step~\ref{step:pathhome_ascending_path}) plus $O(h_{\mathcal{O}_\Delta}) \cdot \poly(\log p)+O(R) \cdot \mbox{(poly-rep)}$.  

This runtime is overall bounded by $O( h_{\mathcal{O}_\Delta})$ times poly-rep. 
But if $\ell$ is inert, then $E_0$ lies on a rim of size~$1$, so we don't need Step~\ref{step:getL}, and we have poly-rep plus $h_{\mathcal{O}_\Delta} \poly(\log p)$.  If $\theta$ is already at the rim, then we don't need Step~\ref{step:pathhome_ascending_path}.  Combined with inertness, this gives runtime $h_{\mathcal{O}_\Delta} \poly(\log p)$. 

Finally, if $\Delta$ is a fundamental discriminant, $\ell$ is split and a prime above $\ell$ generates $\Cl(\mathcal{O}_\Delta)$, then there is only one volcano, obviating the need for Heuristic~\ref{heuristic:uniform-cordillera}.
\end{proof}

The restriction that $|\Delta| < p^2$ is required to ensure that Algorithm~\ref{alg:1728orientation} is heuristically polynomial time.  If $|\Delta|$ is larger, and $\ell$ is inert, this failure of polynomial time could become the bottleneck.  
On the other hand, suppose $\ell$ is split in $K$.  Under the Cohen-Lenstra heuristics, class groups are usually cyclic, and most elements of a cyclic group are generators, so with high probability, Heuristic~\ref{heuristic:uniform-cordillera}  will not be necessary. 

It is also possible to use Algorithm~\ref{alg:walktorim} at Step~\ref{step:pathhome_ascending_path}, instead of the methods of Section~\ref{sec:walk-end}.  This results in a worse runtime, but removes the dependence on GRH.

\begin{remark}
\label{remark:notvec}
One might hope to modify Algorithm~\ref{alg:pathto1728} to produce a shorter path along with a square-root runtime improvement, by removing Step~\ref{step:getL}, and in each \textbf{repeat}, attempting to solve a vectorization problem (see Section~\ref{sec:quantum-rim-walking}) between $E_0$ and $E_{1728}$.  Unfortunately, we cannot:  the problem is that we do not know the correct quadratic order $\mathcal{O}$ with respect to which these oriented curves are primitively oriented.  To overcome this, one might try to factor $\Delta$ and ascend with respect to any square factors, to guarantee that $\Delta$ is fundamental.  Ascending would be polynomial in the largest square prime factor of $\Delta$, which could be very costly.  An alternative that would usually work may be to try guessing $\Delta$, working backward from the largest (and hence most likely) divisors.  Just assuming $\Delta$ is fundamental would work much of the time.
\end{remark}

\begin{example}[\textbf{Finding a path to $E_{1728}$} via Algorithm \ref{alg:pathto1728}]\label{ex:pathto1728}
We again let $p = 179$, $\Delta = -47$, $\ell = 2$, and $E_{\text{init}} = E_{1728}: y^2 = x^3 - x$. As input, we consider the curve $E_{120}: y^2 = x^3 +  (7 \Fpi+86) x + (45 \Fpi+174)$ with $j(E_{120})= 120$, and a trace endomorphism given as $(E_{120}, \theta_{120}, t_{120}, n_{120})$ with $t_{120}= 20, n_{120}= 2^5 \cdot 3^2$ and 
\[
\resizebox{\displaywidth}{!}{$
\displaystyle
\theta_{120}(x, y)= \left( \frac{ (122\Fpi  + 167) x^{288} + (17\Fpi  + 68) x^{287} + \cdots + 174\Fpi  + 157}{ x^{287} + (78\Fpi  + 156) x^{286} + \cdots  + 16\Fpi  + 54}, 
\frac{( 69\Fpi  + 109) x^{431} + (60\Fpi  + 178) x^{430} + \cdots +  98\Fpi  + 124}{ x^{431} + (146\Fpi  + 53) x^{430}  + \cdots + 44\Fpi  + 89} \, y\right).
$
}
\]
We apply Algorithm \ref{alg:pathto1728} to find a path from $E_{120}$ to $E_{1728}$ (see Figure \ref{fig:path_to_Eint}).
Step 1 on input $(E_{120}, \theta_{120}, t_{120}, n_{120})$ produces the $\ell$-suitable and $\ell$-primitive  traced endomorphism  $\theta_{120} \leftarrow \theta_{120} + [-10]$  with $t_{120}\leftarrow 0$ and $n_{120}\leftarrow 188$. 
Here  $\Delta'= t_{120}^2-4 n_{120} = -752$ and its $\ell$-fundamental part is $\Delta= -47$. 
Step \ref{step:walkto1728-up1} calls Algorithm \ref{alg:walktorim} on input $(E_{120},\theta_{120}, t_{120}, n_{120})$ to produce the following ascending path $H_2$ to the rim, see Example~\ref{ex:walktorim_rationalrep}:
\[H_2: (E_{120}, \theta_{120}, 0, 188) \xrightarrow[]{\hspace*{0.2cm} \varphi_{120} \hspace*{0.2cm}} (E_{171}, \theta_{171}, 0, 47) \xrightarrow[]{\hspace*{0.2cm} \varphi_{171} \hspace*{0.2cm}} (E_{5i + 109}, \theta_{5i + 109}, 1, 12).\]
Now we apply Algorithm \ref{alg:cycleCGAction} on input $(E_{5i + 109},\theta_{5i + 109},t_{5i + 109},n_{5i + 109})$ to walk the rim in Step~\ref{step:getL} as in Example~\ref{ex:cycleCGAction}. The list of all the $j$-invariants is $L=\{5 \Fpi +109, 174 \Fpi + 109, 80 \Fpi + 107, 22, 99 \Fpi + 107\}$. In Step~\ref{step:1728orient}, calling Algorithm \ref{alg:1728orientation} on input $\Delta$, we obtain  $\theta_{1728} = (3i + k)/2$ as in Example \ref{ex:1728orientation}. For simplicity in this example, we use Algorithm~\ref{alg:walktorim} in Step~\ref{step:walkfrominit}, instead of the methods of Section~\ref{sec:walk-end}. We apply Algorithms~\ref{alg:suitable-chain} and \ref{alg:primitive} (see Section~\ref{sec:1728-orient-isogeny-chain}) to  $(E_{1728},\theta_{1728}, 0, 47)$ to obtain an $\ell$-primitive isogeny chain endomorphism $\theta'_{1728}=\varphi_{171} \circ \varphi_{1728}$ where $\deg(\varphi_{1728})=16$,  $\deg(\varphi_{171})=3$ and with  $t_{1728} = 2$,  $n_{1728} = 48$ as in Example \ref{ex:suitable-chain}. 
  We call Algorithm \ref{alg:walktorim} on input $(E_{1728}, \varphi_{171}\circ\varphi_{1728},2,48)$ to produce the following ascending path (see Example~\ref{ex:walktorim_chain}):
\[H_1:(E_{1728},\varphi_{171}\circ\varphi_{1728},2,48) \xrightarrow[]{\hspace*{0.2cm} \varphi_{1728}' \hspace*{0.2cm}} (E_{22}, \varphi_{174\Fpi + 109}\circ\varphi_{22}', 1, 12).\] 
Finally, since  $j(E_{22} )=22 \in L$, joining the previous paths, we obtain a path from $E_{1728}$ to $E_{120}$ (see the whole path in Figure \ref{fig:path_to_Eint}) as 
$$H: E_{1728} \xrightarrow[]{\hspace*{0.2cm} \varphi'_{1728}\hspace*{0.2cm}} E_{22}  \xrightarrow[]{\hspace*{0.2cm} \varphi_{22} \hspace*{0.2cm}} E_{99 \Fpi + 107}  
\xrightarrow[]{\hspace*{0.2cm} \varphi_{99 \Fpi + 107} \hspace*{0.2cm}} E_{5 \Fpi + 109} \xrightarrow[]{\hspace*{0.2cm} \hat{\varphi}_{171} \hspace*{0.2cm}} E_{171} \xrightarrow[]{\hspace*{0.2cm} \hat{\varphi}_{120} \hspace*{0.2cm}} E_{120}.$$

\end{example}

\section{Quantum algorithms for \textsc{Vectorization} and \textsc{PrimitiveOrientation} Problems}
\label{sec:quantum_part1}
We will introduce two hard problems: the oriented vectorization and the  primitive orientation problems and then provide quantum algorithms to solve them.

\subsection{Vectorization}
\label{sec:quantum-rim-walking}

Since the class group acts on the rim, a problem closely related to walking along the rim is the following, where we use the terminology \emph{vectorization} in analogy with \cite{Hard_Homo_Space_Couveignes} and \cite[Section 6.1]{chenu2021higherdegree}.  This problem was also recently introduced in \cite[Section 3.1]{WesolowskiOrientations}.

\begin{problem}[\textsc{OrientedVectorization($\Delta$)}]\label{pr:oriented_vec}
  Let $\mathcal{O}$ be the quadratic order of discriminant $\Delta$.	Suppose $(E_1, \iota_1), (E_2, \iota_2) \in \SSOpr$.  
	Find an ideal class $[\mathfrak{b}] \in \operatorname{Cl}(\mathcal{O})$ such that $[\mathfrak{b}] \cdot (E_1, \iota_1) = (E_2, \iota_2)$.
\end{problem}

\begin{remark}
\label{remark:seta}
This problem is somewhat related to the uber isogeny assumption, which asks for $[\mathfrak{b}]$ without knowledge of $\iota_2$; the difficulty of this problem is shown to be crucial for a variety of supersingular isogeny-based schemes \cite{SETA}.
\end{remark}

The following result was implied without details in a more restricted case in \cite[Section 6.1]{chenu2021higherdegree}.  A variation also appears in \cite[Proposition 4]{WesolowskiOrientations}.

\begin{heuristic}
\label{heur:bqf}
The values of a definite binary quadratic form $f(x,y)$, as $x,y \rightarrow \infty$, are powersmooth and coprime to the first $N$ primes with the same probability as randomly chosen integers of the same size.
\end{heuristic}

\begin{proposition}
\label{prop:vectorization}
Assume Heuristic \ref{heur:bqf}.  Suppose $(E_1, \iota_1)$ and $(E_2, \iota_2)$ are given by $\iota_i := \iota_{\theta_i}$ for some endomorphisms $\theta_i \in \End(E_i)$ which can be evaluated on $E_i(\FF_{p^k})$  in time $T_{\theta_i}(k,p) \ge \poly(k \log p)$.  Define $T_{\theta_1,\theta_2}(k,p) := \max\{ T_{\theta_1}(k,p), T_{\theta_2}(k,p) \}$ and $d := \max \{ \deg \theta_1, \deg \theta_2 \}$. 
Then 
\textsc{OrientedVectorization($|\Delta|$)} can be reduced to a hidden shift problem and solved in quantum time $T_{\theta_1,\theta_2}(O(\log^2 d), p) L_{|\Delta|}(1/2)$ under GRH, where, furthermore, the ideal class is $L_{|\Delta|}(1/2)$-smooth and of size $O(\sqrt{|\Delta|})$. 
\end{proposition}

\begin{proof}
	The approach is based on that in Childs-Jao-Soukharev \cite{CJS}, who developed a subexponential means of evaluating the action of the class group (by finding a smooth representative of the needed ideal class), and then applying Kuperberg's algorithm, which requires subexponentially many evaluations.  The difference is that we need to apply the class group action, in the form of isogenies, to \emph{oriented} curves, i.e.\ carry along the orientation.

	The reduction to the hidden shift problem is formalized in \cite[Theorem 3.3]{hidden_shift_KMPW}; the \emph{malleability oracle} in the sense of \cite[Definition 3.2]{hidden_shift_KMPW}, with respect to their notation, is given in terms of $I = G = \ClO$, $O=\SSOpr$, and $f: I \rightarrow O$ defined by $f([\mathfrak{a}]) =  [\mathfrak{a}] \cdot (E_1, \iota_1)$.  Then to find  $[\mathfrak{b}] \in \Cl(\mathcal{O})$  such that  $[\mathfrak{b}] \cdot (E_1,\iota_1) = (E_2,\iota_2)$, we observe that $f$ is malleable, because we can compute $[\mathfrak{a}] \mapsto f([\mathfrak{a}\mathfrak{b}]) = [\mathfrak{a} \mathfrak{b}] \cdot (E_1,\iota_1) = [\mathfrak{a}] \cdot (E_2,\iota_2)$ (this is the malleability oracle at $(E_2, \iota_2)$).

To evaluate the action of $[\mathfrak{a}]$  on $E_i$ takes time $\poly(\log p)L_{|\Delta|}(1/2)$ using the methods of \cite{CJS} or \cite{note_on_CSIDH_BIJ} and involves finding an $L_{|\Delta|}(1/2)$-smooth integral representative $\mathfrak{a}$ which can be evaluated as a composition chain of isogenies.  Unfortunately, to evaluate the action of $[\mathfrak{a}]$ on $\theta_i$, we require a powersmooth representative instead.  Calling on Heuristic~\ref{heur:bqf} and \cite[Section 3.1]{smoothbound_CN} (similarly to the proof of Proposition~\ref{prop:runtimeB}), we can find a representative with norm $L_{|\Delta|}(1/2)$-powersmooth and coprime to the first $\log \deg \theta_i$ primes, by random search.  The time taken is $L_{|\Delta|}(1/2)$, because by Mertens' Theorem, the probability of satisfying the coprimality hypothesis is $\prod_{\substack{p < O(\log \deg \theta)\\p \text{ prime}}}(1 - 1/p) \sim O(1/\log \log \deg \theta_i)$.  Having done this, write the result as $\mathfrak{a} := \prod \mathfrak{a}_k$, where the $N(\mathfrak{a}_k)$ are coprime prime powers.

We also need to evaluate the action of $\mathfrak{a}$ on $\theta_i$ in some way that is distinguishable (since isogeny chains are not unique for a given endomorphism).  For each $j$-invariant we choose a fixed model.  We replace the data of $\theta$ with the data of its linear action on the $O(\log \deg \theta_i)$ smallest prime-torsion subgroups $E[q]$, as well as all the prime-power $N(\mathfrak{a}_k)$-torsion subgroups.  By Chinese Remainder Theorem, this is enough to distinguish different results, since if $\theta - \theta'$ vanishes on all of the prime-power subgroups, then it vanishes on a subgroup (generated by all of the subgroups together), whose size exceeds a fixed multiple of $d$, which implies that $\theta=\theta'$ (this method is inspired by the Schoof algorithm, as adapted for example in \cite[Theorem 81]{Kohelthesis}, \cite[Lemma 4]{EHLMP_reductions}). 

To compute the action on $\theta_i$, we first need to compute $\varphi_{\mathfrak{a}_k}$.  This is done as in Algorithm~\ref{alg:cycleCGAction}, where we consider the linear action of $a + b\theta_i$ on the $N(\mathfrak{a}_k)$-torsion to find the kernel of $\varphi_{\mathfrak{a}_k}$.  In order to compute the linear action of $\varphi_{\mathfrak{a}_k} \circ \theta_i \circ \widehat{\varphi_{\mathfrak{a}_k}} / [N(\mathfrak{a}_k)]$ on the prime or prime-power torsion subgroups $E[q]$ described in the last paragraph, we proceed as follows.   If $q$ is coprime to $N(\mathfrak{a}_k)$, then to find this action, we evaluate $\varphi_{\mathfrak{a}_k} \circ \theta_i \circ \widehat{\varphi_{\mathfrak{a}_k}}$ on $E[q]$ and then apply the action of $[n']$ where $n' \equiv N(\mathfrak{a}_k)^{-1} \pmod q$.  Otherwise we store \textbf{null} for that value of $q$ (by assumption, this occurs only for $q$ larger than $\log \deg \theta_i$).  

This gives a way to evaluate the function $f$ suitable for quantum computation.  Taken together, the time taken for evaluating $[\mathfrak{a}_k]$ is $\poly(\log \deg \theta_i)$ times the time taken to evaluate $\theta_i$ and $\varphi_{\mathfrak{a}_k}$, namely $T_{\theta_1,\theta_2}(O(\log^2 d),p) + \poly(\log p) L_{|\Delta|}(1/2)$.

There is a small caveat that the action of Frobenius may take us out of the orbit of $\ClO$, so this will only work when the oriented curves $E_1$ and $E_2$ are in the same $\ClO$-orbit.  Of course, there are at most two orbits, so in the case of failure, we can apply Frobenius to one of the curves and try again.  

The evaluation algorithm of  \cite{CJS}  runs in time $L_{|\Delta|}(1/2)$ under GRH and results in an $L_{|\Delta|}(1/2)$-smooth isogeny of size $O(\sqrt{|\Delta|})$ \cite[Proposition 3.2]{CJS}.  Our modification above results in the stated runtime.
\end{proof}

\begin{remark}
If we wish to avoid the coprimality aspect of Heuristic~\ref{heur:bqf}, then we can take subexponentially many prime power torsion subgroups, at an increased cost in runtime and memory (thanks to Benjamin Wesolowski for this and other helpful observations and corrections to this proof).
\end{remark}

\begin{remark}
If we wish to avoid Heuristic~\ref{heur:bqf} in  Proposition \ref{prop:vectorization}, we could first transform $\theta_i$ into a powersmooth isogeny chain (using Algorithm~\ref{alg:suitable-chain} at a runtime cost of $T_{\theta_1,\theta_2}(L_{\deg \theta_i}(1/2), p)$) and then use the method for horizontal stepping of Algorithm~\ref{alg:cycleCGAction} to evaluate $[\mathfrak{a}]$ prime-by-prime.  This depends on Heuristic~\ref{heur:translates} instead.  This allows for the representative $\mathfrak{a}$ to be chosen as smooth, not necessarily powersmooth, but incurs an additional runtime cost to the algorithm as a whole.
\end{remark}

\subsection{Primitive orientation computation}\label{sec:quantum-prim-orientation}

The vectorization problem~\ref{pr:oriented_vec} requires knowledge of the order with respect to which $(E,\iota)$ is a primitive orientation. This requirement naturally leads to the following problem:

\begin{problem}[\textsc{PrimitiveOrientation}] \label{pr-vec}
	Given an supersingular elliptic curve $E$, and an endomorphism $\theta \in \End(E)$, determine the quadratic order $\mathcal{O}$ such that $\iota_\theta$ is $\mathcal{O}$-primitive.
\end{problem}

We briefly describe two classical algorithms here for solving Problem~\ref{pr-vec}. Let $f$ be the conductor of $\ZZ[\theta]$, we compute a $B$-powersmooth $f$-suitable translation and factorize $f = \Pi {f_i}^{r_i}$. For any prime power factor ${f_i}^{r_i}$ of $f$, one needs to check whether the translated endomorphism is divisible by ${f_i}^{r_i}$, which amounts to checking whether $\theta$ vanishes on the ${f_i}^{r_i}$-torsion of $E$. We take $B$ to be $L_{d}(1/2)$ with $d = \deg \theta$, as discussed in the proof of Theorem~\ref{thm:intro-classical}, using Algorithm~\ref{alg:suitable-chain}, computing the translation takes time $T_\theta(L_d(1/2),p)$ assuming Heuristic~\ref{heur:translates} with $\ell$ replaced by $f$ in Heuristic~\ref{heur:translates}. Furthermore, evaluating the translated endomorphism on ${\tilde{f}}^{\tilde{r}}$-torsion takes time $\poly(\log p)L_d(1/2)\mulM(p^{\lcm(12,{\tilde{f}}^{2\tilde{r}})})$ where ${\tilde{f}}^{\tilde{r}} = \max \{{f_i}^{r_i}\}$. 
Alternatively, one can compute an integer $T$ with smallest absolute value such that $\theta+T$ is $f$-suitable translation instead of a $B$ powersmooth translation. Checking whether $\theta$ vanishes on the ${f_i}^{r_i}$-torsion of $E$ takes time $\poly(\log p)T_\theta({\tilde{f}}^{2\tilde{r}},p)$ where ${\tilde{f}}^{\tilde{r}} = \max \{{f_i}^{r_i}\}$. Both methods have runtimes polynomial in $\tilde{f}^{r_i}$.

Quantumly  we give the following algorithm that runs in subexponential time. Our method for solving Problem~\ref{pr-vec} has similarities to that of Proposition~\ref{prop:vectorization}, with a hidden subgroup problem in place of the hidden shift problem.  The subexponential runtime in $\Delta$ still arises from the need to evaluate the action of the class group.

\begin{proposition}
	\label{prop:primitive-orientation}
	Assume Heuristic~\ref{heur:bqf}.  Suppose $\theta$ can be evaluated on $E(\FF_{p^k})$ in time $T_\theta(k,p)$.  Then there is a quantum algorithm  to solve \textsc{PrimitiveOrientation} in  time $T_\theta(O(\log^2 \deg \theta),p) + \poly(\log p) L_{|\Delta|}(1/2)$. 
\end{proposition}

\begin{proof}
	Let $\mathcal{O}_\theta := \ZZ[\theta]$.  Compute $\Cl(\mathcal{O}_\theta)$ as a product of cyclic groups with given generators, using the quantum algorithm  \cite[Algorithm 10]{Decompose_finite_abgps_CM}, as described in \cite[Proof of Theorem 4.5 ]{CJS}.
	It is possible to solve the \textsc{PrimitiveOrientation} problem by computing the kernel of the map $\Cl(\mathcal{O}_\theta) \rightarrow \Cl(\mathcal{O})$ (where we do not a priori know $\mathcal{O}$).  This can be done by solving a hidden subgroup problem.  Namely, we consider the action of $\Cl(\mathcal{O}_\theta)$ on $\SSOpr$, defining $f([\mathfrak{b}]) = [\mathfrak{b}] \cdot (E,\iota_\theta)$.  
	We evaluate the action of $\mathfrak{b}$ on $\theta$ as described in the proof of Proposition~\ref{prop:vectorization}. 

	Once the kernel $G$ has been computed in the form of generators $\mathfrak{g}_1, \ldots, \mathfrak{g}_n$, one writes each $\mathfrak{g}_i$ as principal in the maximal order via a generator $\mathfrak{g}_i = (g_i)$.  Then $\mathcal{O}$ is by definition the order generated from $\mathcal{O}_\theta$ by adjoining the $g_i$'s.  One computes the conductor of this order by taking the $\gcd$ of the conductors of the $\ZZ[g_i]$ and $\ZZ[\theta]$, and hence computing the discriminant $\Delta_\mathcal{O}$.  These last computations are polynomial in $\log |\Delta_\theta|$.
\end{proof}

An improvement is available: to evaluate the action of $[\mathfrak{b}]$ on $E$ takes time $\poly(\log p)\exp( \widetilde{O}(\log^{1/3} |\Delta_\theta|) )$ using the methods of Biasse-Iezzi-Jacobson \cite{note_on_CSIDH_BIJ}; they also improve on the computation of $\ClO$.

\section{Quantum algorithm for finding a smooth isogeny to $j=1728$}\label{sec:quantum_part2}

The problems of computing the endomorphism ring of an elliptic curve $E$, computing an $\ell$-power isogeny to an initial curve (such as $j=1728$), and computing a smooth isogeny to an initial curve, are all equivalent \cite{Wesolowski_IsogPathandEndoRing}.
In this section, we modify Algorithm~\ref{alg:pathto1728} to find a smooth isogeny, using the quantum algorithms of the previous section  (Propositions~\ref{prop:vectorization} and \ref{prop:primitive-orientation}).  The resulting quantum algorithm is Algorithm~\ref{alg:pathto1728-quantum}.

\begin{algorithm}
	\caption{\,Finding a smooth isogeny to $\Einit$ (quantum)} \label{alg:pathto1728-quantum}
    \vspace{.2ex}
    \Input {%
           A traced endomorphism $(\Eone,\theta,t,n)$ which can be evaluated on arbitrary points, where the discriminant of $\theta$ is coprime to $p$.
           }
    \Output {%
	    A smooth isogeny $E \rightarrow \Einit$.
        }
    \vspace{.4ex}
	$\Delta \leftarrow t^2 - 4n$
    \label{step:disc2}
    \;
    Choose the smallest prime $\ell$ so that $\ell^2$ does not divide $\Delta$ or $n$.

    $\Delta^* \leftarrow $ the discriminant of the solution to \textsc{PrimitiveOrientation} for $(E, \theta)$ via Proposition~\ref{prop:primitive-orientation}.

    \Repeat{ $\Delta^* = \Delta^{**}$ }{
    
            Call Algorithm \ref{alg:1728orientation} on input $\Delta^*$, to obtain a new traced endomorphism $(\Einit,\theta_{\operatorname{init}},t_{\operatorname{init}}, n_{\operatorname{init}})$.  (Algorithm~\ref{alg:1728orientation} can be suspended and then resumed to find subsequent solutions; see Remark~\ref{remark:continuously}.)
    \label{step:1728orient-quantum}

    Walk from $(\Einit,\theta_{\operatorname{init}},t_{\operatorname{init}}, n_{\operatorname{init}})$ to produce an ascending path $H_1$ from $(\Einit, \theta_{\operatorname{init}},t_{\operatorname{init}}, n_{\operatorname{init}})$ to $(E_0, \theta_0, t_0, n_0)$ on the rim, i.e. where $\ZZ[\theta_0] \subseteq \End(E_0)$ is $\ell$-fundamental (methods of Section~\ref{sec:walk-end}).
        \label{step:walkfrominit-quantum}
    \;
    $\Delta^{**} \leftarrow$ the discriminant of the solution to \textsc{PrimitiveOrientation} for $(E_0, \theta_0)$ via Proposition~\ref{prop:primitive-orientation}.\label{step:delta-prime-quantum}
    }
           Use a quantum computer to solve \textsc{OrientedVectorization($\Delta'$)} as described in Proposition \ref{prop:vectorization}, to find an ideal class $[\mathfrak{a}] \in \Cl(\mathcal{O}_{\Delta'})$ such that $[\mathfrak{a}](E_1, \iota_{\theta_1})$ is $(E_0, \iota_{\theta_0})$ or $(E_0^{(p)}, \iota_{\theta_0}^{(p)})$ (try both). 
      \label{step:vectorization-quantum}
\end{algorithm}

\begin{proposition}
\label{prop:quantum-walkto1728}
Assume GRH, Heuristic~\ref{heur:prime}, \ref{heuristic:uniform-cordillera}, and \ref{heur:bqf},  and the assumptions of Section~\ref{sec:represent}.
Suppose $\theta$ can be evaluated on $E(\FF_{p^k})$   in time $T_\theta(k,p) \ge \poly(k \log p)$.  Let $d = \max\{ \deg \theta, |\Delta| \}$.  Suppose $|\Delta| < p^2$ and $\Delta$ is coprime to $p$. 
Algorithm~\ref{alg:pathto1728-quantum} is correct and succeeds in heuristic expected time $T_\theta(O(\log^2 d), p) L_{|\Delta|}(1/2)$.  The resulting $L_{|\Delta|}(1/2)$-smooth isogeny has norm $O(\sqrt{|\Delta|})$.  
\end{proposition}

\begin{proof}
	The algorithm determines $\Delta^*$ so that $\iota_\theta$ is $\mathcal{O}_{\Delta^*}$-primitive.  In the \textbf{repeat} loop, it finds an orientation of $j=1728$ and a path from that oriented curve to an oriented curve $(E_0, \iota_{\theta_0})$ which is primitive with respect to the same order.  Thus vectorization applies, and finds a smooth isogeny between $(E,\iota_{\theta})$ and $(E_0, \iota_{\theta_0})$.  Combining the path and isogeny, we find a smooth isogeny between $E$ and the initial curve.

	The first two steps take time $O(\log |\Delta|)$.  The third step takes time $T_\theta(\log \deg \theta, p) + \poly(\log p) L_{|\Delta|}(1/2)$ by Proposition~\ref{prop:primitive-orientation}.  Steps \ref{step:1728orient-quantum} and
        \ref{step:walkfrominit-quantum} take polynomial time in $\log p$ and $\log |\Delta|$ by Proposition~\ref{prop:alg:1728orientation} and Proposition~\ref{prop:walk-end}.  Step~\ref{step:delta-prime-quantum} again takes time $T_\theta(\log \deg \theta, p) + \poly(\log p) L_{|\Delta|}(1/2)$.  
    To determine how often we must \textbf{repeat}, we compute that the probability that $\Delta^* = \Delta$ is equal to $h_\mathcal{O}/H_\mathcal{O}$, with $H_\mathcal{O}$ given by \eqref{classnosum} (by consideration of the sizes of $\SSOpr$ (Equation~\eqref{eqn:cl-frob} and Proposition~\ref{prop:sso-r}) and using Heuristic~\ref{heuristic:uniform-cordillera}). Thus, by Lemma~\ref{lem:hurwitz}, the expected number of iterations is $\poly(\log p)$.

Note that the endomorphism found by Algorithm~\ref{alg:1728orientation} is of norm $O(|\Delta|)$.  Therefore the rim endomorphism~$\theta_0$ is also of norm $O(|\Delta|)$.  Thus, 
\textsc{OrientedVectorization} in Step~\ref{step:vectorization-quantum}  takes time $T_\theta(
O(\log^2 d),p)L_{|\Delta|}(1/2)$ (Proposition~\ref{prop:vectorization}).  Note that the evaluation time for $\theta_0$ on small torsion is $O(\log p)$ since we have expressed $\theta_0$ as a linear combination of basis elements, each of which can be evaluated via the chain down to $j=1728$.
\end{proof}

\section{Proofs of Main Theorems and Special Cases}
\label{sec:consequences}

\subsection{Proof of main theorems}
\label{sec:proof-intro}

\begin{theorem}
\label{thm:intro-classical}
Choose a small prime $\ell$ and assume the heuristic assumptions of Proposition~\ref{prop:walkto1728}.
Let $\theta \in \End(E)$ be an endomorphism of degree $d$, such that $L_d(1/2) \ge \poly(\log p)$.  Suppose $\theta$ can be evaluated on points $P \in E(\mathbb{F}_{p^k})$ in time $T_\theta(k,p)$. Let $\Delta'$ be the $\ell$-fundamental part of the discriminant $\Delta$ of $\theta$, and assume that $|\Delta'| \le p^2$.
There is a classical algorithm that, given any such $\theta$, finds an $\ell$-isogeny path of length $O(\log p + h_{\Delta'})$  from $E$ to the curve $\Einit$ of $j$-invariant $j=1728$ in runtime $T_\theta(L_d(1/2),p) + h_{\Delta'} L_d(1/2) \poly(\log p)$. 
\end{theorem}

The runtime comes as a sum of two terms because the algorithm has two steps:  first, evaluate the endomorphism on points in order to create a presentation of the endomorphism that meets the needs of the main algorithm; and then second, use the result to walk in the oriented graph.

\begin{proof}[Proof of Theorem \ref{thm:intro-classical}]
	Suppose $\theta$ is such an endomorphism.  Then set $B = L_d(1/2)$.  We can apply Algorithm~\ref{alg:suitable-chain} (having Algorithm~\ref{alg:refactor-chain} as a subroutine) to $\theta$, whose runtime depends on the evaluation of~$\theta$ on inputs in a field $\mathbb{F}_{p^{O(B^2)}}$.  The runtime for this conversion is therefore $T_\theta(L_d(1/2),p)$.  The result is a prime-power isogeny-chain representation of $\theta$.  We can then use Algorithm~\ref{alg:pathto1728}, with the representation runtime being $L_d(1/2)$, by Proposition~\ref{prop:runtime-summary}.  The classical runtime follows from Proposition~\ref{prop:walkto1728}.  
	\end{proof}

\begin{theorem}
\label{thm:intro-quantum}
Assume GRH, Heuristic~\ref{heur:prime}, \ref{heuristic:uniform-cordillera}, and \ref{heur:bqf},  and the assumptions of Section~\ref{sec:represent}.
Let $\theta \in \End(E)$ be an endomorphism which can be evaluated on points $P \in E(\mathbb{F}_{p^k})$ in time $T_\theta(k,p)$, where $T_\theta(k,p) \ge \poly(k \log p)$.  Suppose $\theta$ has discriminant $\Delta$ coprime to $p$
with $|\Delta| \le p^2$.
Let $d = \max\{ \deg \theta, |\Delta| \}$.  There is a quantum algorithm that, given any such $\theta$, finds an $L_{|\Delta|}(1/2)$-smooth isogeny of norm $O(\sqrt{|\Delta|})$ from $E$ to $j=1728$ in runtime $T_\theta(O(\log^2 d), p)L_{|\Delta|}(1/2)$.
\end{theorem}

	\begin{proof}[Proof of Theorem \ref{thm:intro-quantum}]
	We use Algorithm~\ref{alg:pathto1728-quantum}, with no need to pre-process $\theta$.  Runtime follows from Proposition~\ref{prop:quantum-walkto1728}.  
\end{proof}

\subsection{Special cases}

In this section, we refer to an endomorphism as \emph{insecure} if access to such an endomorphism allows for a polynomial time path-finding algorithm.  Endomorphisms of small size are known to be insecure \cite{BonehLove}.  We obtain a version of this from our methods also.

\begin{theorem}
\label{thm:intro-poly}
Assume the situation of Theorem~\ref{thm:intro-classical}.  In the following special cases, the runtime and path length of Algorithm~\ref{alg:pathto1728} are polynomial in $\log p$:
\begin{enumerate}
    \item The input endomorphism is rationally represented in polynomial space.
    \item $h_{\mathcal{O}_\Delta} = \poly(\log p)$  and $\ell$ is coprime to $\Delta$ and inert in $K$.  In this case, the endomorphism is not even needed as input; only its existence, trace and norm are needed.
\end{enumerate}
\end{theorem}

\begin{proof}
The second case is a consequence of Algorithm~\ref{alg:pathto1728} and Proposition~\ref{prop:walkto1728}, in which the hypotheses imply Steps~\ref{step:pathhome_ascending_path} and \ref{step:getL} are unnecessary.  The first is a consequence of the observation that such endomorphisms have polynomially sized discriminants and class numbers.
\end{proof}

The following result demonstrates for all curves the existence of non-small endomorphisms which are insecure under our algorithm.  (Recall that most curves do not have small endomorphisms.  It is known that there are curves having no endomorphisms of norm smaller than $p^{2/3-\epsilon}$ (see \cite[Proposition B.5]{BonehLove_ARXIV}, 
\cite[Section 4]{ElkiesThesis},
\cite[Proposition 1.4]{MinCMLiftings}).)

\begin{theorem}
\label{thm:consequence-poly}
Suppose $\Delta = f^2 \Delta^*$ where $\Delta^*$ is a discriminant of $\poly(\log p)$ size, $f$ is $\poly(\log p)$-smooth, and $\theta$ is $f$-suitable with $\poly(\log p)$-powersmooth norm, and represented in some fashion so that it can be evaluated in $\poly(\log p)$ time on points of $\poly(\log p)$ size.  Then there is a classical  algorithm to find an $O(\log p)$-powersmooth isogeny to $\Einit$ in time $\poly(\log p)$.  
\end{theorem}

\begin{proof}
The dependence on $\ell$ throughout the paper has been suppressed by assuming $\ell = O(1)$, but it is at worst polynomial throughout.
We refactor $\theta$ in $\poly(\log p)$ time (this is possible by Proposition~\ref{prop:refactor} and the evaluation runtime assumption), to obtain an isogeny chain.
Taking each prime $\ell$ dividing $f$ in turn, we ascend as for as possible on the oriented $\ell$-isogeny volcano.  By $f$-suitability, we can ascend without any further translation or refactoring. Having ascended, we obtain an endomorphism of discriminant $\Delta^*$ of $\poly(\log p)$ size and trace zero, and hence call on Theorem~\ref{thm:intro-poly} with respect to some suitable $\ell$.
\end{proof}

In fact, every elliptic curve has insecure endomorphisms:  one can provide an endomorphism in the form of a closed walk in the $\ell$-isogeny graph that passes through $1728$.   Such a path is guaranteed to exist by the diameter of the graph.  In that case, one hardly needs the algorithms of this paper, as the path to $1728$ is already explicit.  A variation on this theme is to provide a $\poly(\log p)$-powersmooth isogeny chain whose endomorphism has minimal polynomial $x^2 + L^2$ (i.e., $L$ is powersmooth).  Such a chain will be insecure because it explicitly passes through $1728$ and also under the algorithms provided in this paper (by Theorem~\ref{thm:consequence-poly}).  

More interestingly, examples of such endomorphisms exist whose minimal polynomial places them in any field $\QQ(\omega)$ with $\poly(\log p)$ discriminant (not just the Gaussian field as above); indeed one can take any element of the form $L(\omega + k)$ for $k \in \ZZ$ and a $\poly(\log p)$-powersmooth $L$ such that the norm $N(\omega + k)$ is $\poly(\log p)$-powersmooth.

Finally, we remark on one more special case.  When the norm of $\theta$ is well-behaved, and we are already at the rim with respect to $\ell$ (perhaps by choosing $\ell$ judiciously), then we have improved dependence on $p$.  Note that in the following theorem, there is no requirement on the factorization of $\Delta$.

\begin{theorem}
\label{thm:consequence-smooth}
Suppose the norm of $\theta$ has powersmoothness bound $B(p)$, and suppose that $\Delta$ is coprime to $\ell$.  Then there is a classical  algorithm  to find an $\ell$-isogeny path of length $O(\log p + h_\mathcal{O})$ to $\Einit$ in time $h_\mathcal{O} \poly(B(p) \log p) $.
\end{theorem}

\begin{proof}
Use Algorithm~\ref{alg:pathto1728}.  By the assumption on $\Delta$, we need not ascend with $\theta$ (that is, we skip Step~\ref{step:pathhome_ascending_path}).  We only walk horizontally, and those steps are polynomial in $B(p)$ by Proposition~\ref{prop:cycleCGAction}.
\end{proof}

\section{Division by $[\ell]$}\label{sec:division_by_ell}

We conclude with a detailed description and analysis of
McMurdy's algorithm  (Algorithm~\ref{alg:divisionbyell}) which can be used to divide any \emph{isogeny} (not just an endomorphism) by $[\ell]$ if it is a multiple of $[\ell]$. Given a rationally represented traced endomorphism, we apply Algorithm~\ref{alg:divisionbyell} and then adjust the trace and norm accordingly.

We follow the notation of McMurdy \cite{mcmurdy2014explicit}. Let $E_1$ and $E_2$ be two supersingular elliptic curves given by respective short Weierstrass equations
\[ E_1: y^2 = W_1(x), \qquad E_2 : y^2=W_2(x).\] 
with $W_1(x), W_2(x) \in \FF_{p^2}[x]$. Denote by $\psi_{E_1,\ell}$ the $\ell$-division polynomial of $E_1$, made monic, and let $X_i(x)$ and $Y_i(x)$ be the rational functions representing the multiplication-by-$\ell$ map on $E_i$, i.e.\ $[\ell]_{E_i}(x,y) = ( X_i(x), Y_i(x)y )$ for $i=1, 2.$ For a polynomial $P(x)=(x-r_1) \cdots (x-r_n)$ with coefficients in some field $\FF$ whose roots $r_i$ lie in some  field extension $\FF'$ of $\FF$, and a rational function $T(x)$ over $FF'$, define%
$$P(x)\big| T:=\left(x-T(r_1)\right) \cdots \left(x-T(r_n)\right).$$ %
Given $[\ell]\varphi:E_1 \rightarrow E_2$ as a pair of rational maps, where $\varphi: E_1\rightarrow E_2$ is an isogeny, the rational maps of $\varphi$ are obtained as follows.

\begin{proposition}[{\cite[Proposition 2.6]{mcmurdy2014explicit}}] \label{pro:mainmcmurdy}
Suppose that $\varphi: E_1\rightarrow E_2$ is a separable isogeny such that $([\ell]\varphi)(x,y) = (F(x),G(x)y)$ for rational functions $F(x), G(x)$. Write $F(x)$ in lowest terms, i.e.\ as either $\frac{c_F\cdot P(x)}{W_1(x)Q(x)}$ when $\ell = 2$ or  $\frac{c_F\cdot P(x)}{(\psi_{E_1,\ell}(x))^2Q(x)}$ when $\ell \neq 2$, with monic polynomials $P(x), Q(x)$. Set 
\[p(x) = P(x)\big|X_1, \,\,\, q(x) = Q(x) \big| X_1.\]
Then $p(x)=p_0(x)^{\ell^2}$ and $q(x)=q_0(x)^{\ell^2}$ for monic polynomials $p_0(x), q_0(x)$. Moreover, we have $\varphi(x,y) = (f(x),g(x)y)$, where $f(x) = c_F\ell^2\cdot \frac{p_0(x)}{q_0(x)}$ and $g(x) = \frac{G(x)}{Y_2(f(x))}$.
    \end{proposition}

Algorithm \ref{alg:evaluate} computes the polynomials $p(x)$ and $q(x)$ as given in Proposition \ref{pro:mainmcmurdy}. The main division-by-$[\ell]$ process (Algorithm \ref{alg:divisionbyell}) then calls Algorithm~\ref{alg:evaluate} twice.



\begin{algorithm}
    \caption{\,Computing the polynomial $P(x)\big|X_1$}  \label{alg:evaluate}
    \vspace{.2ex}
    \Input {%
            An elliptic curve $E_1$, a monic polynomial $P(x)$ defined over $\FF_{p^m}$, and the rational map $X_1(x)$ associated to $E_1$. 
        }
     
    \Output {
        $P(x)\big|X_1$. 
        }
    %
    %
    \vspace{.4ex}
    Compute  a root $\zeta$ (in some field extension of $F_{p^2}$) of $X_1$. \label{step:alg-eval-root} \;
    
    Compute the $x$-coordinates $x_i$ (in some field extension of $F_{p^2}$) of the points $S_i =(x_i, y_i) \in E_1[\ell]$, indexed by $i=1,\ldots, \ell^2-1$ so that $x_{i+\frac{\ell^2-1}{2}} = x_i$, using the $\ell$-th division polynomial (note that we do not compute the $y_i$ here). 
    Let $S_0= O_{E_1}$. 
    \label{step:alg-eval-torsion}
    \;
    Compute the $x$-coordinates $\mathbf{x}_i(x,y,y_i)$ for $1 \le i \le \frac{\ell^2-1}{2}$ of the maps  representing point addition $(x,y)+S_i$ on~$E_1$, using the values of $x_i$ computed in step \ref{step:alg-eval-torsion} but leaving $y_i$'s as  indeterminates. Set  $\bar{\mathbf{x}}_i(x,y,y_i)=\mathbf{x}_i(x,y,-y_i)$ which is the $x$-coordinate of the  point addition $(x,y)+ (-S_i)$. \label{step:alg-eval-alpha}
         \;

    $N(x) \leftarrow P(x)$ and $D(x) \leftarrow 1$.\label{step:alg-eval-set1A}
    \;

    \For{$i =1, \ldots, \frac{\ell^2-1}{2}$}{\label{step:alg-eval-norm}
    
        Compute  $P(\mathbf{x}_i(x,y,y_i))$ and $P(\bar{\mathbf{x}}_i(x,y,y_i))$ (as rational functions in $x, y$ and $y_i$) using Horner's algorithm. \label{step:alg-eval-Palpha}\\
        
        Compute the numerator $N_i$ and denominator $D_i$ of $P(\mathbf{x}_i) P(\bar{\mathbf{x}}_i)$ as polynomials in $x, y$ and $y_i$. \label{step:alg-eval-computepair}\\
        
        Replace $y^2$ with $W_1(x)$ and $y_i^2$ with $W_1(x_i)$ in $N_i$. Denote the result by $N_i(x)$, as no $y$'s or $y_i$'s should remain. \label{step:alg-eval-substitutenum}\\
        
        Replace $y_i^2$ with $W_1(x_i)$ in $D_i$. Denote the result by $D_i(x)$, as no $y$'s or $y_i$'s should remain.\label{step:alg-eval-substituteden}\\
        
        $N(x) \leftarrow N(x) \cdot N_i(x)$, and $D(x) \leftarrow D(x) \cdot D_i(x)$. \label{step:alg-eval-product}\;
        
    }
   
    $N_P(x)\leftarrow \frac{N(x)}{D(x)}$, $i \leftarrow 0$, $p(x)\leftarrow 0$.   \label{step:alg-eval-set2} \;

    \For{$i =0, \ldots, \deg(P(x))$}{ \label{step:alg-eval-coefficients}
         $a_i \leftarrow N_P(\zeta)$.\label{step:alg-eval-ai}\\
        $p(x) \leftarrow p(x) + a_i x^i$.\\
         $N_P(x) \leftarrow N_P(x)-a_i x^i$.\\
        $N_P(x) \leftarrow N_P(x)/X_1(x)$.\label{step:alg-eval-divideX1}\\
    }

     \Return{ $p(x)$.
     }
\end{algorithm}


\begin{algorithm}
    \caption{\,Division by $[\ell]$.} \label{alg:divisionbyell}
    \vspace{.2ex}
    \Input {%
            Elliptic curves $E_1, E_2$, rational maps $F(x)$ and $G(x)$ where $([\ell]\varphi)(x,y) = (F(x),G(x)y)$ for some isogeny $\varphi: E_1\rightarrow E_2$.
        }
    \Output {
        Rational maps $f(x)$ and $g(x)$ such that $\varphi(x,y) = (f(x),g(x)y)$.
        }

    \vspace{.4ex}
    Determine $c_F$, and the monic polynomials $P(x)$ and $Q(x)$ such that  $F(x)=\frac{c_F\cdot P(x)}{W_1(x) \cdot Q(x)}(\ell = 2)$ or $F(x)=\frac{c_F\cdot P(x)}{(\psi_{E_1,\ell}(x))^2 \cdot Q(x)} (\ell \neq 2)$. \label{step:algdivisionbyellPQ}

    Compute $X_1(x)$ and $Y_2(x)$.  \label{step:algdivisionbyell-X1Y2}
     
    Compute $p(x) \leftarrow P(x)\big|X_1$ using Algorithm \ref{alg:evaluate} on input $E_1, P(x), X_1(x)$. \label{step:algdivisionbyellevaP}\;

    Compute $q(x) \leftarrow Q(x) \big|X_1$ using Algorithm \ref{alg:evaluate} on input $E_1, Q(x), X_1(x)$. In this step we can skip Steps 1--4 in  Algorithm \ref{alg:evaluate} since they were already performed in Step \ref{step:algdivisionbyellevaP} of this algorithm. \label{step:algdivisionbyellevaQ}\;
     
     Compute $p_0(x) \leftarrow p(x)^{1/\ell^2}$ and  $q_0(x) \leftarrow q(x)^{1/\ell^2}$ using a truncated variant of Newton's method.  \label{step:algdivisionbyellell2root}\;

    $f(x) \leftarrow c_F\ell^2\cdot \frac{p_0(x)}{q_0(x)}$,  $g(x) \leftarrow  \frac{G(x)}{Y_2(f(x))}$. \label{step:algdivisionbyellfg}
    
     \Return{ $f(x), g(x)$.
     }
\end{algorithm}

Division by $\ell=2$ has been implemented by McMurdy \cite{mcmurdy2014explicit} (code available at \cite{McMurdyCode}).
Division by odd primes $\ell > 2$ is complicated by the non-vanishing of the $y$-coordinates of the $\ell$-torsion points. Fix an odd prime $\ell>2$. In order to compute $p(x)= P(x)\big|X_1$ and $q(x)= Q(x)\big|X_1$ in Steps \ref{step:algdivisionbyellevaP} and \ref{step:algdivisionbyellevaQ} of Algorithm \ref{alg:divisionbyell}, we compute the rational map $N_P= \prod_i P(\mathbf{x}_i)$ as a function of the variable $x$ only. 
In contrast to the case of 2-torsion points, the $\ell$-torsion points on $E_1$ have non-zero $y$-coordinates, so some $\mathbf{x}_i$ depend not only on $x$ (as in the case $\ell=2$) but also on $y$ and $y_i$ for $i \le (\ell^2-1)/2$. As a consequence, $N_P$ also depends on these variables. To overcome  this obstruction, we employ a new technique presented in Steps \ref{step:alg-eval-norm}--\ref{step:alg-eval-set2}
of Algorithm \ref{alg:evaluate}. In these steps, we compute the products $\mathbf{x}_i \cdot \bar{\mathbf{x}}_i$, and hence the products $P(\mathbf{x}_i) \cdot P(\bar{\mathbf{x}}_i)$. Each product $P(\mathbf{x}_i) \cdot P(\bar{\mathbf{x}}_i)$ is a rational map in $x, y^2$, and $y_i^2$ $(i\le(\ell^2-1)/2)$ by Lemma \ref{lem:poly-in-y^2}. We replace $y^2$ (respectively $y_i^2$) with $W_1(x)$ (respectively $W_1(x_i)$) to obtain rational maps in the variable $x$ only.

\begin{example}[\textbf{Computing the polynomial $P(x)\big|X_1$} via Algorithm~\ref{alg:evaluate}]\label{ex:evaluation}

  Let $\ell=3$, $p = 179$, and $E_{1728}: y^2 = x^3 - x$ the supersingular elliptic curve over $\overline{\mathbb{F}}_p$ with $j=1728$.  Let $X_1(x)$, $Y_1(x)$ be associated to multiplication-by-$3$, i.e.
\[ [3]_{E_{1728}}(x,y) = ( X_1(x), Y_1(x)y ) \quad \mbox{where} \quad
X_1(x)=\frac{20 x^9 + 61 x^7 + 63 x^5 + 175 x^3 + x}{x^8 + 175 x^6 + 63 x^4 + 61 x^2 + 20} \ . \]
Let $P(x)= x^{18} + 122x^{16} + 136 x^{14} + 65 x^{12} + 29 x^{10} + 150 x^8 + 114 x^6 + 43 x^4 + 57 x^2 + 178$. We compute $p(x) = P(x)\big| X_1$ using Algorithm \ref{alg:evaluate} as follows. 

 In Steps \ref{step:alg-eval-root} and \ref{step:alg-eval-torsion}, we  may choose $\zeta = 0$.  Let $\FF_{p^4}$ be generated by $\Fpa$ having minimal polynomial $x^4 + x^2 + 109 x + 2$.  We obtain $S_0= O_{E_{1728}}, S_1= (103, y_1), S_2=(76, y_2), S_3=(24 \Fpa^3 + 39 \Fpa^2 + 119 \Fpa + 102, y_3), S_4=( 155 \Fpa^3 + 140 \Fpa^2 + 60\Fpa + 77, y_4), S_5 = -S_1, S_6 = -S_2, S_7= -S_3, S_8=-S_4$. In   Steps \ref{step:alg-eval-alpha}, we compute $\mathbf{x}_i(x,y, y_i)$ and    $\bar{\mathbf{x}}_i(x,y, y_i)$  as  $\mathbf{x}_0 = x, \bar{\mathbf{x}}_i (x, y, y_i) =\mathbf{x}_i(x, y, -y_i), \forall i,  1 \le i \le 4$ where    %
    \begin{align*}
       \mathbf{x}_1(x, y, y_1) &=\frac{- x^3 +  y^2 - 2  y   y_1 +   y_1^2 - 76  x^2 + 48  x + 68}{x^2 - 27  x + 48}, \\
    \mathbf{x}_2(x, y, y_2) &= (- x^3 +  y^2 - 2  y   y_2 +   y_2^2 + 76  x^2 + 48  x - 68)/( x^2 + 27  x + 48), \\
    \mathbf{x}_3(x, y, y_3) &=\frac{- x^3 +  y^2 - 2  y   y_3 +   y_3^2 + (24  \Fpa^3 + 39  \Fpa^2 - 60  \Fpa - 77)  x^2 - 46  x + (30  \Fpa^3 + 4  \Fpa^2 - 75  \Fpa + 38)}{( x^2 + (-48  \Fpa^3 - 78  \Fpa^2 - 59  \Fpa - 25)  x - 46)},\\
    \mathbf{x}_4(x, y, y_4) &= \frac{- x^3 +  y^2 - 2  y   y_4 +   y_4^2 + (-24  \Fpa^3 - 39  \Fpa^2 + 60  \Fpa + 77)  x^2 - 46  x + (-30  \Fpa^3 - 4  \Fpa^2 + 75  \Fpa - 38}{ x^2 + (48  \Fpa^3 + 78  \Fpa^2 + 59  \Fpa + 25)  x - 46}. 
    \end{align*}

     In Steps \ref{step:alg-eval-set1A}--\ref{step:alg-eval-set2}: We compute the norm $N_P(x)$ of $P(x)$ by first computing $P(\mathbf{x}_i) \cdot P(\bar{\mathbf{x}}_i)=\frac{N_i}{D_i}, 1 \le i \le 4$. We then have $N(x) = P(x) \prod_i N_i= 14  x^{162} + 157  x^{160}  + \cdots  + 22  x^2 + 165$ and  $D(x) =\prod_i D_i =x^{144}    + 107  x^{142}     + \cdots    + 90  x^{2 }   + 75$. Hence   $N_P(x)= \frac{N(x)}{D(x)}$. Finally,  we  compute all the coefficients of $p(x)$ by repeating Steps \ref{step:alg-eval-ai}--\ref{step:alg-eval-divideX1}. The result is
    \[ p(x)= x^{18} + 170 x^{16} + 36 x^{14} + 95 x^{12} + 126 x^{10} + 53 x^8 + 84 x^6 + 143 x^4 + 9x^2 + 178.\]
\end{example}

\begin{example}[{\textbf{Division by $\ell=3$} via Algorithm~\ref{alg:divisionbyell}}]\label{ex:divisionby3}

As before, let $p = 179$ and $E_{1728}: y^2 = x^3 - x$ the supersingular elliptic curve over $\overline{\mathbb{F}}_p$ of $j$-invariant $j(E_{1728}) = 1728$ as in Example \ref{ex:evaluation}. Then the endomorphism ring of $E_{1728}$ contains the endomorphism $[i]$ defined as 
 $[i](x,y):= (-x,\Fpi y)$ with $\Fpi \in \Fpp$ and $\Fpi^2=-1$.

The map $\theta=1+[i]$ is  a separable endomorphism and we have $([3] \theta)(x, y)=\left(\frac{F_1(x)}{F_2(x)}, \frac{G_1(x)}{G_2(x)}y\right)$, defined over $\Fpp$, with 
\begin{align*}
 F_1(x)  &= 169 \Fpi x^{18} + 33 \Fpi x^{16} + 72 \Fpi x^{14} + 66 \Fpi x^{12} + 68 \Fpi x^{10} + 111 \Fpi x^{8 }+ 113 \Fpi x^{6} + 107 \Fpi x^{4} + 146 \Fpi x^{2} + 10 \Fpi, \\
F_2(x) &=x^{17} + 8 x^{15} + 45 x^{13} + 124 x^{11} + 110 x^{9} + 124 x^{7} + 45 x^{5} + 8 x^{3} + x, \\
G_1(x) &=(58 \Fpi + 58) x^{26} + (170 \Fpi + 170) x^{24} + \cdots +  (170 \Fpi + 170) x^{2} + 58\Fpi + 58 , \\
G_2(x) &=x^{26} + 12 x^{24} + 2 x^{22} + 66 x^{20} + 128 x^{18} + 44 x^{16} + 171 x^{14} + 44 x^{12} + 128 x^{10} + 66 x^{8} + 2 x^{6} + 12 x^{4} + x^2 . 
\end{align*}
We apply Algorithm \ref{alg:divisionbyell} to divide $[3] \theta$ by $3$ to obtain $\theta=[f(x), g(x)y]$ as follows.

In Step \ref{step:algdivisionbyellPQ}, we write $F(x)=\frac{c_F\cdot P(x)}{(\psi_{E_{1728},3}(x))^2 \cdot Q(x)} $ where $c_F=169 \Fpi$, $\psi_{E_{1728},3}(x)=x^4 + 177 x^2 + 119$ and    %
    \begin{align*}
     P(x) = x^{18} + 122x^{16} + \cdots  57 x^2 + 178, \qquad 
     Q(x) = x^9 + 12 x^7 + 30 x^5 + 143 x^3 + 9 x.
     \end{align*}

In Step \ref{step:algdivisionbyell-X1Y2}, we compute $X_1$ and $Y_2$ using the formula for multiplication by 3 map on $E_{1728}$. Here, $X_1$ is as given in Example \ref{ex:evaluation} and
\[ Y_2= \frac{126 x^{12} + 92 x^{10} + 153 x^8 + 136 x^6 + 139 x^4 + 63 x^2 + 159}{x^{12} + 173 x^{10} + 11 x^8 + 175 x^6 + 56 x^4 + 59 x^2 + 53}. \]
Then  we compute $p(x) = P(x)\big|X_1$ and $q(x) = Q(x)\big|X_1$ in  Steps \ref{step:algdivisionbyellevaP} and \ref{step:algdivisionbyellevaQ} using Algorithm \ref{alg:evaluate}     to obtain 
    $p(x)= x^{18} + 170 x^{16} + \cdots + 9x^2 + 178, $ and $q(x)= x^9$. 
In Step \ref{step:algdivisionbyellell2root}, computing $9$-th roots of $p(x)$ and $q(x)$ yields 
    $p_0(x)= x^2 + 178$ and $q_0(x)= x$.  The final output is \[ f(x)=c_F\ell^2\cdot \frac{p_0(x)}{q_0(x)}=\frac{89 \Fpi x^2 + 90 \Fpi}{x} \, , \quad g(x)=\frac{G(x)}{Y_2(f(x))}=\frac{(134\Fpi + 134) x^2 + 134\Fpi + 134}{x^2}. \]
\end{example}

To determine the complexity of Algorithm \ref{alg:evaluate}, we first prove the following lemma which is needed in the proof of Proposition \ref{prop:comple-alg-evaluate}.

\begin{lemma} \label{lem:poly-in-y^2}
         Fix $0 \le i \le \frac{\ell^2-1}{2}$, the products $\mathbf{x}_i \bar{\mathbf{x}}_i$ and $P(\mathbf{x}_i) P(\bar{\mathbf{x}}_i)$  are rational functions in $x,y^2$, and $y_i^2$.
\end{lemma}

 \begin{proof}
   By direct computation, both $\mathbf{x}_i+\bar{\mathbf{x}}_i$ and $\mathbf{x}_i \bar{\mathbf{x}}_i$ are rational functions in $x,y^2$, and $y_i^2$.  As a symmetric polynomial in $\mathbf{x}_i$ and $\overline{\mathbf{x}}_i$, the quantity $P(\mathbf{x}_i)P(\overline{\mathbf{x}}_i)$ is a polynomial in $\mathbf{x}_i+\bar{\mathbf{x}}_i$ and $\mathbf{x}_i \bar{\mathbf{x}}_i$, hence also a rational function in $x, y^2$ and $y_i^2$.
\end{proof}

\begin{proposition}\label{prop:comple-alg-evaluate}
Algorithm \ref{alg:evaluate} is correct and has runtime $O(\deg^2(P)\mulM(p^m))$.
\end{proposition}
\begin{proof}
Algorithm \ref{alg:evaluate} is correct by \cite[Pages 8--9]{mcmurdy2014explicit} and Lemma \ref{lem:poly-in-y^2}. Steps~\ref{step:alg-eval-root}-\ref{step:alg-eval-alpha} are negligible because they require a fixed number of operations in an extension of $\mathbb{F}_{p^2}$ of degree $O(\ell^2)$. Since $P(x)\in \FF_{p^m}[x]$ and $E_1[\ell]$ is defined over an extension of $\FF_{p^2}$ of degree at most $\ell^2$ by Lemma \ref{lem:torsion-basis}, all the arithmetic in the remaining steps takes place in a field extension of $\FF_{p^2}$ of degree lcm$(\ell^2,m) = O(m)$. 

In the first loop (Steps~\ref{step:alg-eval-norm}-\ref{step:alg-eval-product}), the most costly steps are~\ref{step:alg-eval-computepair} and~\ref{step:alg-eval-product} which both require $O(\deg^2(P))$ operations; the remaining steps are linear in $\deg P$ when Horner's algorithm is used. In the second loop (Steps~\ref{step:alg-eval-coefficients}-\ref{step:alg-eval-set2}), $p(x)$ is computed as described in \cite[Page 9]{mcmurdy2014explicit}. Step~\ref{step:alg-eval-ai} requires $O(\deg P)$ field operations using Horner's algorithm again. Since $X_1$ has degree $O(\ell^2)$,  Step~\ref{step:alg-eval-set2} also takes $O(\deg P)$ operations. Hence the second loop takes $O(\deg^2(P))$ field operations. 
 \end{proof}

\begin{proposition}\label{prop:comple-alg-divisionbyell}
Algorithm \ref{alg:divisionbyell} is correct and has runtime $O(\deg^2(\varphi)\mulM(p))$.
\end{proposition}

\begin{proof}
The correctness of Algorithm~\ref{alg:divisionbyell} follows from \cite[Proposition 2.6]{mcmurdy2014explicit}. By Lemma \ref{lemma:isog12}, $\varphi$ is defined over $\FF_{p^{12}}$,  
so all the rational functions appearing in the algorithm belong to $\FF_{p^{12}}(x)$. We also note that $P(x)$ and $Q(x)$ have degree $O(\deg \varphi)$, hence so do $p(x)$, $q(x)$, $p_0(x)$ and $q_0(x)$. 

Since $\psi_{E_1,\ell}(x)$ and $W_1(x)$ have fixed degree, Step~\ref{step:algdivisionbyellPQ} requires $O(\deg \varphi)$ field operations. Steps~\ref{step:algdivisionbyellell2root} and~\ref{step:algdivisionbyellfg} take $\widetilde{O}(\deg \varphi)$ operations using fast polynomial arithmetic; see~\cite[Theorem 1.2]{harvey2019polynomial}. Here, to extract an $\ell^2$-th root of $p(x)$, we apply a truncated variant of Newton's method (see \cite[Sections 9.4 and 9.6]{vzGathenGerhard}) to the polynomial $H(y) = y^{\ell^2} - p(x)$ and compute the sequence of polynomials 
\[ f_0(x) = x^{\deg p} \ , \qquad f_{i+1}(x) = f_i(x) - \left \lfloor \frac{H(f_i(x))}{H'(f_i(x)} \right \rfloor  \quad (i \ge 0)  \]
to obtain $p_0(x)$ after at most $\lceil \log_2(\deg p) \rceil$ iterations; similarly for $q_0(x)$.

The runtime of Algorithm~\ref{alg:divisionbyell} is thus dominated by Steps~\ref{step:algdivisionbyellevaP} and~\ref{step:algdivisionbyellevaQ}, which have runtime $O(\deg^2(\varphi) \mulM(p^{12})) = O(\deg^2(\varphi) \mulM(p))$.
\end{proof}

\bibliographystyle{abbrv}  
\bibliography{refs}

\begin{thebibliography}{10}

\bibitem{Apostol}
T.~M. Apostol.
\newblock {\em Introduction to analytic number theory}.
\newblock Undergraduate Texts in Mathematics. Springer-Verlag, New
  York-Heidelberg, 1976.

\bibitem{papertwo}
S.~Arpin, M.~Chen, K.~E. Lauter, R.~Scheidler, K.~E. Stange, and H.~T.~N. Tran.
\newblock Orientations and cycles in supersingular isogeny graphs, 2022.
\newblock In preparation.

\bibitem{github}
S.~Arpin, M.~Chen, K.~E. Lauter, R.~Scheidler, K.~E. Stange, and H.~T.~N. Tran.
\newblock Win5 github repository, 2022.
\newblock \url{https://github.com/SarahArpin/WIN5}.

\bibitem{bank2019cycles}
E.~Bank, C.~Camacho-Navarro, K.~Eisentr{\"a}ger, T.~Morrison, and J.~Park.
\newblock Cycles in the supersingular l-isogeny graph and corresponding
  endomorphisms.
\newblock In {\em Research Directions in Number Theory}, pages 41--66.
  Springer, 2019.

\bibitem{note_on_CSIDH_BIJ}
J.-F. Biasse, A.~Iezzi, and M.~J. Jacobson.
\newblock A note on the security of csidh.
\newblock In D.~Chakraborty and T.~Iwata, editors, {\em Progress in Cryptology
  -- INDOCRYPT 2018}, pages 153--168, Cham, 2018. Springer International
  Publishing.

\bibitem{BrokerCharlesLauter_EvalLargeDeg}
R.~Br\"{o}ker, D.~Charles, and K.~Lauter.
\newblock Evaluating large degree isogenies and applications to pairing based
  cryptography.
\newblock In {\em Pairing-{B}ased {C}ryptography---{P}airing 2008}, volume 5209
  of {\em Lecture Notes in Comput. Sci.}, pages 100--112. Springer, Berlin,
  2008.

\bibitem{castryck-sidh-attack}
W.~Castryck and T.~Decru.
\newblock An efficient key recovery attack on sidh (preliminary version).
\newblock {\em Cryptology ePrint Archive}, 2022.

\bibitem{CSIDH}
W.~Castryck, T.~Lange, C.~Martindale, L.~Panny, and J.~Renes.
\newblock C{SIDH}: an efficient post-quantum commutative group action.
\newblock In {\em Advances in {C}ryptology---{ASIACRYPT} 2018. {P}art {III}},
  volume 11274 of {\em Lecture Notes in Comput. Sci.}, pages 395--427.
  Springer, Cham, 2018.

\bibitem{CPV}
W.~Castryck, L.~Panny, and F.~Vercauteren.
\newblock Rational isogenies from irrational endomorphisms.
\newblock In {\em Advances in cryptology---{EUROCRYPT} 2020. {P}art {II}},
  volume 12106 of {\em Lecture Notes in Comput. Sci.}, pages 523--548.
  Springer, Cham, [2020] \copyright 2020.

\bibitem{CharlesGorenLauter}
D.~X. Charles, E.~Z. Goren, and K.~E. Lauter.
\newblock Cryptographic hash functions from expander graphs.
\newblock {\em J. Cryptology}, 22(1):93--113, 2009.
\newblock \url{https://eprint.iacr.org/2006/021}.

\bibitem{chenu2021higherdegree}
M.~Chenu and B.~Smith.
\newblock {Higher-degree supersingular group actions}.
\newblock {\em Math. Cryptology}, 1(1):1--15, 2021.

\bibitem{Decompose_finite_abgps_CM}
K.~K.~H. Cheung and M.~Mosca.
\newblock Decomposing finite abelian groups.
\newblock {\em Quantum Info. Comput.}, 1(3):26–32, Oct 2001.

\bibitem{CJS}
A.~Childs, D.~Jao, and V.~Soukharev.
\newblock Constructing elliptic curve isogenies in quantum subexponential time.
\newblock {\em J. Math. Cryptol.}, 8(1):1--29, 2014.

\bibitem{Cohen_CourseInCompAlgNT}
H.~Cohen.
\newblock {\em A course in computational algebraic number theory}, volume 138
  of {\em Graduate Texts in Mathematics}.
\newblock Springer-Verlag, Berlin, 1993.

\bibitem{colo2019orienting}
L.~Col\`o and D.~Kohel.
\newblock Orienting supersingular isogeny graphs.
\newblock {\em J. Math. Cryptol.}, 14(1):414--437, 2020.

\bibitem{smoothbound_CN}
J.~S. Coron and D.~Naccache.
\newblock Security analysis of the {G}ennaro-{H}alevi-{R}abin signature scheme.
\newblock In B.~Preneel, editor, {\em Advances in {C}ryptology --- {EUROCRYPT}
  2000}, pages 91--101, Berlin, Heidelberg, 2000. Springer Berlin Heidelberg.

\bibitem{Hard_Homo_Space_Couveignes}
J.-M. Couveignes.
\newblock Hard homogeneous spaces.
\newblock Cryptology ePrint Archive, Report 2006/291, 2006.
\newblock \url{https://ia.cr/2006/291}.

\bibitem{Cox_primesoftheform}
D.~A. Cox.
\newblock {\em Primes of the form {$x^2 + ny^2$}}.
\newblock Pure and Applied Mathematics (Hoboken). John Wiley \& Sons, Inc.,
  Hoboken, NJ, second edition, 2013.
\newblock Fermat, class field theory, and complex multiplication.

\bibitem{CGPT}
E.~Croot, A.~Granville, R.~Pemantle, and P.~Tetali.
\newblock On sharp transitions in making squares.
\newblock {\em Annals of Mathematics}, 175(3):1507--1550, 2012.

\bibitem{DD}
P.~Dartois and L.~De~Feo.
\newblock On the security of osidh.
\newblock In {\em Public-Key Cryptography – PKC 2022: 25th IACR International
  Conference on Practice and Theory of Public-Key Cryptography, Virtual Event,
  March 8–11, 2022, Proceedings, Part I}, volume 13177 of {\em Lecture Notes
  in Comput. Sci.}, pages 52--81. Springer, Cham, [2022] \copyright 2022.

\bibitem{MathOfIsog}
L.~{De Feo}.
\newblock Mathematics of isogeny based cryptography.
\newblock 2017.
\newblock \url{https://arxiv.org/abs/1711.04062}.

\bibitem{SETA}
L.~De~Feo, C.~Delpech~de Saint~Guilhem, T.~B. Fouotsa, P.~Kutas, A.~Leroux,
  C.~Petit, J.~Silva, and B.~Wesolowski.
\newblock {\em Séta: Supersingular Encryption from Torsion Attacks}, pages
  249--278.
\newblock Advances in Cryptology – ASIACRYPT 2021. Springer International
  Publishing, Cham, 2021.

\bibitem{deFeoKiefferSmith_ordinarykeyexch}
L.~De~Feo, J.~Kieffer, and B.~Smith.
\newblock Towards practical key exchange from ordinary isogeny graphs.
\newblock In {\em Advances in cryptology---{ASIACRYPT} 2018. {P}art {III}},
  volume 11274 of {\em Lecture Notes in Comput. Sci.}, pages 365--394.
  Springer, Cham, 2018.

\bibitem{deQuehenEtAl_ImprovedTorPt}
V.~de~Quehen, P.~Kutas, C.~Leonardi, C.~Martindale, L.~Panny, C.~Petit, and
  K.~E. Stange.
\newblock {\em Improved Torsion-Point Attacks on SIDH Variants}, pages
  432--470.
\newblock Advances in Cryptology – CRYPTO 2021. Springer International
  Publishing, Cham, 2021.

\bibitem{EHLMP_reductions}
K.~Eisentr\"{a}ger, S.~Hallgren, K.~Lauter, T.~Morrison, and C.~Petit.
\newblock Supersingular isogeny graphs and endomorphism rings: reductions and
  solutions.
\newblock In {\em Advances in cryptology---{EUROCRYPT} 2018. {P}art {III}},
  volume 10822 of {\em Lecture Notes in Comput. Sci.}, pages 329--368.
  Springer, Cham, 2018.

\bibitem{ElkiesThesis}
N.~D. Elkies.
\newblock The existence of infinitely many supersingular primes for every
  elliptic curve over {${\bf Q}$}.
\newblock {\em Invent. Math.}, 89(3):561--567, 1987.

\bibitem{FiteSutherlandSatoTateGroups}
F.~Fit\'{e} and A.~V. Sutherland.
\newblock Sato-{T}ate groups of {$y^2=x^8+c$} and {$y^2=x^7-cx$}.
\newblock In {\em Frobenius distributions: {L}ang-{T}rotter and {S}ato-{T}ate
  conjectures}, volume 663 of {\em Contemp. Math.}, pages 103--126. Amer. Math.
  Soc., Providence, RI, 2016.

\bibitem{GalbraithPetitSilva_IdProtocols}
S.~D. Galbraith, C.~Petit, and J.~Silva.
\newblock Identification protocols and signature schemes based on supersingular
  isogeny problems.
\newblock {\em J. Cryptology}, 33(1):130--175, 2020.

\bibitem{harvey2019polynomial}
D.~Harvey and J.~van Der~Hoeven.
\newblock Polynomial multiplication over finite fields in time {$O (n \log
  n)$}.
\newblock 2019.
\newblock \url{https://hal.archives-ouvertes.fr/hal-02070816/document}.

\bibitem{IonicaJoux}
S.~Ionica and A.~Joux.
\newblock Pairing the volcano.
\newblock In {\em Algorithmic number theory}, volume 6197 of {\em Lecture Notes
  in Comput. Sci.}, pages 201--208. Springer, Berlin, 2010.

\bibitem{Kaneko}
M.~Kaneko.
\newblock Supersingular {$j$}-invariants as singular moduli {${\rm mod}\, p$}.
\newblock {\em Osaka J. Math.}, 26(4):849--855, 1989.

\bibitem{kieffer2018accelerating}
J.~Kieffer.
\newblock Accelerating the couveignes rostovtsev stolbunov key exchange
  protocol.
\newblock Master's thesis, l’Université Paris IV, 2018.
\newblock \url{https://arxiv.org/pdf/1804.10128.pdf}.

\bibitem{knuth81}
D.~E. Knuth.
\newblock {\em The art of computer programming. {V}ol. 2}.
\newblock Addison-Wesley Series in Computer Science and Information Processing.
  Addison-Wesley Publishing Co., Reading, Mass., second edition, 1981.
\newblock Seminumerical algorithms.

\bibitem{KLPT}
D.~Kohel, K.~Lauter, C.~Petit, and J.-P. Tignol.
\newblock On the quaternion $\ell$-isogeny path problem.
\newblock {\em LMS J.Comput. Math.}, 2014.

\bibitem{Kohelthesis}
D.~R. Kohel.
\newblock {\em Endomorphism rings of elliptic curves over finite fields}.
\newblock PhD thesis, University of California at Berkeley, 1996.

\bibitem{hidden_shift_KMPW}
P.~Kutas, S.-P. Merz, C.~Petit, and C.~Weitkaemper.
\newblock One-way functions and malleability oracles: hidden shift attacks on
  isogeny-based protocols.
\newblock In A.~Canteaut and F.~Standaert, editors, {\em Advances in Cryptology
  – EUROCRYPT 2021}, Lecture Notes in Computer Science, pages 242--271.
  Springer, June.

\bibitem{BonehLove_ARXIV}
J.~Love and D.~Boneh.
\newblock Supersingular curves with small non-integer endomorphisms, 2020.
\newblock \url{https://arxiv.org/abs/1910.03180}.

\bibitem{BonehLove}
J.~Love and D.~Boneh.
\newblock Supersingular curves with small noninteger endomorphisms.
\newblock In {\em A{NTS} {XIV}---{P}roceedings of the {F}ourteenth
  {A}lgorithmic {N}umber {T}heory {S}ymposium}, volume~4 of {\em Open Book
  Ser.}, pages 7--22. Math. Sci. Publ., Berkeley, CA, 2020.

\bibitem{maino-attack-sidh}
L.~Maino and C.~Martindale.
\newblock An attack on sidh with arbitrary starting curve.
\newblock {\em Cryptology ePrint Archive}, 2022.

\bibitem{McMurdyCode}
K.~McMurdy.
\newblock \url{https://phobos.ramapo.edu/~kmcmurdy/research/SAGE_ssEndos/}.
\newblock Accessed Jan 10, 2022.

\bibitem{mcmurdy2014explicit}
K.~McMurdy.
\newblock Explicit representation of the endomorphism rings of supersingular
  elliptic curves.
\newblock
  \url{https://phobos.ramapo.edu/~kmcmurdy/research/McMurdy-ssEndoRings.pdf},
  2014.

\bibitem{onuki2021}
H.~Onuki.
\newblock On oriented supersingular elliptic curves.
\newblock {\em Finite Fields App.}, 69, 2021.

\bibitem{Pizer_RamGraphsHecke}
A.~K. Pizer.
\newblock Ramanujan graphs and {H}ecke operators.
\newblock {\em Bull. Amer. Math. Soc. (N.S.)}, 23(1):127--137, 1990.

\bibitem{Robin}
G.~Robin.
\newblock Grandes valeurs de la fonction somme des diviseurs et hypoth\`ese de
  {R}iemann.
\newblock {\em J. Math. Pures Appl. (9)}, 63(2):187--213, 1984.

\bibitem{Sardari}
N.~T. Sardari.
\newblock Diameter of {R}amanujan graphs and random {C}ayley graphs.
\newblock {\em Combinatorica}, 39(2):427--446, 2019.

\bibitem{Schoof_PrimalityTesting}
R.~Schoof.
\newblock Four primality testing algorithms.
\newblock In {\em Algorithmic number theory: lattices, number fields, curves
  and cryptography}, volume~44 of {\em Math. Sci. Res. Inst. Publ.}, pages
  101--126. Cambridge Univ. Press, Cambridge, 2008.

\bibitem{shumow2009isogenies}
D.~Shumow.
\newblock Isogenies of elliptic curves: a computational approach.
\newblock Master's thesis, University of Washington, 2009.
\newblock \url{https://arxiv.org/abs/0910.5370}.

\bibitem{silverman2009arithmetic}
J.~H. Silverman.
\newblock {\em The arithmetic of elliptic curves}, volume 106 of {\em Graduate
  Texts in Mathematics}.
\newblock Springer, Dordrecht, second edition, 2009.

\bibitem{kate1728note}
K.~E. Stange.
\newblock Frobenius and the endomorphism ring of {$j = 1728$}, 2021.
\newblock \url{http://math.colorado.edu/~kstange/papers/1728.pdf}.

\bibitem{MR3356295}
G.~Tenenbaum.
\newblock On ultrafriable integers.
\newblock {\em Q. J. Math.}, 66(1):333--351, 2015.

\bibitem{sagemath}
{The Sage Developers}.
\newblock {\em {S}ageMath, the {S}age {M}athematics {S}oftware {S}ystem
  ({V}ersion 9.4)}, 2022.
\newblock {\tt https://www.sagemath.org}.

\bibitem{Velu}
J.~V\'{e}lu.
\newblock Isog\'{e}nies entre courbes elliptiques.
\newblock {\em C. R. Acad. Sci. Paris S\'{e}r. A-B}, 273:A238--A241, 1971.

\bibitem{voight}
J.~Voight.
\newblock {\em Quaternion algebras}, volume 288 of {\em Graduate Texts in
  Mathematics}.
\newblock Springer, Cham, [2021] \copyright 2021.

\bibitem{vzGathenGerhard}
J.~von~zur Gathen and J.~Gerhard.
\newblock {\em Modern computer algebra}.
\newblock Cambridge University Press, Cambridge, third edition, 2013.

\bibitem{vzGathenShoup}
J.~von~zur Gathen and V.~Shoup.
\newblock Computing {F}robenius maps and factoring polynomials.
\newblock {\em Comput. Complexity}, 2(3):187--224, 1992.

\bibitem{waterhouse1969}
W.~C. Waterhouse.
\newblock Abelian varieties over finite fields.
\newblock In {\em Annales scientifiques de l'{\'E}cole normale sup{\'e}rieure},
  volume~2, pages 521--560, 1969.

\bibitem{WesolowskiOrientations}
B.~Wesolowski.
\newblock Orientations and the supersingular endomorphism ring problem.
\newblock In {\em Advances in cryptology---{EUROCRYPT}}, Lecture Notes in
  Comput. Sci.(to appear). 2022.
\newblock \url{https://ia.cr/2021/1583}.

\bibitem{Wesolowski_IsogPathandEndoRing}
B.~Wesolowski.
\newblock The supersingular isogeny path and endomorphism ring problems are
  equivalent.
\newblock {\em FOCS 2021-62nd Annual IEEE Symposium on Foundations of Computer
  Science}, 2022.

\bibitem{Cohen_Lenstra_Heuristics}
H.~C. Williams and H.~te~Riele.
\newblock {New computations concerning the Cohen-Lenstra Heuristics}.
\newblock {\em Experimental Mathematics}, 12(1):99 -- 113, 2003.

\bibitem{MinCMLiftings}
T.~Yang.
\newblock Minimal {CM} liftings of supersingular elliptic curves.
\newblock {\em Pure and applied mathematics quarterly}, 4(4):1317--1326, 2008.

\end{thebibliography}

\end{document}